\documentclass[12pt]{extarticle}
\usepackage{amssymb,amsfonts,amsthm,amsmath,bbold}
\usepackage{mathrsfs,pifont}
\usepackage{mathabx}
\usepackage[all]{xy}
\usepackage{array}
\usepackage{authblk}
 \usepackage{mathpple}
\usepackage{amssymb}
\usepackage{bm}
\usepackage{ytableau}
\usepackage{amsmath}
\usepackage{amsmath,amsfonts, amscd,amsthm, graphicx}
\usepackage{comment}
\usepackage[OT1]{fontenc} 
\usepackage[all]{xy}
\usepackage{color}
\usepackage{tikz-cd}
\usepackage{bm}
\usepackage{hyperref}

\usepackage[top=1.2in, bottom=1.2in, left=0.5in,
right=0.5in]{geometry}

\usepackage[font={small,sl}]{caption}
\usepackage{tikz}
\usepackage{tikz-cd}
\usetikzlibrary{matrix,arrows}

\usepackage{hyperref}
\usepackage[backrefs,msc-links]{amsrefs}

\newcommand{\blt}{\bullet}

\newcounter{Chapcounter}

\newcommand{\chapter}[1] 
{ {\centering          
  \addtocounter{Chapcounter}{1} \Large \underline{\textbf{ \color{blue} Chapter \theChapcounter: ~#1}} }   
  \addcontentsline{toc}{section}{ \color{blue} Chapter:~\theChapcounter~~ #1}    
}

\theoremstyle{definition}

\theoremstyle{plain}
\newtheorem{thm}{Theorem}[section]

\newtheorem{thm-defn}{Theorem/Definition}[section]
\newtheorem{lem}[thm]{Lemma}
\newtheorem{lem-defn}[thm]{Lemma/Definition}
\newtheorem{prop}[thm]{Proposition}
\newtheorem{cor}[thm]{Corollary}
\newtheorem{conjecture}[thm]{Conjecture}
\newtheorem{prop-defn}[thm]{Proposition-Definition}

\newtheorem{defn}[thm]{Definition}%[section]

\newtheorem{thm-alg}[thm]{Theorem/Algorithm}

\newcommand{\mbf}{\mathbf}
\newcommand{\mbb}{\mathbb}
\newcommand{\mf}{\mathfrak}
\newcommand{\mc}{\mathcal}

\title{Maulik-Okounkov quantum loop groups and Drinfeld double of preprojective $K$-theoretic Hall algebras}
\author{Tianqing Zhu}
\affil{Department of Mathematics, Columbia University \\
       tz2611@columbia.edu}

\begin{document}

\maketitle
\begin{abstract}
In this paper we prove the following results: Given the Drinfeld double $\mathcal{A}^{ext}_{Q}$ of the localised preprojective $K$-theoretic Hall algebra $\mathcal{A}^{+}_{Q}$ of quiver type $Q$ with the Cartan elements, there is a $\mathbb{Q}(q,t_e)_{e\in E}$-Hopf algebra isomorphism between $\mc{A}^{ext}_{Q}$ and the localised Maulik-Okounkov quantum loop group $U^{MO}_{q}(\hat{\mathfrak{g}}_{Q})$ of quiver type $Q$. Moreover, we prove the isomorphism of $\mathbb{Z}[q^{\pm1},t_{e}^{\pm1}]_{e\in E}$-algebras between the positive/negative half of the integral Maulik-Okounkov quantum loop group $U_{q}^{MO,\pm,\mathbb{Z}}(\hat{\mathfrak{g}}_{Q})$ with the (opposite) algebra of the integral preprojective (nilpotent) $K$-theoretic Hall algebra $\mathcal{A}^{+,\mathbb{Z}}_{Q}$ ($(\mathcal{A}^{+,nilp,\mathbb{Z}}_{Q})^{op}$) of the same quiver type $Q$. As the application, we prove that one can identify the wall subalgebra $U_{q}^{MO,\mbb{Z}}(\mf{g}_{w})$ as the root subalgebra $\mc{B}_{\mbf{m},w}^{\mbb{Z}}$ in the slope subalgebra $\mc{B}_{\mbf{m}}^{\mbb{Z}}$ as the quasitriangular Hopf $\mbb{Z}[q^{\pm1},t_e^{\pm1}]_{e\in E}$-algebras. Moreover we use the freeness of the wall subalgebra in MO quantum loop groups to prove the freeness of the preprojective $K$-theoretic Hall algebra for arbitrary torus $\mbb{C}_q^*\subset A\subset T$.

\end{abstract}

\tableofcontents
\section{\textbf{Introduction}}
\subsection{Quantum groups from KHA and stable envelopes}
\subsubsection{} 
The stable envelope is a powerful tool in the study of both geometric representation theory and  enumerative geometry of symplectic resolutions. It was initially constructed by Maulik and Okounkov \cite{MO19} in the equivariant cohomology setting, and then it was later introduced in \cite{OS22}\cite{O15}\cite{AO21}\cite{O21} in the $K$-theory and elliptic cohomology settings.

In the case of Nakajima quiver varieties, one important application of the stable envelope is constructing the geometric $R$-matrix, and then use the FRT formalism to construct the quantum groups. In the cohomology case, the corresponding quantum group is called the Maulik-Okounkov Yangian $Y_{\hbar}^{MO}(\mf{g}_{Q})$, or the MO Yangian. In the $K$-theory case, the corresponding one is called the Maulik-Okounkov quantum loop group $U_{q}^{MO}(\hat{\mf{g}}_{Q})$,or we can call it the MO quantum loop group.

In more detail, stable envelope is a well-defined class that connects the enumerative geometry to the geometric representation theory. Many enumerative invariants, for example, as vertex functions in quasimap counting, or small $J$-functions, can be packed into some difference/differential equations  that appears in the representation theory of quantum groups \cite{MO19}\cite{O15}\cite{OS22}. It has been studied in detail for some examples in \cite{AO17}\cite{AO21}\cite{Dn22}\cite{Dn22-2}\cite{DJ24}\cite{JS25}\cite{KPSZ21}\cite{PSZ20}\cite{S16}\cite{Z24-3}\cite{Z24-4}. Moreover, stable envelope is also a key concept for constructing the 3d mirror symmetry in the context of the enumerative geometry, which has also been studied in detail in \cite{BD23}\cite{BR23}\cite{BD24-2}\cite{KS22}\cite{KS23}\cite{RSZ22}]\cite{RSVZ19}\cite{SZ22}.

\subsubsection{} Fix a quiver $Q=(I,E)$ with vertices $I$ and arrows $E$, in this paper we allow the quiver $Q$ to have multiple edges and multiple loops. Quantum group is the Hopf algebra associated with a quiver type $Q$. The first emergence of quantum groups can be traced back to the 80s in the quantum integrable model theory \cite{FRT16}. It is formulated \cite{Dr86}\cite{J85} as the Hopf algebra deformation of the universal enveloping algebra $U(\mf{g})$ of a Lie algebra $\mf{g}$.

Generally there are two ways to realise the quantum group. The first one is given by the FRT formalism \cite{FRT16}, and it means the quantum group is viewed as the algebra generated by the matrix coefficients of the $R$-matrix, which is a solution for the Yang-Baxter equation. The second one is the Drinfeld realisation \cite{Dr87}, which means that we think of the quantum groups generated by the positive, Cartan and negative half with generators written in a generating function. This formulation is often used for Yangian algebras and quantum affine algebras. 

It is a natural question to ask if we fix the quiver type $Q$ of the Yangian or the quantum affine algebras, whether the algebra generated by FRT formalism is isomorphic to the algebra given by the Drinfeld realisation. For the case of the finite ADE type and some other non simply-laced finite type, these has been proved in many references \cite{Dr87}\cite{DF93}\cite{JLM18}\cite{JLM20}\cite{JLM20-1}\cite{LP20}\cite{LP22}. In general, such an isomorphism for general type quiver Q is still unsolved. It is also a very important problem in both representation theory of quantum groups and quantum integrable systems.

\subsubsection{}
In the geometric representation theory, both Drinfeld realisation and FRT formalism can be realised as the cohomology/K-theory of moduli objects. We still fix the quiver $Q=(I,E)$, and the geometric object over here is the moduli of quiver representations associated to $Q$. 

For the Drinfeld realisation, one usually associate the positive half of the quantum group with the cohomological Hall algebra, or the $K$-theoretic Hall algebra for the moduli stack of quiver representations. It was first introduced by Kontsevich and Soibelman \cite{KoSo08}\cite{KoSo10} in the study of the Donaldson-Thomas invariants and wall-crossing formula for it, which has been generalised to many other cases in the study of representation theory and moduli object counting in the enumerative geometry\cite{Dav17}\cite{Ef12}\cite{YZ18}\cite{YZ20}.

In this paper we focus on the preprojective type and nilpotent type \cite{YZ18}, which means that the quiver moduli are chosen as $[\mu^{-1}_{\mbf{v}}(0)/G_{\mbf{v}}]$ for the double quiver of $Q$ or the nilpotent quiver moduli $\Lambda_{\mbf{v}}\subset[\mu^{-1}_{\mbf{v}}(0)/G_{\mbf{v}}]$, which is substack of nilpotent quiver representations. For the case of the preprojective CoHA of quiver type $Q$, the corresponding algebra is regarded as the positive half of the Yangian $Y_{\hbar}^+(\mf{g}_{Q})$. For the case of the preprojective KHA of quiver $Q$, the corresponding algebra is thought of as the positive half of the quantum affine algebras $U_{q}^{+}(\hat{\mf{g}}_{Q})$. The whole quantum group is then realised as the double of such Hall algebras with the multiplication of tautological classes.

For the FRT formalism, the geometric object here is the Nakajima quiver varieties $\mc{M}_{Q}(\mbf{v},\mbf{w})$ \cite{Nak98}\cite{Nak01}. The quantum group from FRT formalism are constructed from the stable envelope class in $\mc{M}_{Q}(\mbf{v},\mbf{w})^A\times\mc{M}_{Q}(\mbf{v},\mbf{w})$ where $A\subset\text{Ker}(q)$ is some suitable torus acting over $\mc{M}_{Q}(\mbf{v},\mbf{w})$. In the case of the equivariant cohomology, the stable envelope will give the cohomological geometric $R$-matrix, which generates the Maulik-Okounkov Yangian algebra $Y_{\hbar}^{MO}(\mf{g}_{Q})$. Similarly in the equivariant $K$-theory, the stable envelope gives the $K$-theoretic geometric $R$-matrix, which generates the Maulik-Okounkov quantum loop group $U_{q}^{MO}(\hat{\mf{g}}_{Q})$.

It is an important conjecture that the double of the cohomological Hall algebra or the $K$-theoretic Hall algebra is isomorphic to the corresponding MO Yangian algebra or the MO quantum loop group. In the cohomological case, this has been proved in \cite{BD23}\cite{SV23}.
\subsection{Main result of the paper}
In this paper the main goal is to prove the isomorphism of algebras between the double of the preprojective $K$-theoretic Hall algebra and the MO quantum loop group.

We denote by $\mc{A}^{+,\mbb{Z}}_{Q}$ as the preprojective $K$-theoretic Hall algebra of quiver type $Q$, and it is defined in Section \ref{subsection:preprojective_k_theoretic_hall_algebra}. On the other side, we consider the Lusztig nilpotent $K$-theoretic Hall algebra $\mc{A}^{+,nilp,\mbb{Z}}_{Q}$, which is defined in \ref{subsection:nilpotent_k_theoretic_hall_algebra}. We also denote $\mc{A}^{0,\mbb{Z}}_{Q}$ as the $\mbb{Z}[q^{\pm1},t_{e}^{\pm1}]$-algebra generated by the tautological classes. As the $\mbb{Z}[q^{\pm1},t_{e}^{\pm1}]$-module, we consider the following integral form $\mc{A}^{ext,\mbb{Z}}_{Q}$:
\begin{align*}
\mc{A}^{ext,\mbb{Z}}_{Q}:=\mc{A}^{+,\mbb{Z}}_{Q}\otimes\mc{A}^{0,\mbb{Z}}_{Q}\otimes(\mc{A}^{+,nilp,\mbb{Z}}_{Q})^{op}.
\end{align*}

As the $\mbb{Z}[q^{\pm1},t_{e}^{\pm1}]_{e\in E}$-algebras, there is an algebra map $\mc{A}^{+,nilp,\mbb{Z}}_{Q}\rightarrow\mc{A}^{+,\mbb{Z}}_{Q}$ from the nilpotent KHA to the preprojective KHA, which is an isomorphism after being localised to $\mbb{Q}(q,t_e)_{e\in E}$. We denote $\mc{A}^{ext}_{Q}$ as the algebra $\mc{A}^{ext,\mbb{Z}}_{Q}$ after being localised to $\mbb{Q}(q,t_e)_{e\in E}$. By \cite{Nak01}\cite{N22}\cite{VV22}, there is an algebra action of $\mc{A}^{ext,\mbb{Z}}_{Q}$ over the localised equivariant $K$-theory of Nakajima quiver varieties $K(\mbf{w}):=K_{T_{\mbf{w}}}(\mc{M}_{Q}(\mbf{w}))_{loc}$.

The first main result of the paper is that we have the isomorphism of the integral form of the double of KHA $\mc{A}^{ext,\mbb{Z}}_{Q}$ defined in \ref{integral-form-double-KHA} and the integral MO quantum loop group $U_{q}^{MO,\mbb{Z}}(\hat{\mf{g}}_{Q})$ as the following:
\begin{thm}[See Theorem \ref{Main-theorem-on-integral-form:theorem} \ref{Isomorphism-main-theorem-on-negative-half:label} \ref{isomorphism-of-integral-positive-half:theorem}]
The Maulik-Okounkov quantum loop group $U_{q}^{MO,\mbb{Z}}(\hat{\mf{g}}_{Q})$ admits the triangular decomposition:
\begin{align*}
U_{q}^{MO,\mbb{Z}}(\hat{\mf{g}}_{Q})\cong U_{q}^{MO,\mbb{Z},+}(\hat{\mf{g}}_{Q})\otimes U_{q}^{MO,\mbb{Z},0}(\hat{\mf{g}}_{Q})\otimes U_{q}^{MO,\mbb{Z},-}(\hat{\mf{g}}_{Q})
\end{align*}
such that as $\mbb{N}^I$-graded $\mbb{Z}[q^{\pm1},t_{e}^{\pm1}]_{e\in E}$-algebras, the negative half $U_{q}^{MO,\mbb{Z},-}(\hat{\mf{g}}_{Q})$ is isomorphic to $(\mc{A}^{+,\mbb{Z},nilp}_{Q})^{op}$ the opposite algebra of the nilpotent $K$-theoretic Hall algebra, and the positive half is isomorphic to $\mc{A}^{+,\mbb{Z}}_{Q}$ the preprojective $K$-theoretic Hall algebra. The Cartan part $U_{q}^{MO,\mbb{Z},0}$ is isomorphic to $\mc{A}^{0,\mbb{Z}}_{0}$.

In other words, we have the isomorphism of $\mbb{Z}[q^{\pm},t_e^{\pm1}]_{e\in E}$-algebras:
\begin{align*}
\mc{A}^{ext,\mbb{Z}}_{Q}\cong U_{q}^{MO,\mbb{Z}}(\hat{\mf{g}}_{Q}).
\end{align*}

Moreover, the above isomorphisms intertwine the action over $K_{T_{\mbf{w}}}(\mc{M}_{Q}(\mbf{w}))$.
\end{thm}

On both sides for $\mc{A}^{ext,\mbb{Z}}_{Q}$ and $U^{MO,\mbb{Z}}_{q}(\hat{\mf{g}}_{Q})$, they both admit the root factorisation. On the side of the MO quantum loop group $U^{MO,\mbb{Z}}_{q}(\hat{\mf{g}}_{Q})$, it can be factorised as the wall subalgebra $U_{q}^{MO}(\mf{g}_{w})$, where the wall $w$ refers to the affine hyperplane arrangement dual to the affine root $\alpha$ in the real Picard space $\text{Pic}(\mc{M}_{Q}(\mbf{w}))\otimes\mbb{R}\cong\mbb{R}^{|I|}$.

Similarly, on the side of the double KHA $\mc{A}^{ext}_{Q}$, one also admits the factorisation given by the slope subalgebra $\mc{B}_{\mbf{m}}$ with $\mbf{m}\in\mbb{Q}^{|I|}\subset\mbb{R}^{|I|}$. This subalgebra can be thought of as the algebra generated by the wall subalgebra such that the corresponding wall $w$ contains the point $\mbf{m}$. This means that one can give a refined subalgebra $\mc{B}_{\mbf{m},w}$, which we call it the root subalgebra of the slope subalgebra $\mc{B}_{\mbf{m}}$.

It turns out that when the above isomorphism is restricted to the wall subalgebra and the root subalgebra, we can have the isomorphism as the $\mbb{Z}[q^{\pm1},t_e^{\pm1}]_{e\in E}$-Hopf subalgebra on both sides:
\begin{prop}[See Proposition \ref{isomorphism-of-localised-wall-subalgebra:proposition} \ref{isomorphism-of-integral-wall-subalgebra:proposition}]
There is an isomorphism of quasi-triangular $\mbb{Z}[q^{\pm1},t_{e}^{\pm1}]$-Hopf algebras 
\begin{align*}
(\mc{B}_{\mbf{m},w}^{\mbb{Z}},R_{\mbf{m},w}^+,\Delta_{\mbf{m}},S_{\mbf{m}},\epsilon,\eta)\cong (U_{q}^{MO,\mbb{Z}}(\mf{g}_{w}),q^{-\Omega}(R_{w}^-)^{-1},\Delta_{\mbf{m}}^{MO,op},S_{\mbf{m}}^{MO},\epsilon,\eta)
\end{align*}
which intertwines the action over $K_{T_{\mbf{w}}}(\mc{M}_{Q}(\mbf{w}))$.

\end{prop}

The second main result of the paper is that we have the isomorphism of the integral double KHA $\mc{A}^{ext}_{Q}$ and the integral MO quantum loop group as Hopf $\mbb{Z}[q^{\pm1},t_{e}^{\pm1}]_{e\in E}$-algebras:
\begin{thm}[See Theorem \ref{main-theorem-just-localised-algebra} and \ref{isomorphism-of-slope-wall-quasi-triangular:label}]
There is an isomorphism of Hopf $\mbb{Z}[q^{\pm1},t_{e}^{\pm1}]_{e\in E}$-algebras between the Maulik-Okounkov quantum loop group and the integral extended double KHA $\mc{A}^{ext,\mbb{Z}}_{Q}$ defined in \ref{integral-form-double-KHA}:
\begin{align}
(U_{q}^{MO,\mbb{Z}}(\hat{\mf{g}}_{Q}),\Delta_{\mbf{m}}^{MO,op},S_{\mbf{m}},\epsilon,\eta)\cong(\mc{A}^{ext,\mbb{Z}}_{Q},\Delta_{(\mbf{m})},S_{\mbf{m}},\epsilon,\eta)
\end{align}
which intertwines the action over $K_{T_{\mbf{w}}}(\mc{M}_{Q}(\mbf{w}))$. Here the coproduct $\Delta_{(\mbf{m})}$ is defined in \ref{definition-of-m-universal-coproduct}. Moreover, when restricted to the wall subalgebra on both sides, we have an isomorphism of quasi-triangular $\mbb{Z}[q^{\pm1},t_{e}^{\pm1}]$-Hopf algebras 
\begin{align*}
(\mc{B}_{\mbf{m},w}^{\mbb{Z}},R_{\mbf{m},w}^+,\Delta_{\mbf{m}},S_{\mbf{m}},\epsilon,\eta)\cong (U_{q}^{MO,\mbb{Z}}(\mf{g}_{w}),q^{-\Omega}(R_{w}^-)^{-1},\Delta_{\mbf{m}}^{MO,op},S_{\mbf{m}}^{MO},\epsilon,\eta)
\end{align*}
\end{thm}

\textbf{Remark. }Specifically, when we take $\mbf{m}=\mbf{0}$. It is expected that $\mc{B}_{\mbf{0}}$ should be the Hopf algebra deformation of the universal enveloping algebra of the BPS Lie algebra $U(\mf{g}_{Q}^{BPS})$. On the right hand side, if we think of $U_{q}^{MO}(\mf{g}_{\mbf{0}})$ as the algebra generated by the wall subalgebra $U_{q}^{MO}(\mf{g}_{w})$ such that the wall $w$ contains $\mbf{0}$, we can think of $U_{q}^{MO}(\mf{g}_{\mbf{0}})$ as the Hopf algebra deformation of the universal enveloping algebra of MO Lie algebra $U(\mf{g}_{Q}^{MO})$. Thus this statement can be thought of as the Hopf deformation of the isomorphism of Lie algebras:
\begin{align*}
\mf{g}^{BPS}_{Q}\cong\mf{g}^{MO}_{Q}
\end{align*}
which has been proved in \cite{BD23}.

As the application, we prove the freeness of the preprojective KHA for arbitrary equivariant parametres:
\begin{thm}[See Theorem \ref{freeness-of-preKHA:label}]
Given $\mbb{C}_{q}^*\subset A\subset T$ a subtorus of $T$ which contains $\mbb{C}_{q}^*$. The $A$-equivariant $K$-theory $K_{A}(\mc{Y}_{\mbf{v}})$ of the preprojective stack is a free $K_{A}(pt)$-module.
\end{thm}
It would be really interesting to investigate a proof of the freeness of the preprojective KHA without using the stable envelopes, where the freeness depends on the freeness of the equivariant $K$-theory of the Nakajima quiver varieties.

\subsection{Strategy of the proof}
The proof of our first main Theorem \ref{Main-theorem-on-integral-form:theorem} mostly depends on the proof of the isomorphism of the positive half stated in \ref{isomorphism-of-integral-positive-half:theorem}. Since both preprojective KHA and MO quantum loop group admit the factorisation property, which are stated in Theorem \ref{Main-theorem-on-slope-factorisation:Theorem}, \ref{slope-factoristaion-integral} and Proposition \ref{factorisation-for-wall-subalgebras:label}, it is equivalent to proving the isomorphism on each wall algebra and slope subalgebra pieces. 

The proof can be factorised into the following two steps:
\begin{enumerate}
	\item The Hopf $\mbb{Q}(q,t_e)_{e\in E}$-algebra embedding in Theorem \ref{Hopf-embedding}
	\begin{align*}
    (\mc{B}_{\mbf{m},w},\Delta_{\mbf{m}},S_{\mbf{m}},\eta,\epsilon)\hookrightarrow (U_{q}^{MO}(\mf{g}_{w}),\Delta_{\mbf{m}}^{MO},S_{\mbf{m}}^{MO},\eta,\epsilon).
    \end{align*}
    \item The isomorphism of $\mbb{Z}[q^{\pm1},t_{e}^{\pm1}]_{e\in E}$-algebras
    \begin{align*}
    \mc{B}_{\mu\bm{\theta},w}^{+,\mbb{Z}}\cong U_{q}^{MO,+,\mbb{Z}}(\mf{g}_{w}).
    \end{align*}
    and here $\mu\in\mbb{Q}$ and $\bm{\theta}=(1,\cdots,1)$, which is the key in the proof of Theorem \ref{isomorphism-of-integral-positive-half:theorem}.
\end{enumerate}

For the first step, we first need to prove that the coproduct operation $\Delta_{\mbf{m}}$ coincides with the geometric coproduct $\Delta_{\mbf{m}}^{MO}$ when restricted to $\mc{B}_{\mbf{m},w}$. After that, since $U_{q}^{MO,+,\mbb{Z}}(\mf{g}_{w})$ is generated by the matrix coefficients of the wall $R$-matrix $R_{w}^{\pm}$. One then can follow the computation as given in \ref{Step III:Injectivity as Hopf algebras:Step-proof} to show the injectivity with the Theorem \ref{injectivity-of-localised-double-KHA:theorem}. The first step also implies that we have an injective map of $\mbb{Q}(q,t_e)_{e\in E}$-algebras
\begin{align}
\mc{A}^{ext}_{Q}\hookrightarrow U_{q}^{MO}(\hat{\mf{g}}_{Q}).
\end{align}
For the detail of the proof one can refer to Section \ref{section:_textbf_isomorphism_as_the_hopf_algebras}.

For the second step, the key observation is on two aspects: The action map of both $\mc{B}_{\mbf{m},w,\mbf{v}}^{+,\mbb{Z}}$ and $U_{q}^{MO,+,\mbb{Z}}(\mf{g}_{w})$ on $K_{T_{\mbf{w}}}(\mc{M}_{Q}(\mbf{0},\mbf{w}))\rightarrow K_{T_{\mbf{w}}}(\mc{M}_{Q}(\mbf{v},\mbf{w}))$ factor through $K_{T_{\mbf{w}}}(\mc{M}_{Q}(\mbf{0},\mbf{v},\mbf{w}))$. In fact, when $\mbf{w}=\mbf{v}$, one has the slope factorisation of $K_{T_{\mbf{w}}}(\mc{M}_{Q}(\mbf{0},\mbf{v},\mbf{w}))$ in terms of the slope subalgebra given by Padurariu and Toda \cite{PT25} in Theorem \ref{categorical-slope-factorisation:Theorem}. In this case one can also have the commutative diagram \ref{The-most-important-diagram}, and this commutative diagram implies the isomorphism. For the details one can refer to Section \ref{sub:isomorphism_of_the_positive_integral_half}.

The rest of the proof is for the negative half on both sides, i.e. the isomorphism of the opposite algebra of the nilpotent KHA $(\mc{A}^{+,nilp,\mbb{Z}}_{Q})^+$ and the negative half of the MO quantum loop group $U_q^{MO,-,\mbb{Z}}(\hat{\mf{g}}_{Q})$. The proof is quite similar to the positive half case, and one thing that we can use over here is to replace the negative half of the MO quantum loop group by the positive half of the nilpotent MO quantum loop group $U_{q}^{MO,nilp,+,\mbb{Z}}(\hat{\mf{g}}_{Q})$, which is introduced in Section \ref{subsection:nilpotent_maulik_okounkov_quantum_loop_groups}. It shares the same property as mentioned above as the positive half of the MO quantum loop group. Therefore one can use the similar strategy listed above to prove the isomorphism of the negative half.

\subsection{Outline of the paper}

The structure of the paper is organised as follows:

In Section \ref{section:_textbf_k_theoretic_hall_algebras_and_geometric_modules}, we introduce the basic notion for the preprojective $K$-theoretic Hall algebra $\mc{A}^{+,\mbb{Z}}_{Q}$ and the nilpotent $K$-theoretic Hall algebra $\mc{A}^{+,nilp,\mbb{Z}}_{Q}$. We also introduce their localised form and their double. Then we will introduce their algebra action over the corresponding equivariant $K$-theory on Nakajima quiver varieties $\mc{M}_{Q}(\mbf{w})$ and nilpotent quiver varieties $\mc{L}_{Q}(\mbf{w})$ respectively. We will KHA to stand for either the preprojective $K$-theoretic Hall algebra or the nilpotent $K$-theoretic Hall algebra.

In Section \ref{section:slope_filtration_and_shuffle_realisations_of_k_theoretic_hall_algebras}, we introduce the slope filtration and the shuffle realisation for the KHA $\mc{A}^{+}_{Q}$ and its preprojective and nilpotent integral version $\mc{A}^{+,\mbb{Z}}_{Q}$ and $\mc{A}^{+,nilp,\mbb{Z}}_{Q}$ respectively. We also introduce the slope subalgebra $\mc{B}_{\mbf{m}}^{\pm}$ inside of these algebras using the slope filtrations, and after doing the bialgebra pairing, we can also generate a Hopf algebra $\mc{B}_{\mbf{m}}$ with the coproduct $\Delta_{\mbf{m}}$ defined as \ref{coproduct-m-slope-subalgebra-positive} and \ref{coproduct-m-slope-subalgebra-negative}. Moreover, we show that $\mc{B}_{\mbf{m}}$ is generated by the primitive elements in the sense of \ref{primitive-element-definition}. In this way we introduce the root subalgebra $\mc{B}_{\mbf{m},w}$ and its integral version $\mc{B}_{\mbf{m},w}^{+,\mbb{Z}}$, $\mc{B}_{\mbf{m},w}^{+,nilp,\mbb{Z}}$ respectively.

In Section \ref{section:_textbf_stable_envelopes_and_maulik_okounkov_quantum_loop_groups} we introduce the stable envelopes for both Nakajima quiver varieties $\mc{M}_{Q}(\mbf{v},\mbf{w})$ and the nilpotent quiver varieties $\mc{L}_{Q}(\mbf{v},\mbf{w})$. Using these stable envelopes we introduce the geometric $R$-matrix and the nilpotent geometric $R$-matrix and their factorisation property in \ref{factorisation-geometry} \ref{nilpotent-factorisation}. Using the geometric $R$-matrices, we define the Maulik-Okounkov quantum loop group $U_{q}^{MO}(\hat{\mf{g}}_{Q})$ with its integral form $U_{q}^{MO,\mbb{Z}}(\hat{\mf{g}}_{Q})$, and the corresponding nilpotent Maulik-Okounkov quantum loop group $U_{q}^{MO,nilp}(\hat{\mf{g}}_{Q})$ and its integral form $U_{q}^{MO,nilp,\mbb{Z}}(\hat{\mf{g}}_{Q})$.

In Section \ref{section:_textbf_isomorphism_as_the_hopf_algebras} and Section \ref{section:isomorphism_as_the_integral_form} we give the proof of the main theorems \ref{big-boss-theorem:label} and \ref{Main-theorem-on-integral-form:theorem}.

\subsection{Future directions and related works}
\subsubsection{}
Many aspects of the $K$-theoretic Hall algebra have not been studied as well as those for the cohomological Hall algebra, such as the integrality structure \cite{Dav23} and so on. In fact, in the story of the KHA, this corresponds to the conjecture that the preprojective $K$-theoretic Hall algebra is a free $\mbb{Z}[q^{\pm1},t_{e}^{\pm1}]_{e\in E}$-module. In fact, this would lead to the main theorem \ref{Main-theorem-on-integral-form:theorem} of the isomorphism of $\mbb{Z}[q^{\pm1},t_{e}^{\pm1}]_{e\in E}$-algebras. 

On the other hand, from the aspects of the slope filtration, one can get some more refined structure of the factorisation on $K$-theoretic Hall algebra such as the slope subalgebra from the slope filtration, which should have strong connection with the BPS Lie algebra in cohomological setting \cite{DM20} and KBPS Lie algebra in the $K$-theory setting \cite{Pa19}. The statement can be roughly stated as follows:
\begin{conjecture}
For the slope subalgebra $\mc{B}_{\mbf{0}}^{+}$, it is a Hopf algebra deformation of the universal enveloping algebra of the positive half BPS Lie algebra $U(\mf{n}_{Q}^{BPS})$.
\end{conjecture}
While the difference is that the BPS Lie algebra comes from the perverse filtration, and the slope subalgebra $\mc{B}_{\mbf{0}}^+$ comes from the slope filtration. It would be an interesting question to connect the perverse filtration on CoHA and the slope filtration on KHA. Unfortunately, besides the shuffle algebra interpretation, for now we still lack the precise geometric understanding for the slope filtration of various kind of KHA, even for the nilpotent KHA and preprojective KHA that we are using in this paper.

As a result of the conjecture, this implies the Kac polynomial conjecture for the slope subalgebra $\mc{B}_{\mbf{0}}^{+}$, which was stated in \cite{N22}. 

Moreover, the analog of the Kac polynomial for $\mc{B}_{\mbf{m}}$ should be computed from the Kac polynomial on each root subalgebra for $\mc{B}_{\mbf{m},w}$, which can be thought of as a consequence of the above conjecture. Moreover, it is expected that the Kac polynomial for $\mc{B}_{\mbf{m}}$ should also be controlled by the Kac polynomial of $\mc{B}_{\mbf{0}}$ in \cite{NS25}.

\subsubsection{}
In the paper \cite{Z24} \cite{Z24-2}, we have shown the isomorphism of the MO quantum loop group of affine type $A$ and the quantum toroidal algebras using the techniques on computing the monodromy representation for the Dubrovin connections, or the quantum differential equations. The method over there is to use the comparison of the computation of monodromy representations by reducing the quantum difference equations on algebraic and geometric side to compare the universal $R$-matrix on slope subalgebras $\mc{B}_{\mbf{m}}$ and wall subalgebras $U_q^{MO}(\mf{g}_{w})$.

For the general case, we can construct the algebraic and geometric quantum difference equations for arbitrary Nakajima quiver varieties. For now the author does not have too much understanding connecting the two proofs, and it would be a really interesting question.

\subsubsection{}
From the side of the representation theory for the quantum affine algebras, equivariant $K$-theory of Nakajima quiver varieties gives a subclass of weighted representations for the quantum affine algebras. For the other types of the important modules such as the Kirillov-Reshetikhin modules \cite{KR90} for quantum affine algebras $U_{q}(\hat{\mf{g}}_{Q})$ and MacMahon modules for the quantum toroidal algebras \cite{FJMM12}, they can be realised as the equivariant critical K-theory of the some quiver varieties as in \cite{VV22} and \cite{RSYZ20}. It would be really interesting if we can extend the above isomorphism to the critical $K$-theoretic Hall algebra constructed in \cite{Pa19}\cite{Pa23} and the geometric quantum loop group constructed by the critical $K$-theoretic stable envelope which is being developed by \cite{COZZ25}.

\subsection*{Acknowledgements.}
The author is very thankful to Yalong Cao, Andrei Negu\c{t}, Andrei Okounkov, Yehao Zhou and Zijun Zhou for many insightful discussions throughout these topics on stable envelopes, shuffle algebras and quantum groups. The author is partially supported by the international collaboration grant BMSTC and ACZSP (Grant no. Z221100002722017) and by the National Key R \& D Program of
China (Grant no. 2020YFA0713000).

\section{\textbf{K-theoretic Hall algebras and geometric modules}}\label{section:_textbf_k_theoretic_hall_algebras_and_geometric_modules}
In this Section we review the construction of the preprojective $K$-theoretic Hall algebra, nilpotent $K$-theoretic Hall algebra and their geometric modules. For the reference of the $K$-theoretic Hall algebras, one can refer to \cite{VV22}\cite{YZ18}

Let $Q=(I,E)$ be a quiver with a finite vertex set $I$ and a finite edge set $E$. Edge loops and multiple edges are allowed.

We set the base field as:
\begin{align}
\mbb{F}=\mbb{Q}(q,t_e)_{e\in E}.
\end{align}
Let $\mbf{n}=(n_i\geq0)_{i\in I}$ be a sequence of non-negative integers indexed by $I$, and we define
\begin{align*}
\mbf{n}!=\prod_{i\in I}n_i!.
\end{align*}

\subsection{Preprojective K-theoretic Hall algebra}\label{subsection:preprojective_k_theoretic_hall_algebra}

For any $\mbf{n}\in\mbb{N}^I$, we consider the stack of $\mbf{n}$-dimensional quiver representations of $Q$:
\begin{align*}
\mc{X}_{\mbf{n}}=\bigoplus_{ij=e\in E}T^*\text{Hom}(V_{i},V_{j})/\prod_{i\in I}GL(V_i).
\end{align*}
Here $V_i$ denotes a vector space of dimension $n_i$ for every node $i\in I$. We now impose the moment map:
\begin{align*}
\mu:\bigoplus_{ij=e\in E}T^*\text{Hom}(V_i,V_j)\rightarrow\bigoplus_{i\in I}\mf{gl}(V_i)^*,\qquad\mu(A_{e},B_{e})=\sum_{e\in E}(A_eB_e-B_eA_e).
\end{align*}

We consider the moduli stack of preprojective $\mbb{C}[\overline{Q}]$-representations:
\begin{align}\label{preprojective-stack}
\mc{Y}_{\mbf{n}}:=[\mu^{-1}_{\mbf{n}}(0)/\prod_{i\in I}GL(V_i)].
\end{align}
and here $\overline{Q}=(I,E\sqcup E^{op})$ stands for the double quiver of $Q$. As a $\mbb{Z}[q^{\pm1},t_{e}^{\pm1}]_{e\in E}$-module, the \textbf{preprojective K-theoretic Hall algebra }of the quiver type $Q$ is the direct sum of the equivariant algebraic $K$-theory groups of the cotangent bundle of the stack $\mc{Y}_{\mbf{n}}$.
\begin{align}
\mc{A}_{Q}^{+,\mbb{Z}}=(\bigoplus_{\mbf{n}\in\mbb{N}^I}K_{T}(\mc{Y}_{\mbf{n}}),*)
\end{align}
where the torus $T=\mbb{C}^*_q\times\prod_{e\in E}\mbb{C}^*_{t_e}$ acts on $\mc{Y}_{\mbf{n}}$ as follows: 
\begin{align*}
(q,t_e)_{e\in E}\cdot(X_e,Y_e)_{e\in E}=(\frac{X_e}{t_e},\frac{t_eY_e}{q})_{e\in E}.
\end{align*}

The Hall product is given by the following correspondence:
\begin{equation}\label{Preprojective-Hall-Product}
\begin{tikzcd}
&\mc{Y}_{\mbf{n},\mbf{m}}\arrow[ld,"\pi_1"]\arrow[rd,"\pi_2"]&\\
\mc{Y}_{\mbf{n}+\mbf{m}}&&\mc{Y}_{\mbf{n}}\times\mc{Y}_{\mbf{m}}
\end{tikzcd}
\end{equation}
where $\mc{Y}_{\mbf{n},\mbf{m}}$ is the moduli stack of the correspondence:
\begin{equation*}
\begin{aligned}
\mc{Y}_{\mbf{n},\mbf{m}}:=&\{(X_e,Y_e)_{e\in E}\in\bigoplus_{ij=e\in E}\text{Hom}(\mbb{C}^{n_i+m_i},\mbb{C}^{n_j+m_j})\oplus\text{Hom}(\mbb{C}^{n_i+m_i},\mbb{C}^{n_j+m_j})\\
&|\sum_{e\in E}(X_{i(e)}Y_{o(e)}-Y_{i(e)}X_{o(e)})=0\text{ and }(X_e,Y_e)\text{ preserves }\mbb{C}^{n_i}\}.
\end{aligned}
\end{equation*}

The Hall product is defined as:
\begin{equation}\label{Hall-product-expression}
\begin{aligned}
&*:K_{T}(\mc{Y}_{\mbf{n}})\otimes K_{T}(\mc{Y}_{\mbf{m}})\rightarrow K_{T}(\mc{Y}_{\mbf{m}+\mbf{n}})\\
&\alpha\boxtimes\beta\mapsto (\pi_1)_{*}(\text{sdet}[\sum_{i\in I}\frac{\mc{V}_i'}{q\mc{V}_i''}]\cdot\pi_2^!(\alpha\boxtimes\beta)).
\end{aligned}
\end{equation}

Here $\pi_2^!$ is the refined Gysin pullback defined and explained in \cite{YZ18}. The line bundle $\text{sdet}(\cdots)$ was chosen in \cite{N23} in order to match the formula appearing in the computation for the stable envelopes. Here we also fix this line bundle for the computation. The tautological bundles $\mc{V}_i'$, $\mc{V}_{i}''$ are corresponding to the one induced on $\mc{Y}_{\mbf{n}}$ and $\mc{Y}_{\mbf{m}}$ respectively. The Hall product makes $\mc{A}^{+,\mbb{Z}}_{Q}$ as a $\mbb{Z}[q^{\pm1},t_{e}^{\pm1}]_{e\in E}$-algebra.

\textbf{Remark.} It should be noted that the choice of the line bundle $\text{sdet}(\cdots)$ can be set for others if we change the polarisation in the definition of the $K$-theoretic stable envelopes. If we change the line bundle, everything in this Section stated is still true.

The following theorem has been proved in \cite{VV22}:
\begin{thm}[See Lemma 2.4.2 in \cite{VV22}]\label{torsion-freeness-of-KHA:theorem}
The preprojective $K$-theoretic Hall algebra $\mc{A}_{Q}^{+,\mbb{Z}}$ is a torsion-free $\mbb{Z}[q^{\pm1},t_{e}^{\pm1}]_{e\in E}$-module.
\end{thm}

In many parts of the paper, we may consider its localization with respect to the fraction field $\mbb{F}=\mbb{Q}(q,t_{e})_{e\in E}$, and we will use the following notations:
\begin{align*}
\mc{A}_{Q}^+=\mc{A}_{Q}^{+,\mbb{Z}}\otimes_{\mbb{Z}[q^{\pm1},t_{e}^{\pm1}]_{e\in E}}\mbb{Q}(q,t_e)_{e\in E}.
\end{align*}

\subsection{Nilpotent K-theoretic Hall algebra}\label{subsection:nilpotent_k_theoretic_hall_algebra}
Other than the preprojective $K$-theoretic Hall algebra, we also need to introduce the nilpotent $K$-theoretic Hall algebra, one can also refer to \cite{VV22} for the detailed construction.

Now given $(A_e,B_e)\in\bigoplus_{ij\in E}T^*\text{Hom}(V_i,V_j)$, we say that $X$ is \textbf{nilpotent} if there exists a flag $\{L^l\}$ of $I$-graded vector spaces $V=\bigoplus_{i\in I}V_i$ such that:
\begin{align*}
A_e(L^l)\subset L^{l-1},\qquad B_e(L^l)\subset L^{l-1}
\end{align*}
and we denote the subspace of nilpotent representations in $\bigoplus_{ij\in E}T^*\text{Hom}(V_i,V_j)$ as $E^0_{\mbf{v}}$.

The Lusztig nilpotent quiver variety is defined as:
\begin{align}\label{nilpotent-stack}
\Lambda_{\mbf{v}}:=[\mu^{-1}_{\mbf{v}}(0)\cap E^{0}_{\mbf{v}}/G_{\mbf{v}}].
\end{align}

Similar to the construction of the preprojective $K$-theoretic Hall algebra, we can define the nilpotent $K$-theoretic Hall algebra as:
\begin{align*}
\mc{A}^{+,nilp,\mbb{Z}}_{Q}:=(\bigoplus_{\mbf{v}\in\mbb{N}^I}K_{T}(\Lambda_{\mbf{v}}),*)
\end{align*}
where the Hall product is defined similarly as \ref{Preprojective-Hall-Product}:
\begin{equation*}
\begin{tikzcd}
&\Lambda_{\mbf{n},\mbf{m}}\arrow[ld,"\pi_1"]\arrow[rd,"\pi_2"]&\\
\Lambda_{\mbf{n}+\mbf{m}}&&\Lambda_{\mbf{n}}\times\Lambda_{\mbf{m}}
\end{tikzcd}
\end{equation*}
with $\Lambda_{\mbf{n},\mbf{m}}$ the nilpotent version of the correspondence:
\begin{equation*}
\begin{aligned}
\Lambda_{\mbf{n},\mbf{m}}:=&\{(X_e,Y_e)_{e\in E}\in\bigoplus_{ij=e\in E}\text{Hom}(\mbb{C}^{n_i+m_i},\mbb{C}^{n_j+m_j})\oplus\text{Hom}(\mbb{C}^{n_i+m_i},\mbb{C}^{n_j+m_j})\\
&|\sum_{e\in E}(X_{i(e)}Y_{o(e)}-Y_{i(e)}X_{o(e)})=0\text{ and }(X_e,Y_e)\text{ preserves }\mbb{C}^{n_i}\text{ and }(X_e,Y_e)\text{ are nilpotent.}\}.
\end{aligned}
\end{equation*}

It can be checked that the natural closed embedding $i:\Lambda_{\mbf{n}}\hookrightarrow\mc{Y}_{\mbf{n}}$ induce the morphism of the Hall algebras:
\begin{align*}
i_*:\mc{A}^{+,nilp,\mbb{Z}}_{Q}\rightarrow\mc{A}^{+,\mbb{Z}}_{Q}.
\end{align*}

The following has also been proved in \cite{VV22}:
\begin{thm}[See Lemma 2.4.1 in \cite{VV22}]\label{freeness-of-nilpotent-KHA-and-isomorphism-after-localisation:theorem}
The nilpotent $K$-theoretic Hall algebra $\mc{A}^{+,nilp,\mbb{Z}}_{Q}$ is a free $\mbb{Z}[q^{\pm1},t_{e}^{\pm1}]$-module. Moreover, after localising to $\mbb{F}:=\mbb{Q}(q,t_e)_{e\in E}$, the morphism $i_*$ induces an isomorphism between $\mc{A}^{+,\mbb{Z}}_{Q,loc}$ and $\mc{A}^{+,nilp,\mbb{Z}}_{Q,loc}$.
\end{thm}

In simplicity, we will use $\mc{A}^+_{Q}$ as the localised form of $\mc{A}^{+,\mbb{Z}}_{Q}$ or $\mc{A}^{+,nilp,\mbb{Z}}_{Q}$. Moreover, it has been shown in Theorem 1.2 of \cite{N21} that the localised KHA $\mc{A}^{+}_{Q}$ is generated by its spherical part $\bigoplus_{i\in I}\mc{A}^{+}_{Q,\mbf{e}_i}$. This means that $\mc{A}^{+}_{Q}$ is generated by $e_{i,d}$ with $i\in I$ and $d\in\mbb{Z}$ such that:
\begin{align}\label{positive-spherical-generator}
e_{i,d}=\mc{L}_{i}^d\in\mc{A}^{+}_{Q,\mbf{e}_i}.
\end{align}

Below in the paper, we will change the twisted line bundle in the definition of the Hall product \ref{Hall-product-expression} of the nilpotent KHA $\mc{A}^{+,nilp,\mbb{Z}}_{Q}$ by $\text{sdet}[\sum_{i\in I}\frac{\mc{V}_i'}{q\mc{V}_i''}-\sum_{e=ij\in E}\frac{t_e\mc{V}_i'}{q\mc{V}_j'}]$.

\subsection{Localised and integral extended double KHA}\label{subsection:localised_and_integral_extended_double_kha}
\subsubsection{Localised form extended double KHA $\mc{A}^{ext}_{Q}$}\label{subsub:localised_and_integral_extended_double_kha}
The extended double KHA algebra $\mc{A}^{ext}_{Q}$ is defined as follows: 
\begin{equation}\label{ext-double-KHA}
\mc{A}^{ext}_{Q}:=\mc{A}_{Q}^+\otimes\mc{A}_{Q}^{0,ext}\otimes\mc{A}_{Q}^{-}
\end{equation}
with:
\begin{align}
\mc{A}^{+}_{Q}:=\mbb{F}[e_{i,d}]_{i\in I,d\in\mbb{Z}},\qquad\mc{A}^{-}_{Q}:=\mbb{F}[f_{i,d}]_{i\in I,d\in\mbb{Z}}
\end{align}
and here we use the same notation $e_{i,d}$ as in \ref{positive-spherical-generator} for the positive half. It is convenient for us to write down the following generating functions:
\begin{align*}
e_{i}(z)=\sum_{d\in\mbb{Z}}\frac{e_{i,d}}{z^d},\qquad f_{i}(z)=\sum_{d\in\mbb{Z}}\frac{f_{i,d}}{z^d}.
\end{align*}

Here $\mc{A}_{Q}^{0,ext}$ is the polynomial ring:
\begin{align}\label{Cartan-part-double-KHA-localised}
\mc{A}_{Q}^{0,ext}:=\mbb{F}[a_{i,\pm d},b_{i,\pm d},q^{\pm\frac{v_i}{2}},q^{\pm\frac{w_i}{2}}]_{i\in I,d\geq1}
\end{align}
with $q^{\pm\frac{w_i}{2}}$ and $b_{i,\pm d}$ are central elements. The rest of the generators have the following relations:
\begin{align*}
&e_{i}(z)h_{j}^{\pm}(w)=h_{j}^{\pm}(w)e_{i}(z)\frac{\zeta_{ij}(z/w)}{\zeta_{ji}(w/z)},\qquad e_i(z)e_j(w)\zeta_{ji}(\frac{w}{z})=e_j(w)e_i(z)\zeta_{ij}(\frac{z}{w})\\
&f_{i}(z)h_{j}^{\pm}(w)=h_{j}^{\pm}(w)f_{i}(z)\frac{\zeta_{ji}(w/z)}{\zeta_{ij}(z/w)},\qquad f_i(z)f_j(w)\zeta_{ij}(\frac{z}{w})=f_j(w)f_i(z)\zeta_{ji}(\frac{w}{z})
\end{align*}
\begin{align*}
e_i(z)q^{\pm\frac{v_j}{2}}=q^{\pm\frac{v_j}{2}}e_j(z)\cdot q^{\mp\frac{\delta_{ij}}{2}}
\end{align*}
\begin{align*}
f_{i}(z)q^{\pm\frac{v_j}{2}}=q^{\pm\frac{v_j}{2}}f_j(z)\cdot q^{\pm\frac{\delta_{ij}}{2}}
\end{align*}
\begin{align*}
[e_i(z),a_{j,d}]=e_{i}(z)\cdot\delta_{ij}z^d(q^{-d}-1)
\end{align*}
\begin{align*}
[f_i(z),a_{j,d}]=f_{i}(z)\cdot\delta_{ij}z^d(1-q^{-d})
\end{align*}

\begin{align*}
[e_{i,d},f_{j,k}]=\delta_{ij}\cdot\gamma_i
\begin{cases}
-h_{i,d+k}&\text{if }d+k>0\\
h_{i,0}^{-1}-h_{i,0}&\text{if }d+k=0\\
h_{i,d+k}&\text{if }d+k<0
\end{cases}.
\end{align*}

Here $h_{i}^{\pm}(z)$ is written as:
\begin{equation}\label{Definition-of-Cartan-current}
\begin{aligned}
h_{i}^{\pm}(z)=&h_{i,0}^{\pm1}+\sum_{d=1}^{\infty}\frac{h_{i,\pm d}}{z^{\pm d}}:=(q^{\pm\frac{1}{2}})^{w_i-2v_i+\sum_{e=ij}v_j+\sum_{e=ji}v_j}\times\\
&\times\exp(\sum_{d=1}^{\infty}\frac{b_{i,\pm d}-a_{i,\pm d}(1+q^{\pm d})+\sum_{e=ij}a_{j,\pm d}q^{\pm d}t_{e}^{\mp d}+\sum_{e=ji}a_{j,\pm d}t_{e}^{\pm d}}{dz^{\pm d}})
\end{aligned}
\end{equation}
and $\zeta_{ij}(x)$ is defined as:
\begin{align}\label{zeta-function}
\zeta_{ij}(x)=(\frac{1-xq^{-1}}{1-x})^{\delta^i_j}\prod_{e=ij\in E}(1-t_ex)\prod_{e=ji\in E}(1-\frac{q}{t_ex})
\end{align}
and $\gamma_i$ is defined as:
\begin{align*}
\gamma_i=\frac{\prod_{e=ii}[(1-t_e^{-1}q)(1-t_e)]}{1-q^{-1}}.
\end{align*}

The Drinfeld coproduct structure over $\mc{A}^{ext}_Q$ can be written as:
\begin{align*}
&\Delta(h_{i}^{\pm}(z))=h_{i}^{\pm}(z)\otimes h_{i}^{\pm}(z)\\
&\Delta(e_i(z))=e_i(z)\otimes 1+h_{i}^+(z)\otimes e_{i}(z)\\
&\Delta(f_{i}(z))=f_{i}(z)\otimes h_{i}^-(z)+1\otimes f_{i}(z).
\end{align*}

In fact the generators for the localised double KHA $\mc{A}^{ext}_{Q}$ can be reduced to some smaller number of generators. The following proposition was proved in \cite{N23}:
\begin{prop}[See Proposition 2.10 in \cite{N23}]\label{minimal-set-generators-KHA:proposition}
The algebra $\mc{A}^{ext}_{Q}$ is generated by $\{e_{i,0},f_{i,0},q^{\pm\frac{v_i}{2}},a_{i,\pm1}\}_{i\in I}$ and $\{q^{\pm\frac{w_i}{2}},b_{i,\pm d}\}_{i\in I,d\geq0}$.
\end{prop}

\subsection{Geometric action on quiver varieties}
In this subsection we introduce the geometric action of the KHA over the equivariant $K$-theory of Nakajima quiver varieties. One can refer to \cite{Nak01}\cite{N22}\cite{N23} for details.

\subsubsection{Nakajima quiver variety}
Nakajima quiver variety was first introduced in \cite{Nak98}\cite{Nak01}. Here we review the construction.

For the quiver $Q$ and any vector $\mbf{v},\mbf{w}\in\mbb{N}^{I}$, consider the affine space:
\begin{align*}
T^*\text{Rep}_{Q}(\mbf{v},\mbf{w})=\bigoplus_{ij=e\in E}[\text{Hom}(V_i,V_j)\oplus\text{Hom}(V_j,V_i)]\bigoplus_{i\in I}[\text{Hom}(W_i,V_i)\oplus\text{Hom}(V_i,W_i)].
\end{align*}

Here $\text{dim}(V_i)=v_i,\text{dim}(W_i)=w_i$ for all $i\in I$. And points of the affine space above can be denoted by quadruples:
\begin{align}
(X_e,Y_e,A_i,B_i)_{e\in E,i\in I}.
\end{align}

Consider the action of $G_{\mbf{v}}=\prod_{i\in I}GL(V_i)$ on $T^*\text{Rep}_{Q}(\mbf{v},\mbf{w})$ by conjugating $X_e, Y_e$ via left-multiplying $A_i$ and right-multiplying $B_i$. Now we choose the stability conditionb $\bm{\theta}:G_{\mbf{v}}\rightarrow\mbb{C}^*$:
\begin{align*}
\bm{\theta}(\{g_i\}_{i\in I})=\prod_{i\in I}\text{det}(g_i)^{\theta_i},\qquad\theta_i\in\mbb{Z}.
\end{align*}

In this paper we fix the stability condition to be $\bm{\theta}=(-1,\cdots,-1)$, and it has the corresponding stable points:
\begin{align*}
T^*\text{Rep}_{Q}(\mbf{v},\mbf{w})^{s}\subset T^*\text{Rep}_{Q}(\mbf{v},\mbf{w})
\end{align*}
such that for the quadruples $(X_e,Y_e,A_i,B_i)_{e\in E,i\in I}$ there exists no collection of proper subspaces $\{V_i'\subset V_i\}_{i\in I}$ preserved by the maps $X_{e}$ and
$Y_{e}$, and contains $\text{Im}(A_i)$ for all $i\in I$.

The Hamiltonian action of $G_{\mbf{v}}$ on $T^*\text{Rep}_{Q}(\mbf{v},\mbf{w})$ induces the moment map:
\begin{equation*}
\begin{tikzcd}
T^*\text{Rep}_{Q}(\mbf{v},\mbf{w})\arrow[r,"\mu_{\mbf{v},\mbf{w}}"]&\text{Lie}(G_{\mbf{v}})=\bigoplus_{i\in I}\text{Hom}(V_i,V_i)
\end{tikzcd}
\end{equation*}
which can be written as:
\begin{align*}
\mu_{\mbf{v},\mbf{w}}((X_e,Y_e,A_i,B_i)_{e\in E,i\in I})=\sum_{e\in E}(X_{t(e)}Y_{h(e)}-Y_{t(e)}X_{h(e)})+\sum_{i\in I}A_iB_i.
\end{align*}
If we write $\mu^{-1}_{\mbf{v},\mbf{w}}(0)^{s}=\mu^{-1}_{\mbf{v},\mbf{w}}(0)\cap T^*\text{Rep}_{Q}(\mbf{v},\mbf{w})^s_{\mbf{v},\mbf{w}}$, and then there is a geometric quotient:
\begin{align*}
\mc{M}_{Q}(\mbf{v},\mbf{w})=\mu^{-1}_{\mbf{v},\mbf{w}}(0)^{s}/G_{\mbf{v}}
\end{align*}

which is called the \textbf{Nakajima quiver variety} for the quiver $Q$ associated to the dimension vector $\mbf{v}$, $\mbf{w}$. It is a smooth quasi-projective variety of dimension $2[\langle\mbf{v},\mbf{v}\rangle+\mbf{v}\cdot(\mbf{w}-\mbf{v})]$ \cite{Nak98}. Here:
\begin{align}\label{notation-for-inner-product-on-dimension-vector}
\langle\mbf{a},\mbf{b}\rangle:=\sum_{i,j\in I}a_ib_j\#_{ij},\qquad\mbf{a}\cdot\mbf{b}=\sum_{i\in I}a_ib_i.
\end{align}

\subsubsection{Nilpotent quiver variety}
Fix a quiver variety $\mc{M}_{Q}(\mbf{v},\mbf{w})$, we define its nilpotent quiver variety $\mc{L}_{Q}(\mbf{v},\mbf{w})$ as the attracting set of the $\mbb{C}^*$-action over $\mc{M}_{Q}(\mbf{v},\mbf{w})$ given by:
\begin{align*}
z\cdot(X,Y,I,J)=(z^{a}X,z^{a^*}Y,z^aI,z^{a^*}J),a,a^*\in\mbb{Z}_{<0}.
\end{align*}
It is a projective subvariety of $\mc{M}_{Q}(\mbf{v},\mbf{w})$. Alternatively, the variety $\mc{L}_{Q}(\mbf{v},\mbf{w})$ can also be described as:
\begin{align}\label{GIT-description-for-nilpotent-quiver-variety}
\mc{L}_{Q}(\mbf{v},\mbf{w})\cong(\mu^{-1}(0)^{s}_{\mbf{v},\mbf{w}}\cap(\Lambda_{\mbf{v}}\times(\bigoplus_{i\in I}\text{Hom}(W_i,V_i))))/G_{\mbf{v}}.
\end{align}

Let us denote the natural closed inclusion map by $i:\mc{L}_{Q}(\mbf{v},\mbf{w})\hookrightarrow\mc{M}_{Q}(\mbf{v},\mbf{w})$.

\subsubsection{Equivariant K-theory}

On $\mc{M}_{Q}(\mbf{v},\mbf{w})$ and $\mc{L}_{Q}(\mbf{v},\mbf{w})$ there is an algebraic group action
\begin{align}\label{torus-action-given-fixed}
T_{\mbf{w}}=T\times\prod_{i\in I}GL(W_i),\qquad T:=\mbb{C}^*_q\times \prod_{e\in E}\mbb{C}^*_{t_e}
\end{align}
which is written as:
\begin{align*}
(q,t_e,U_i)_{e\in E,i\in I}\cdot(X_e,Y_e,A_i,B_i)_{e\in E,i\in I}=(\frac{X_e}{t_e},\frac{t_eY_e}{q},A_iU_i^{-1},\frac{U_iB_i}{q})_{e\in E,i\in I}.
\end{align*}
Now with respect to the action above, the $T_{\mbf{w}}$-equivariant algebraic $K$-theory groups of Nakajima quiver varieties are modules over the ring
\begin{align*}
K_{T_{\mbf{w}}}(pt)=\mbb{Z}[q^{\pm1},t_{e}^{\pm1}][a_{ik}^{\pm1}]^{\text{Sym}}_{e\in E,i\in I,1\leq k\leq w_i}
\end{align*}

and here $a_{ik}$ stands for the equivariant parametres of the maximal torus in $\prod_{i\in I}GL(W_i)$.

It has been proved in \cite{KN18} that the equivariant $K$-theory of Nakajima quiver varieties $K_{T_{\mbf{w}}}(\mc{M}_{Q}(\mbf{v},\mbf{w}))$ is generated by the tautological bundles $\mc{V}_i$ and $K_{T_{\mbf{w}}}(pt)$.

The following important theorem has been proved in \cite{Nak01}:
\begin{thm}[See Theorem 7.3.5 in \cite{Nak01}]\label{perfect-pairing:label}
$K_{A\times G_{\mbf{w}}}(\mc{M}_{Q}(\mbf{v},\mbf{w}))$ is a free $K_{A\times G_{\mbf{w}}}(pt)$-module of finite rank for arbitrary $\mbb{C}_q^*\subset A\subset T$. Moreover, there is a perfect pairing:
\begin{align*}
K_{A\times G_{\mbf{w}}}(\mc{M}_{Q}(\mbf{v},\mbf{w}))\otimes K_{A\times G_{\mbf{w}}}(\mc{L}_{Q}(\mbf{v},\mbf{w}))\rightarrow K_{A\times G_{\mbf{w}}}(pt),\qquad(\mc{F},\mc{G})\mapsto p_{*}(\mc{F}\otimes^{\mbb{L}}_{\mc{M}_{Q}(\mbf{v},\mbf{w})}\mc{G})
\end{align*}
of $K_{T_{\mbf{w}}}(pt)$-modules. Here $p$ is the canonical map from $\mc{L}_{Q}(\mbf{v},\mbf{w})$ to a point. $i:\mc{L}_{Q}(\mbf{v},\mbf{w})\hookrightarrow\mc{M}_{Q}(\mbf{v},\mbf{w})$ is the natural closed embedding.
\end{thm}
\begin{proof}
The proof is basically the same as the proof in Section seven of \cite{Nak01} that one just need to note that the hyperkahler metric over $\mc{M}_{Q}(\mbf{v},\mbf{w})$ is also invariant under $\prod_{e\in E}\mbb{C}^*$-action.
\end{proof}

On the other hand, the theorem means that we have the following isomorphism of $K_{A\times G_{\mbf{w}}}(pt)$-modules:
\begin{equation*}
\begin{aligned}
&\text{Hom}_{K_{A\times G_{\mbf{w}}}(pt)}(K_{A\times G_{\mbf{w}}}(\mc{M}_{Q}(\mbf{v}_1,\mbf{w})),K_{A\times G_{\mbf{w}}}(\mc{M}_{Q}(\mbf{v}_2,\mbf{w})))\\
\cong&\text{Hom}_{K_{A\times G_{\mbf{w}}}(pt)}(K_{A\times G_{\mbf{w}}}(\mc{L}_{Q}(\mbf{v}_2,\mbf{w})),K_{A\times G_{\mbf{w}}}(\mc{L}_{Q}(\mbf{v}_1,\mbf{w}))).
\end{aligned}
\end{equation*}

This perfect pairing allows us to treat $K_{A\times G_{\mbf{w}}}(\mc{L}_{Q}(\mbf{v},\mbf{w}))$ as the $K_{A\times G_{\mbf{w}}}(pt)$-dual of $K_{A\times G_{\mbf{w}}}(\mc{M}_{Q}(\mbf{v},\mbf{w}))$:
$$K_{A\times G_{\mbf{w}}}(\mc{L}_{Q}(\mbf{v},\mbf{w}))\cong K_{A\times G_{\mbf{w}}}(M_{Q}(\mbf{v},\mbf{w}))^{\vee}.$$

In the following context, unless we have mentioned, we will always use $A$ to be the full torus $T$.

\subsubsection{Non-localised action from $\mc{A}^{+,\mbb{Z}}_{Q}$}
The action of $\mc{A}^{+,\mbb{Z}}_{Q}$ on $K_{T_{\mbf{w}}}(\mc{M}_{Q}(\mbf{w}))$ can be described by the following diagram:
\begin{equation}\label{Hecke-correspondence-preprojective-KHA}
\begin{tikzcd}
&\mc{M}_{Q}(\mbf{v},\mbf{v}+\mbf{n},\mbf{w})\arrow[dl,"p"]\arrow[d,"\pi_{+}"]\arrow[dr,"\pi_{-}"]&\\
\mc{Y}_{\mbf{n}}&\mc{M}_{Q}(\mbf{v}+\mbf{n},\mbf{w})&\mc{M}_{Q}(\mbf{v},\mbf{w})
\end{tikzcd}
\end{equation}
and here $\mc{M}_{Q}(\mbf{v},\mbf{v}+\mbf{n},\mbf{w})$ is the moduli stack parametrising the short exact sequences
\begin{align}\label{short-exact-sequence-for-Hecke-moduli}
0\rightarrow K_{\bullet}\rightarrow V_{\bullet}^{+}\rightarrow V_{\bullet}^{-}\rightarrow0
\end{align}
where $V_{\bullet}^+$ and $V_{\bullet}^{-}$ are stable quiver representations with dimension vector $\mbf{v}+\mbf{n},\mbf{v}$ respectively and of the framing vector $\mbf{w}$. The diagram \ref{Hecke-correspondence-preprojective-KHA} gives the following map:
\begin{equation*}
\begin{aligned}
&\mc{A}^{+,\mbb{Z}}_{Q,\mbf{n}}\otimes K_{T_{\mbf{w}}}(\mc{M}_{Q}(\mbf{v},\mbf{w}))\rightarrow K_{T_{\mbf{w}}}(\mc{M}_{Q}(\mbf{v}+\mbf{n},\mbf{w}))\\
&\alpha\otimes\beta\mapsto\pi_{+*}(\text{sdet}[\sum_{e=ji\in E}\frac{t_e\mc{V}_j^+}{q\mc{K}_i}-\sum_{i\in I}\frac{\mc{V}_i^+}{q\mc{K}_i}]\cdot(p\times\pi_-)^{!}(\alpha\boxtimes\beta)).
\end{aligned}
\end{equation*}

It has been proved in \cite{N23} that this gives the action of $\mc{A}^{+,\mbb{Z}}_{Q}$ on $K_{T_{\mbf{w}}}(\mc{M}_{Q}(\mbf{w}))$. 

On the other hand, since $\pi_{-}$ is not proper, the action
\begin{equation*}
\begin{aligned}
&(\mc{A}^{+,\mbb{Z}}_{Q,\mbf{n}})^{op}\otimes K_{T_{\mbf{w}}}(\mc{M}_{Q}(\mbf{v}+\mbf{n},\mbf{w}))\rightarrow K_{T_{\mbf{w}}}(\mc{M}_{Q}(\mbf{v},\mbf{w}))\\
&\alpha\otimes\beta\mapsto\pi_{-*}(\text{sdet}[\sum_{e=ji\in E}\frac{q\mc{V}_j^-}{t_e\mc{K}_i}-\sum_{i\in I}\frac{q\mc{V}_i^-}{\mc{K}_i}+\sum_{i\in I}\frac{W_i}{\mc{K}_i}]\cdot(p\times\pi_+)^{!}(\alpha\boxtimes\beta))
\end{aligned}
\end{equation*}
can only be defined after localisation to the $T$-fixed point part.

\subsubsection{Localised action from $\mc{A}^{ext}_{Q}$}
\label{subsub:localised_action_from_mc_a}
We denote the localized $K$-theory groups as:
\begin{equation*}
\begin{aligned}
K(\mbf{v},\mbf{w})=K_{T_{\mbf{w}}}(\mc{M}_Q(\mbf{v},\mbf{w}))\otimes_{\mbb{Z}[q^{\pm1},t_{e}^{\pm1}][a_{ik}^{\pm1}]^{\text{Sym}}_{e\in E,i\in I,1\leq k\leq w_i}}\mbb{Q}(q,t_e)(a_{ik})^{\text{Sym}}_{e\in E,i\in I,1\leq k\leq w_i}.
\end{aligned}
\end{equation*}

Thus we can consider the direct sum:
\begin{align*}
K(\mbf{w})=\bigoplus_{\mbf{v}\in\mbb{N}^I}K(\mbf{v},\mbf{w}).
\end{align*}

Here we give the geometric action of $\mc{A}_{Q}$ over $K(\mbf{w})$. We define the stack $\mc{M}_{Q}(\mbf{v},\mbf{v}+\mbf{e}_i,\mbf{w})$ parametrising the short exact sequences:
\begin{align*}
0\rightarrow K_{\bullet}\rightarrow V_{\bullet}^+\rightarrow V_{\bullet}^-\rightarrow0,\qquad K_{\bullet}\in\mc{Y}_{\mbf{e}_i}, V_{\bullet}^{\pm}\in\mc{M}_{Q}(\mbf{v}(+\mbf{e}_i),\mbf{w}).
\end{align*}
This space gives the natural projection map as follows:

\begin{equation}\label{corresponding-for-basics}
\begin{tikzcd}
&\mc{M}_{Q}(\mbf{v},\mbf{v}+\mbf{e}_i,\mbf{w})\arrow[dl,"p"]\arrow[d,"\pi_+"]\arrow[dr,"\pi_-"]&\\
\mc{Y}_{\mbf{e}_i}&\mc{M}_{Q}(\mbf{v}+\mbf{e}_i,\mbf{w})&\mc{M}_{Q}(\mbf{v},\mbf{w}).
\end{tikzcd}
\end{equation}

Moreover, it was proved in \cite{Nak01}\cite{N23} that $\mc{M}_{Q}(\mbf{v},\mbf{v}+\mbf{e}_i,\mbf{w})$ is a quasiprojective scheme. Using the map in the above diagram \ref{corresponding-for-basics}, we can define the operator
\begin{align*}
e_{i,d}=\pi_{+*}(\mc{L}_i^d\cdot\text{sdet}[\sum_{e=ji\in E}\frac{t_e\mc{V}_j^+}{q\mc{L}_i}-\frac{\mc{V}_i^+}{q\mc{L}_i}]\cdot\pi_{-}^*)
\end{align*}
\begin{align*}
f_{i,d}=\pi_{-*}(\mc{L}_i^d\cdot\text{sdet}[\sum_{e=ij\in E}\frac{q\mc{V}_j^-}{t_e\mc{L}_i}-\frac{q\mc{V}_i^-}{\mc{L}_i}+\frac{W_i}{\mc{L}_i}]\cdot\pi_+^*).
\end{align*}

This action is well-defined over the localised $K$-theory $K(\mbf{w}):=\bigoplus_{\mbf{v}\in\mbb{N}^I}K_{T_{\mbf{w}}}(\mc{M}_{Q}(\mbf{v},\mbf{w}))$. Also the action of $\mc{A}^{ext,0}_{Q}$ over $K(\mbf{w})$ is given by the multiplication of the tautological classes
\begin{align*}
a_{i,d}\mapsto p_{d}(\mc{V}_i(1-q^{-1}))\otimes(-)
\end{align*}

\begin{align*}
b_{i,d}\mapsto p_{d}(\mc{W}_i(1-q^{-1}))\otimes(-)
\end{align*}
and here $p_{d}$ is the power sum for the Chern roots, i.e.
\begin{align*}
p_d(x_1,\cdots,x_n):=x_1^d+\cdots+x_n^d.
\end{align*}

Since we know that $\mc{A}_{Q}^+$ is generated by $e_{i,d}$ after being localised over the equivariant parametres $\{q,t_e\}_{e\in E}$, one can write for arbitrary $F\in\mc{A}_{Q}^+$ as the polynomial over $e_{i,d}$ as follows:
\begin{align*}
F=\sum_{\mbf{i},\mbf{d}}a_{\mbf{i},\mbf{d}}e_{i_1,d_1}*\cdots*e_{i_n,d_n}:K(\mbf{w})\rightarrow K(\mbf{w})
\end{align*}
with $a_{\mbf{i},\mbf{d}}\in\mbb{Q}(q,t_e)_{e\in E}$ such that
\begin{align*}
a_{\mbf{i},\mbf{d}}e_{i_1,d_1}*\cdots*e_{i_n,d_n}:K(\mbf{v},\mbf{w})\rightarrow K(\mbf{v}+\mbf{e}_{i_1}+\cdots+\mbf{e}_{i_n},\mbf{w})
\end{align*}
can be written as a chain of the correspondences with a rational coefficients $a_{\mbf{i},\mbf{d}}\in K_{T}(pt)_{loc}$.

Also for each $i\in I$, we can consider the tautological bundle $V_i$ of rank $v_i$, whose fibre over a point is the vector space $V_i$. And we can formally write down:
\begin{align*}
[\mc{V}_i]=x_{i1}+\cdots+x_{iv_i}\in K(\mbf{v},\mbf{w})
\end{align*}
and here the symbols $x_{ia}$ are the Chern roots for any symmetric Laurent polynomial over $x_{ia}$. We abbreviate:
\begin{align*}
\mbf{X}_{\mbf{v}}=\{\cdots,x_{i1},\cdots,x_{iv_i},\cdots\}.
\end{align*}

Now from the Kirwan surjectivity \cite{MN18} we know that the Laurent polynomial $p(\mbf{X}_{\mbf{v}})$ and $K_{T_{\mbf{w}}}(pt)$ generate $K_{T_{\mbf{w}}}(\mc{M}_{Q}(\mbf{v},\mbf{w}))$ for any $\mbf{v},\mbf{w}\in\mbb{N}^{I}$.

\subsubsection{Shuffle formula for the action}
In this subsubsection we use the result in Section \ref{sec:slope_filtration_and_shuffle_realisations_of_k_theoretic_hall_algebras} for the shuffle realisation of the KHA $\mc{A}^{ext}_{Q}$.

It turns out that after the localisation, the geometric action is of $\mc{A}^{ext}_{Q}$ on $K(\mbf{w})$ can be written by the following \cite{N22}: Given $F\in\mc{A}_{Q,\mbf{n}}^+$ and $G\in\mc{A}_{Q,\mbf{n}}^{-}$, we have that:
\begin{align}\label{positive-action}
F\cdot p(\mbf{X}_{\mbf{v}})=\frac{1}{\mbf{n}!}\int^{+}\frac{F(\mbf{Z}_{\mbf{n}})}{\tilde{\zeta}(\frac{\mbf{Z}_{\mbf{n}}}{\mbf{Z}_{\mbf{n}}})}p(\mbf{X}_{\mbf{v}+\mbf{n}}-\mbf{Z}_{\mbf{n}})\tilde{\zeta}(\frac{\mbf{Z}_{\mbf{n}}}{\mbf{X}_{\mbf{v}+\mbf{n}}})\wedge^*(\frac{\mbf{Z}_{\mbf{n}}q}{\mbf{W}})
\end{align}

\begin{align}\label{negative-action}
G\cdot p(\mbf{X}_{\mbf{v}})=\frac{1}{\mbf{n}!}\int^{-}\frac{G(\mbf{Z}_{\mbf{n}})}{\tilde{\zeta}(\frac{\mbf{Z}_{\mbf{n}}}{\mbf{Z}_{\mbf{n}}})}p(\mbf{X}_{\mbf{v}-\mbf{n}}+\mbf{Z}_{\mbf{n}})\tilde{\zeta}(\frac{\mbf{X}_{\mbf{v}-\mbf{n}}}{\mbf{Z}_{\mbf{n}}})^{-1}\wedge^*(\frac{\mbf{Z}_{\mbf{n}}}{\mbf{W}})
\end{align}

\begin{align}\label{Cartan-action}
h_{i}^{\pm}(z_{i1})=\frac{\tilde{\zeta}(\frac{\mbf{Z}_{e_i}}{\mbf{X}_{\mbf{v}}})}{\tilde{\zeta}(\frac{X_{\mbf{v}}}{\mbf{Z}_{e_i}})}\cdot\frac{\wedge^*(\frac{z_{i1}q}{W_i})}{\wedge^*(\frac{z_{i1}}{W_i})}
\end{align}
and here the integral sign $\int^{\pm}$ has been interpreted in Section 4.17 of \cite{N22}.
The integral formula is well-defined on the localised equivariant $K$-theory $K(\mbf{w})$ of quiver varieties. 

The above notations stands for the following:
\begin{align*}
F(\mbf{Z}_{\mbf{n}})=F(\cdots,z_{i1},\cdots,z_{in_i},\cdots)_{i\in I}\in\mc{A}^{\pm}_{Q}
\end{align*}
\begin{align*}
\tilde{\zeta}_{ij}(x)=\frac{\zeta_{ij}(x)}{(1-xq^{-1})^{\delta_{ij}}(1-q^{-1}x^{-1})^{\delta_{ij}}}
\end{align*}
\begin{align*}
\tilde{\zeta}(\frac{\mbf{Z}_{\mbf{n}}}{\mbf{X}_{\mbf{v}}})=\prod_{1\leq a\leq n_i}^{i\in I}\prod_{1\leq b\leq v_j}^{j\in I}\tilde{\zeta}(\frac{z_{ia}}{x_{jb}}),\qquad\tilde{\zeta}(\frac{\mbf{X}_{\mbf{v}}}{\mbf{Z}_{\mbf{n}}})=\prod_{1\leq a\leq n_i}^{i\in I}\prod_{1\leq b\leq v_j}^{j\in I}\tilde{\zeta}(\frac{x_{jb}}{z_{ia}})
\end{align*}
\begin{align*}
\tilde{\zeta}(\frac{\mbf{Z}_{\mbf{n}}}{\mbf{Z}_{\mbf{n}}})=\prod_{\substack{1\leq a\leq n_i,1\leq b\leq n_j\\(i,a)\neq(j,b)}}\frac{\zeta_{ij}(\frac{z_{ia}}{z_{jb}})}{(1-\frac{z_{ia}}{qz_{jb}})^{\delta^i_j}}
\end{align*}
and here $\zeta_{ij}(x)$ is defined in \ref{zeta-function}.
\begin{align*}
\wedge^*(\frac{\mbf{Z}_{\mbf{n}}q^{0,1}}{\mbf{W}})=\prod^{i\in I}_{1\leq a\leq n_i}\wedge^*(\frac{z_{ia}q^{0,1}}{W_{i}}).
\end{align*}

The notation of the integral $\int^{\pm}$ represents the following integral type:
\begin{align*}
\int^{+}T(\cdots,z_{ia},\cdots)=\sum^{\text{functions}}_{\sigma:\{(i,a)\}\rightarrow\{\pm1\}}\int^{|q/t_e|^{\pm1},|t_e|^{\pm1}>1}_{|z_{ia}|=r^{\sigma(i,a)}}T(\cdots,z_{ia},\cdots)\prod_{(i,a)}\frac{\sigma(i,a)dz_{ia}}{2\pi\sqrt{-1}z_{ia}}
\end{align*}

\begin{align*}
\int^{-}T(\cdots,z_{ia},\cdots)=\sum^{\text{functions}}_{\sigma:\{(i,a)\}\rightarrow\{\pm1\}}\int^{|q/t_e|^{\pm1},|t_e|^{\pm1}<1}_{|z_{ia}|=r^{\sigma(i,a)}}T(\cdots,z_{ia},\cdots)\prod_{(i,a)}\frac{\sigma(i,a)dz_{ia}}{2\pi\sqrt{-1}z_{ia}}.
\end{align*}

For example, if $F=e_{i_1,d_1}*\cdots*e_{i_n,d_n}$ or $G=f_{i_1,d_1}*\cdots*f_{i_n,d_n}$, we have that:
\begin{equation*}
\int_{\{0, \infty\} \succ z_1 \succ \cdots \succ z_n}\frac{z_1^{d_1}\cdots z_{n}^{d_n}}{\prod_{1\leq a<b\leq n}\tilde{\zeta}_{i_bi_a}(z_b/z_a)}p(\mbf{X}_{\mbf{v}+\mbf{n}}-\mbf{Z}_{\mbf{n}})\tilde{\zeta}(\frac{\mbf{Z}_{\mbf{n}}}{\mbf{X}_{\mbf{v}+\mbf{n}}})\wedge^*(\frac{\mbf{Z}_{\mbf{n}}q}{\mbf{W}})\prod_{a=1}^{n}\frac{dz_a}{2\pi\sqrt{-1}z_{a}}
\end{equation*}

\begin{equation*}
\int_{\{0, \infty\} \succ z_1 \succ \cdots \succ z_n}\frac{z_1^{d_1}\cdots z_{n}^{d_n}}{\prod_{1\leq a<b\leq n}\tilde{\zeta}_{i_bi_a}(z_b/z_a)}p(\mbf{X}_{\mbf{v}-\mbf{n}}+\mbf{Z}_{\mbf{n}})\tilde{\zeta}(\frac{\mbf{X}_{\mbf{v}-\mbf{n}}}{\mbf{Z}_{\mbf{n}}})^{-1}\wedge^*(\frac{\mbf{Z}_{\mbf{n}}}{\mbf{W}})^{-1}\prod_{a=1}^{n}\frac{dz_a}{2\pi\sqrt{-1}z_{a}}.
\end{equation*}

It has been proved in \cite{N22} that the above integral formula gives a well-defined algebra action of $\mc{A}^{ext}_{Q}$ on $K(\mbf{w})$.

The following two theorems has been proved in \cite{N23} and it will be useful in this paper:
\begin{thm}[See Proposition 2.18 of \cite{N23}]\label{surjectivity-noetherian-lemma-localised:theorem}
The action described above in \ref{positive-action} gives a surjective map of $K_{T_{\mbf{w}}}(pt)_{loc}$-modules:
\begin{align*}
\mc{A}_{Q,\mbf{v}}^+\otimes\mbb{F}_{\mbf{w}}\twoheadrightarrow K_{T_{\mbf{w}}}(\mc{M}_{Q}(\mbf{v},\mbf{w}))_{loc},\qquad\mbb{F}_{\mbf{w}}:=K_{T_{\mbf{w}}}(pt)_{loc}.
\end{align*}

Moreover, associated with the perfect pairing \ref{perfect-pairing:label}, every covector in $\text{Hom}_{K_{T_{\mbf{w}}}(pt)_{loc}}(K_{T_{\mbf{w}}}(\mc{M}_{Q}(\mbf{v},\mbf{w}))_{loc},K_{T_{\mbf{w}}}(pt)_{loc})$ is generated by $\mc{A}^{-}_{Q,\mbf{v}}$ via the action \ref{negative-action}.
\end{thm}

\begin{thm}[See Proposition 2.17 of \cite{N23}]\label{injectivity-of-localised-double-KHA:theorem}
There is an injective map of $K_{T}(pt)_{loc}=\mbb{Q}(q,t_e)_{e\in E}$-algebras
\begin{align*}
\mc{A}_{Q}^{ext}\hookrightarrow\prod_{\mbf{w}}\text{End}(K(\mbf{w})).
\end{align*}
\end{thm}

\subsubsection{Integral form of the geometric action}
Since the map $\pi_{+}$ is proper and $\pi_{-}\times p$ is l.c.i, the integral KHA $\mc{A}^{+,\mbb{Z}}_{Q}$ has the natural action over the integral equivariant $K$-theory $K_{T_{\mbf{w}}}(\mc{M}_{Q}(\mbf{w}))$. While for the negative half $\mc{A}^{-}_{Q}$, it is constructed in the following way:

First we can see that there is an algebra action of $\mc{A}^{+,nilp,\mbb{Z}}_{Q}$ over $K_{T_{\mbf{w}}}(\mc{L}_{Q}(\mbf{w}))$ by the following:
\begin{equation*}
\begin{tikzcd}
&\mc{L}_{Q}(\mbf{v},\mbf{v}+\mbf{n},\mbf{w})\arrow[dl,"p_{\mc{L}}"]\arrow[d,"\pi_{+,\mc{L}}"]\arrow[dr,"\pi_{-,\mc{L}}"]&\\
\Lambda_{\mbf{n}}^Q&\mc{L}_{Q}(\mbf{v}+\mbf{n},\mbf{w})&\mc{L}_{Q}(\mbf{v},\mbf{w})
\end{tikzcd}
\end{equation*}
and here the moduli stack $\mc{L}_{Q}(\mbf{v},\mbf{v}+\mbf{n},\mbf{w})$ is similar to the definition of the moduli stack $\mc{M}_{Q}(\mbf{v},\mbf{v}+\mbf{n},\mbf{w})$, it is given by the short exact sequence \ref{short-exact-sequence-for-Hecke-moduli} with the nilpotent conditions on $K_{\bullet}$, $V_{\bullet}^+$ and $V_{\bullet}^-$.

The action is thus written as:
\begin{equation}\label{positive-action-nilpotent}
\begin{aligned}
&\mc{A}^{+,nilp,\mbb{Z}}_{Q,\mbf{n}}\otimes K_{T_{\mbf{w}}}(\mc{L}_{Q}(\mbf{v},\mbf{w}))\rightarrow K_{T_{\mbf{w}}}(\mc{L}_{Q}(\mbf{v}+\mbf{n},\mbf{w}))\\
&\alpha\otimes\beta\mapsto\pi_{+,\mc{L},*}(\text{sdet}[\sum_{e=ji\in E}\frac{q\mc{V}_j^-}{t_e\mc{K}_i}-\sum_{i\in I}\frac{q\mc{V}_i^-}{\mc{K}_i}+\sum_{i\in I}\frac{W_i}{\mc{K}_i}]\cdot(p_{\mc{L}}\times\pi_{-,\mc{L}})^!(\alpha\otimes\beta)).
\end{aligned}
\end{equation}

On the other hand, since $\mc{L}_{Q}(\mbf{v},\mbf{v}+\mbf{n},\mbf{w})$ is a closed subvariety in $\mc{L}_{Q}(\mbf{v},\mbf{w})\times\mc{L}_{Q}(\mbf{v}+\mbf{n},\mbf{w})$, and $\mc{L}_{Q}(\mbf{v},\mbf{w})\times\mc{L}_{Q}(\mbf{v}+\mbf{n},\mbf{w})$ is a projective variety, the maps $\pi_{\pm,\mc{L}}$ are both proper maps. Therefore the inverse side of the action:
\begin{equation}\label{negative-action-nilpotent}
\begin{aligned}
&(\mc{A}^{+,nilp,\mbb{Z}}_{Q,\mbf{n}})^{op}\otimes K_{T_{\mbf{w}}}(\mc{L}_{Q}(\mbf{v}+\mbf{n},\mbf{w}))\rightarrow K_{T_{\mbf{w}}}(\mc{L}_{Q}(\mbf{v},\mbf{w}))\\
&\alpha\otimes\beta\mapsto\pi_{-,\mc{L},*}(\text{sdet}[\sum_{e=ji\in E}\frac{t_e\mc{V}_j^+}{\mc{K}_i}-\sum_{i\in I}\frac{\mc{V}_i^+}{q\mc{K}_i}]\cdot(p_{\mc{L}}\times\pi_{+,\mc{L}})^{!}(\alpha\boxtimes\beta))
\end{aligned}
\end{equation}
is well-defined before the localisation. In this case we denote the corresponding algebra as $\mc{A}^{-,nilp,\mbb{Z}}_{Q}:=(\mc{A}^{+,nilp,\mbb{Z}}_{Q})^{op}$. Using the isomorphism:
\begin{align}\label{transpose-isomorphism}
\text{Hom}_{K_{T_{\mbf{w}}(pt)}}(K_{T_{\mbf{w}}}(\mc{L}_{Q}(\mbf{v},\mbf{w})),K_{T_{\mbf{w}}}(\mc{L}_{Q}(\mbf{v}+\mbf{n},\mbf{w})))\cong\text{Hom}_{K_{T_{\mbf{w}}(pt)}}(K_{T_{\mbf{w}}}(\mc{M}_{Q}(\mbf{v}+\mbf{n},\mbf{w})),K_{T_{\mbf{w}}}(\mc{M}_{Q}(\mbf{v},\mbf{w})))
\end{align}
induced by the perfect pairing \ref{perfect-pairing:label}, one can see that the action of $\mc{A}^{\pm,nilp,\mbb{Z}}_{Q}$ coincides with the one on $(\mc{A}^{\mp,\mbb{Z}}_{Q})$. 

The following proposition will be the key to the construction of the main theorem, which is the analog of the Proposition \ref{surjectivity-noetherian-lemma-localised:theorem}:

\begin{prop}\label{Integral-nilpotent-noetherian-lemam:label}
There is a surjective map of $K_{T_{\mbf{w}}}(pt)$-modules
\begin{equation*}
\mc{A}^{+,nilp,\mbb{Z}}_{Q,\mbf{v}}\otimes K_{T_{\mbf{w}}}(pt)\twoheadrightarrow K_{T_{\mbf{w}}}(\mc{L}_Q(\mbf{v},\mbf{w}))
\end{equation*}
\end{prop}
\begin{proof}
Recall that $\mc{L}_{Q}(\mbf{v},\mbf{w})$ can also be written as in \ref{GIT-description-for-nilpotent-quiver-variety}, 
and the map of the correspondence in the proposition can be written as:
\begin{equation*}
\begin{tikzcd}
\mc{L}_{Q}(\mbf{0},\mbf{v},\mbf{w})\arrow[r,"\pi"]\arrow[d,"p"]&\mc{L}_{Q}(\mbf{v},\mbf{w})\\
\Lambda_{\mbf{v}}
\end{tikzcd}
\end{equation*}
while by definition $\mc{L}_{Q}(\mbf{0},\mbf{v},\mbf{w})$ is isomorphic to $\mc{M}_{Q}(\mbf{0},\mbf{v},\mbf{w})\cap\mc{L}_{Q}(\mbf{v},\mbf{w})$, i.e. stable nilpotent representations with the condition $B_i=0\in\bigoplus_{i\in I}\text{Hom}(W_i,V_i)$. While this means that $\mc{L}_{Q}(\mbf{0},\mbf{v},\mbf{w})$ is the same as $\mc{L}_{Q}(\mbf{v},\mbf{w})$. Thus it only remains to prove that $p^*$ is surjective. 

Using the proof in \cite{N23}, one can have the following commutative diagram:
\begin{equation}\label{nilpotent-important-diagram}
\begin{tikzcd}
\mc{L}_{Q}(\mbf{v},\mbf{w})\arrow[r,hook,"j"]\arrow[dr,"p"]&\text{Tot}_{\Lambda_{\mbf{v}}}(\bigoplus_{i\in I}\text{Hom}(W_i,V_i))\arrow[d,"\tau"]\\
&\Lambda_{\mbf{v}}
\end{tikzcd}
\end{equation}
and here $j$ is an open embedding and $\tau$ is an affine fibration. Since both $\tau^*$ and $j^*$ are surjective, we can conclude that $p^*$ is surjective. Thus the proposition is proved.

\end{proof}

\section{Slope filtration and Shuffle realisations of K-theoretic Hall algebras}\label{sec:slope_filtration_and_shuffle_realisations_of_k_theoretic_hall_algebras}
One of the important tool for us to make full use of the $K$-theoretic Hall algebra is to give a shuffle realisation of the KHA. This will transfer geometric conditions into the combinatorics of the color-symmetric Laurent polynomials. 

\label{section:slope_filtration_and_shuffle_realisations_of_k_theoretic_hall_algebras}
\subsection{Slope filtration for derived categories of quiver moduli}
At the beginning, we give a geometric introduction of what is a slope filtration for the $K$-theoretic Hall algebra which will be introduced in Section \ref{sub:slope_subalgebras_and_factorizations_of_r_matrices}.

Now we focus on the derived categories $D^b(\text{Coh}_{T}(\mc{Y}_{\mbf{n}}))$ and $D^b(\text{Coh}_{T}(\Lambda_{\mbf{n}}))$ of $T$-equivariant coherent sheaves over the quiver moduli $\mc{Y}_{\mbf{n}}$ and $\Lambda_{\mbf{n}}$ as defined in \ref{preprojective-stack} and \ref{nilpotent-stack}.

Recall that the diagonal one-dimensional torus $z\cdot\text{Id}\subset\prod_{i\in I}GL(V_i)$ acts trivially on $\mu^{-1}_{\mbf{n}}(0)$ and $\mu^{-1}_{\mbf{n}}(0)\cap E^{0}_{\mbf{v}}$. We define $D^b(\text{Coh}_{T}(\mc{Y}_{\mbf{n}}))_{k}$ and $D^b(\text{Coh}_{T}(\Lambda_{\mbf{n}}))_{k}$ be the category of complexes on which $z\cdot\text{Id}$ acts with weight $k\in\mbb{Z}$. We have the following orthogonal decomposition:
\begin{equation*}
D^b(\text{Coh}_{T}(\mc{Y}_{\mbf{n}}))=\bigoplus_{k\in\mbb{Z}}D^b(\text{Coh}_{T}(\mc{Y}_{\mbf{n}}))_{k},\qquad D^b(\text{Coh}_{T}(\Lambda_{\mbf{n}}))=\bigoplus_{k\in\mbb{Z}}D^b(\text{Coh}_{T}(\Lambda_{\mbf{n}}))_{k}.
\end{equation*}

Thus this induce the orthogonal decomposition over the equivariant $K$-theory:
\begin{align*}
K_{T}(\mc{Y}_{\mbf{n}})=\bigoplus_{k\in\mbb{Z}}K_{T}(\mc{Y}_{\mbf{n}})_{k},\qquad K_{T}(\Lambda^Q_{\mbf{n}})=\bigoplus_{k\in\mbb{Z}}K_{T}(\Lambda^Q_{\mbf{n}})_{k}.
\end{align*}

This induces the horizontal degree decomposition for the preprojective and nilpotent $K$-theoretic Hall algebras:
\begin{equation*}
\begin{aligned}
&\mc{A}^{+,\mbb{Z}}_{Q}=\bigoplus_{(k,\mbf{n})\in\mbb{Z}\times\mbb{N}^I}\mc{A}^{+,\mbb{Z}}_{Q,k,\mbf{n}},\qquad\mc{A}^{+,\mbb{Z}}_{Q,k,\mbf{n}}:=K_{T}(\mc{Y}_{\mbf{n}})_{k}\\
&\mc{A}^{+,nilp,\mbb{Z}}_{Q}=\bigoplus_{(k,\mbf{n})\in\mbb{Z}\times\mbb{N}^I}\mc{A}^{+,nilp,\mbb{Z}}_{Q,k,\mbf{n}},\qquad\mc{A}^{+,nilp,\mbb{Z}}_{Q,k,\mbf{n}}:=K_{T}(\Lambda_{\mbf{n}})_{k}.
\end{aligned}
\end{equation*}

Now we can define the slope filtration for the $K$-theoretic Hall algebras in the following way.

We first introduce some notations for the torus action. We denote $\sigma_{\mbf{k}}:[\text{pt}/\mbb{C}^*]\times\mc{Y}_{\mbf{n}}\rightarrow\mc{Y}_{\mbf{n}}$ as the cocharacter acting as $\text{diag}(\underbrace{z,\cdots,z}_{k_i copies},1,\cdots,1)\in GL(V_i)$ for the $i$-th node of the group. Given $\mc{F}\in D^b(\text{Coh}_{T}(\mc{Y}_{\mbf{n}}))$ or in $D^b(\text{Coh}(\Lambda_{\mbf{n}}))$, we define $\text{deg}_{\mbf{k}}(\mc{F})$ as the degree of the complex $\mc{F}$ under the action of the torus $\sigma_{\mbf{k}}$.

\begin{defn}\label{Slope-filtration:Definition}
Given a rational vector $\mbf{m}\in\mbb{Q}^I$, we define the slope $\leq\mbf{m}$-subspace $\mc{A}^{+,\mbb{Z}}_{Q,\leq\mbf{m}}$ (resp. $\mc{A}^{+,nilp,\mbb{Z}}_{Q,\leq\mbf{m}}$) of $\mc{A}^{+,\mbb{Z}}_{Q}$ (resp. $\mc{A}^{+,nilp,\mbb{Z}}_{Q}$) as the subspace generated by the elements $\mc{F}\in D^b(\text{Coh}_{T}(\mc{Y}_{\mbf{n}}))$ (resp. $D^b(\text{Coh}_{T}(\Lambda_{\mbf{n}}))$) such that:
\begin{align*}
\text{deg}_{\mbf{k}}(\mc{F})\leq\mbf{m}\cdot\mbf{k}+\langle\mbf{k},\mbf{n}-\mbf{k}\rangle
\end{align*}
and here $-\cdot-$ and $\langle-,-\rangle$ are defined as \ref{notation-for-inner-product-on-dimension-vector}.
\end{defn}

\begin{defn}\label{Slope-subalgebra-definition-derived
:label}
Given a rational vector $\mbf{m}\in\mbb{Q}^I$, we define the \textbf{slope subalgebra} $\mc{B}^{+,\mbb{Z}}_{\mbf{m}}$ (resp. $\mc{B}_{\mbf{m}}^{+,nilp,\mbb{Z}}$) as the subspace generated by the elements $\mc{F}\in D^b(\text{Coh}_{T}(\mc{Y}_{\mbf{n}}))$ (resp. $D^b(\text{Coh}_{T}(\Lambda_{\mbf{n}}))$) such that:
\begin{align*}
\mbf{m}\cdot\mbf{k}-\langle\mbf{k},\mbf{n}-\mbf{k}\rangle\leq\text{deg}_{\mbf{k}}(\mc{F})\leq\mbf{m}\cdot\mbf{k}+\langle\mbf{k},\mbf{n}-\mbf{k}\rangle.
\end{align*}

\end{defn}

\textbf{Remark.} Note that this definition coincides with the definition of the quasi-BPS categories given in Definition 2.30 in \cite{PT25}. For the nilpotent case, one just need to use the matrix factorization category of tripled quivers $(\tilde{Q},\tilde{W})$ with tripled potential with the nilpotent support.

Later we will use the shuffle algebra model to describe these subalgebras.

\subsection{Shuffle realisation of the KHA}
One good algebraic model to described the $K$-theoretic Hall algebra is given by the shuffle algebra realisation. i.e. the space of colored-symmetric Laurent polynomials. For the details of the construction, one can refer to \cite{N20}\cite{N21}\cite{N22}

Note that we have the following chain of closed embedding of quotient stacks
\begin{align*}
[\mu^{-1}_{\mbf{v}}(0)\cap E^{0}_{\mbf{v}}/G_{\mbf{v}}]\hookrightarrow[\mu^{-1}_{\mbf{v}}(0)/G_{\mbf{v}}]\hookrightarrow[\bigoplus_{ij\in E}\text{Hom}(V_i,V_j)/\prod_{i\in I}GL(V_i)]
\end{align*}
these closed embeddings induce the algebra morphism:
\begin{align*}
\mc{A}^{+,nilp,\mbb{Z}}_{Q}\rightarrow\mc{A}^{+,\mbb{Z}}_{Q}\rightarrow(\bigoplus_{\mbf{n}\in\mbb{N}^I}K_{T}([\bigoplus_{ij\in E}\text{Hom}(V_i,V_j)/\prod_{i\in I}GL(V_i)]),*).
\end{align*}

Contracting to the original point will give the isomorphism:
\begin{align*}
K_{T}([\bigoplus_{ij\in E}\text{Hom}(V_i,V_j)/\prod_{i\in I}GL(V_i)])\cong K_{T}([pt/\prod_{i\in I}GL(V_i)])\cong\mbb{Z}[q^{\pm1},t_{e}^{\pm1}]_{e\in E}[\cdots,z_{i1}^{\pm1},\cdots, z_{in_i}^{\pm1},\cdots]_{i\in I}^{\text{Sym}}.
\end{align*}

We have the algebra morphism:
\begin{align}\label{integral-shuffle-realisation}
\mc{A}^{+,nilp,\mbb{Z}}_{Q}\rightarrow\mc{A}^{+,\mbb{Z}}_{Q}\rightarrow(\bigoplus_{\mbf{n}\in\mbb{N}^I}\mbb{Z}[q^{\pm1},t_{e}^{\pm1}]_{e\in E}[\cdots,z_{i1}^{\pm1},\cdots, z_{in_i}^{\pm1},\cdots]_{i\in I}^{\text{Sym}},*).
\end{align}

We denote the left hand side as the integral big shuffle algebra:
\begin{align*}
\mc{V}^{\mbb{Z}}_{Q}:=(\bigoplus_{\mbf{n}\in\mbb{N}^I}\mbb{Z}[q^{\pm1},t_{e}^{\pm1}]_{e\in E}[\cdots,z_{i1}^{\pm1},\cdots, z_{in_i}^{\pm1},\cdots]_{i\in I}^{\text{Sym}},*).
\end{align*}

The Hall product on $\mc{V}_{Q}^{\mbb{Z}}$ can be written as the following shuffle product:
\begin{align*}
&F(\cdots,z_{i1},\cdots,z_{in_i},\cdots)*F'(\cdots,z_{i1},\cdots,z_{in_i'},\cdots)=\\
&\text{Sym}[\frac{F(\cdots,z_{i1},\cdots,z_{in_i},\cdots)F'(\cdots,z_{i,n_i+1},\cdots,z_{i,n_i+n_i'},\cdots)}{\mbf{n}!\cdot\mbf{n}'!}\prod^{i,j\in I}_{\substack{1\leq a\leq n_i\\ n_j<b\leq n_j+n_j'}}\zeta_{ij}(\frac{z_{ia}}{z_{jb}})].
\end{align*}

Here:
\begin{align}\label{zeta-function}
\zeta_{ij}(x)=(\frac{1-xq^{-1}}{1-x})^{\delta^i_j}\prod_{e=ij\in E}(1-t_ex)\prod_{e=ji\in E}(1-\frac{q}{t_ex}).
\end{align}
It can be seen that $\zeta_{ij}(x)$ has simple poles at $z_{ia}-z_{ib}$ for all $i\in I$ and all $a<b$. Also these poles vanish when taking the symmetrization, as the orders if such poles in a symmetric rational function must be even.

It has been proved in \cite{VV22} that the above algebra morphism \ref{integral-shuffle-realisation} is an injective $\mbb{Z}[q^{\pm1},t_{e}^{\pm1}]$-algebra morphism. This means that one can use the shuffle elements in $\mc{V}_{Q}^{\mbb{Z}}$ to describe the elements in $\mc{A}^{+,nilp,\mbb{Z}}_{Q}$ and $\mc{A}^{+,\mbb{Z}}_{Q}$.

For the localised form $\mc{A}^{+}_{Q}$, we consider the localised big shuffle algebra over $\mbb{F}:=\mbb{Q}(q,t_e)_{e\in E}$
\begin{align*}
\mc{V}_{Q}=\bigoplus_{\mbf{n}\in\mbb{N}^I}\mbb{F}[\cdots,z_{i1}^{\pm1},\cdots,z_{in_i}^{\pm1},\cdots]^{Sym}_{i\in I}.
\end{align*}

\begin{defn}[See \cite{N21}]
The \textbf{shuffle algebra} is defined as the subspace:
\begin{align*}
\mc{S}_{Q}^+\subset\mc{V}_{Q}
\end{align*}
of Laurent polynomials $F(\cdots,z_{i1},\cdots,z_{in_i},\cdots)$ that satisfy the "wheel conditions":
\begin{align*}
F|_{z_{ia}=\frac{qz_{jb}}{t_e}}=F|_{z_{ja}=t_ez_{ib}=qz_{jc}}=0
\end{align*}
for all edges $e=ij$ and all $a\neq c$( and further $a\neq b\neq c$ if $i=j$)
\end{defn}

The following theorem \cite{N21} shows that $\mc{A}_{Q}^+$ is the shuffle realisation of the preprojective K-theoretic Hall algebra:

\begin{thm}\label{wheel-shuffle-theorem:theorem}[See \cite{N21}]
There is an isomorphism of $\mbb{F}$-algebras:
\begin{align*}
Y:\mc{A}_{Q}^+\rightarrow\mc{S}_{Q}^+,\qquad e_{i,d}\mapsto z_{i1}^d.
\end{align*}
\end{thm}

From now on we will always use $\mc{A}_{Q}^+$ as both the localised preprojective KHA and the shuffle algebra realisation of the localised preprojective KHA.

We list some properties about the shuffle algebra $\mc{A}^+_{Q}$:
\begin{itemize}
	\item As an $\mbb{F}$-algebra, $\mc{A}^+$ is generated by $\{z_{i1}^d\}^{d\in\mbb{Z}}_{i\in I}$.
	\item The algebra $\mc{A}^+$ is $\mbb{N}^I\times\mbb{Z}$ graded via:
	\begin{align*}
	\text{deg}(F)=(\mbf{n},d)
	\end{align*}
	if $F$ lies in the $\mbf{n}$-th direct summand and has homogeneous degree $d$. And we denote the \textbf{horizontal degree and vertical degree} as:
	\begin{align}\label{horizontal-vertical-degree}
	\text{hdeg}(F)=\mbf{n},\qquad\text{vdeg}(F)=d
	\end{align}
	We denote the graded pieces of the shuffle algebra by:
	\begin{align*}
	\mc{A}^+=\bigoplus_{\mbf{n}\in\mbb{N}^I}\mc{A}^+_{\mbf{n}}=\bigoplus_{(\mbf{n},d)\in\mbb{N}^I\times\mbb{Z}}\mc{A}^+_{\mbf{n},d}.
	\end{align*}
	\item For any $\mbf{k}\in\mbb{Z}^{I}$ we have a shift automorphism:
	\begin{equation}\label{Shift-automorphism}
	\begin{tikzcd}
	\mc{A}^+_{Q}\arrow[r,"\tau_{\mbf{k}}"]&\mc{A}^+_{Q}
	\end{tikzcd}
	,\qquad F(\cdots,z_{ia},\cdots)\mapsto F(\cdots,z_{ia},\cdots)\prod_{i\in I,a\geq1}z_{ia}^{k_i}.
	\end{equation}

	Similarly, for the $\mc{A}^{+,op}$ the opposite algebra, we can also have the shift automorphism:
	\begin{equation*}
	\begin{tikzcd}
	\mc{A}^{+,op}_{Q}\arrow[r,"\tau_{\mbf{k}}"]&\mc{A}^{+,op}_{Q}
	\end{tikzcd}
	,\qquad G(\cdots,z_{ia},\cdots)\mapsto G(\cdots,z_{ia},\cdots)\prod_{i\in I,a\geq1}z_{ia}^{-k_i}.
	\end{equation*}
\end{itemize}

Also note that the shift automorphism \ref{Shift-automorphism} can be restricted to the integral shuffle algebra $\mc{A}^{+,nilp,\mbb{Z}}$ and $\mc{A}^{+,\mbb{Z}}_{Q}$. i.e. $\tau_{\mbf{k}}$ also gives the automorphisms:
\begin{align*}
\tau_{\mbf{k}}:\mc{A}^{+,nilp,\mbb{Z}}_{Q}\rightarrow \mc{A}^{+,nilp,\mbb{Z}}_{Q},\qquad\tau_{\mbf{k}}:\mc{A}^{+,\mbb{Z}}_{Q}\rightarrow \mc{A}^{+,\mbb{Z}}_{Q}.
\end{align*}

\subsubsection{\textbf{Drinfeld double of shuffle algebras}}

Another way of realising the Drinfeld double of $\mc{A}^{+}_Q$ is by the following:
\begin{align}\label{Hopf-double-drinfeld}
\mc{A}_{Q}=\mc{A}^+_{Q}\otimes\mbb{F}[h_{i,\pm0}^{\pm1},h_{i,\pm1},h_{i,\pm2},\cdots]_{i\in I}\otimes\mc{A}^{+,op}_{Q}/(\text{relation})
\end{align}
We denote the generators in $\mc{A}^+$ and $\mc{A}^{+,op}$ by:
\begin{align*}
e_{i,d}=z_{i1}^d\in\mc{A}^+_{Q},\qquad f_{i,d}=z_{i1}^d\in\mc{A}^{+,op}_{Q}.
\end{align*}

We can write them into the generating series:
\begin{align*}
e_{i}(z)=\sum_{d\in\mbb{Z}}\frac{e_{i,d}}{z^d},\qquad f_{i}(z)=\sum_{d\in\mbb{Z}}\frac{f_{i,d}}{z^d},\qquad h_{i}^{\pm}(w)=\sum_{d=0}^{\infty}\frac{h_{i,\pm d}}{w^{\pm d}}.
\end{align*}

We set:
\begin{align*}
&e_{i}(z)h_{j}^{\pm}(w)=h_{j}^{\pm}(w)e_{i}(z)\frac{\zeta_{ij}(z/w)}{\zeta_{ji}(w/z)}\\
&f_{i}(z)h_{j}^{\pm}(w)=h_{j}^{\pm}(w)f_{i}(z)\frac{\zeta_{ji}(w/z)}{\zeta_{ij}(z/w)}.
\end{align*}

The grading can be extended to the whole of $\mc{A}_{Q}$ by setting
\begin{align*}
\text{deg}(h_{i,\pm d})=(0,\pm d)
\end{align*}
the shift automorphism can be extended to automorphisms:
\begin{align*}
\tau_{\mbf{k}}:\mc{A}_{Q}\rightarrow\mc{A}_{Q}
\end{align*}
by setting $\tau_{\mbf{k}}(h_{i,\pm d})=h_{i,\pm d}$ for all $i\in I$ and $d\in\mbb{N}$.

The coproduct can be defined as:
\begin{align*}
&\Delta(h_{i}^{\pm}(z))=h_{i}^{\pm}(z)\otimes h_{i}^{\pm}(z)\\
&\Delta(e_i(z))=e_i(z)\otimes 1+h_{i}^+(z)\otimes e_{i}(z)\\
&\Delta(f_{i}(z))=f_{i}(z)\otimes h_{i}^-(z)+1\otimes f_{i}(z).
\end{align*}

We denote the extended subalgebras:
\begin{align*}
&\mc{A}^{\geq}=\mc{A}^+\otimes\mbb{F}[h_{i,+0}^{\pm1},h_{i,1},h_{i,2},\cdots]_{i\in I}\\
&\mc{A}^{\leq}=\mc{A}^{+,op}\otimes\mbb{F}[h_{i,-0}^{\pm1},h_{i,-1},h_{i,-2},\cdots]_{i\in I}.
\end{align*}
Also one can construct the nondegenerate Hopf pairing:
\begin{align}\label{Hopf-algebra-pairing}
\langle-,-\rangle:\mc{A}^{\geq}\otimes\mc{A}^{\leq}\rightarrow\mbb{F}
\end{align}
which is defined by the following formulas:
\begin{align*}
\langle h_{i}^{+}(z),h_{j}^{-}(w)\rangle=\frac{\zeta_{ij}(z/w)}{\zeta_{ji}(w/z)}
\end{align*}
and 
\begin{align*}
\langle e_{i,d},f_{j,k}\rangle=\delta^{i}_{j}\gamma_{i}\delta_{d+k}^0,\qquad\gamma_i=\frac{\prod_{e=ii\in E}[(1-t_e)(1-\frac{q}{t_e})]}{(1-\frac{1}{q})}.
\end{align*}

Using the pairing, we can define the Hopf algebra structure on $\mc{A}_{Q}$ defined in \ref{Hopf-double-drinfeld}. It is easy to see that $\mc{A}_{Q}$ is a subalgebra of $\mc{A}^{ext}_{Q}$ defined in \ref{ext-double-KHA}. Throughout the paper we will only use the realisation as $\mc{A}^{ext}_{Q}$.

\subsubsection{\textbf{Coproduct and pairing formula in terms of the shuffle algebra}}
Using the Drinfeld pairing for the double of the shuffle algebra, one can induce a coproduct structure over the shuffle algebra which can be expressed as follows: 
For $F\in\mc{A}^+_{Q,\mbf{n}}$ and $G\in\mc{A}^{+,op}_{\mbf{n}}$, we have the coproduct formula:
\begin{align}\label{shuffle-infty-coproduct-positive}
\Delta(F)=\sum_{[0\leq k_i\leq n_i]_{i\in I}}\frac{\prod_{k_j<b\leq n_j}^{j\in I}h_{j}^{+}(z_{jb})F(\cdots,z_{i1},\cdots,z_{ik_i}\otimes z_{i,k_i+1},\cdots,z_{in_i},\cdots)}{\prod_{1\leq a\leq k_i}^{i\in I}\prod_{k_j<b\leq n_j}^{j\in I}\zeta_{ji}(z_{jb}/z_{ia})}\in\mc{A}^{0}_{Q}\mc{A}^{+}_{Q}\hat{\otimes}\mc{A}^{+}_{Q}
\end{align}

\begin{align}\label{shuffle-infty-coproduct-negative}
\Delta(G)=\sum_{[0\leq k_i\leq n_i]_{i\in I}}\frac{F(\cdots,z_{i1},\cdots,z_{ik_i}\otimes z_{i,k_i+1},\cdots,z_{in_i},\cdots)\prod_{1\leq a\leq k_i}^{i\in I}h_{j}^{-}(z_{ia})}{\prod_{1\leq a\leq k_i}^{i\in I}\prod_{k_j<b\leq n_j}^{j\in I}\zeta_{ji}(z_{ia}/z_{jb})}\in\mc{A}^{-}_{Q}\hat{\otimes}\mc{A}^{-}_{Q}\mc{A}^0_{Q}.
\end{align}

We expand the denominator as a power series in the range $\lvert z_{ia}\lvert<<\lvert z_{jb}\lvert$, and place all the powers of $z_{ia}$ to the left of the $\otimes$ sign and all the powers of $z_{jb}$ to the right of the $\otimes$ sign.

The bialgebra pairing can be written in terms of the residue integral:
\begin{align*}
\langle F,f_{i_1,d_1}*\cdots*f_{i_n,d_n}\rangle=\int_{\lvert z_1\lvert<<\cdots<<\lvert z_{n}\lvert}\frac{z_1^{d_1}\cdots z_{n}^{d_n}F(z_1,\cdots,z_n)}{\prod_{1\leq a<b\leq n}\zeta_{i_ai_b}(z_a/z_b)}\prod_{a=1}^{n}\frac{dz_a}{2\pi iz_a}
\end{align*}

\begin{align*}
\langle e_{i_1,d_1}*\cdots*e_{i_n,d_n},G\rangle=\int_{|z_1|>>\cdots>>|z_{n}|}\frac{z_1^{d_1}\cdots z_{n}^{d_n}G(z_1,\cdots,z_n)}{\prod_{1\leq a<b\leq n}\zeta_{i_bi_a}(z_b/z_a)}\prod_{a=1}^{n}\frac{dz_a}{2\pi iz_a}.
\end{align*}

\subsection{Slope subalgebras and factorizations of $R$-matrices}\label{sub:slope_subalgebras_and_factorizations_of_r_matrices}
One of the convenient application for the shuffle realisation of the KHA is that one can introduce the slope filtration and slope subalgebras for $\mc{A}^{\pm}_{Q}$. One can also refer to \cite{N22} for details.

\subsubsection{Slope filtration in the shuffle settings}
Fix a rational vector $\mbf{m}\in\mbb{Q}^{I}$, we have defined the slope filtration $\mc{A}^{+,\mbb{Z}}_{Q,\leq\mbf{m}}$ in Definition \ref{Slope-filtration:Definition}. Such a filtration can be lifted to the localised form in a similar way, and we denote the corresponding subspace as $\mc{A}^{+}_{Q,\leq\mbf{m}}$. For now we consider such a filtration in terms of the shuffle algebras.

Now consider $\mbb{N}^I\subset\mbb{Z}^I\subset\mbb{Q}^I$ the space of sequence of non-negative integers. We denote:
\begin{align*}
\mbf{e}_i=\underbrace{(0,\cdots,0,1}_{i\text{-th position}},0).
\end{align*}

\begin{defn}[See \cite{N21}]
Let $\mbf{m}\in\mbb{Q}^{I}$. We say that a shuffle element $F\in\mc{V}_{Q}$ has slope $\leq\mbf{m}$ if:
\begin{align}\label{positive-slope-condition}
\lim_{\xi\rightarrow\infty}\frac{F(\cdots,\xi z_{i1},\cdots,\xi z_{ik_i},z_{i,k_i+1},\cdots,z_{in_i},\cdots)}{\xi^{\mbf{m}\cdot\mbf{k}+\langle\mbf{k},\mbf{n}-\mbf{k}\rangle}}
\end{align}
is finite for all $\mbf{0}\leq\mbf{k}\leq\mbf{n}$. Similarly, we will say that $G\in\mc{A}^{-}$ has slope $\leq\mbf{m}$ if:
\begin{align}\label{negative-slope-condition}
\lim_{\xi\rightarrow0}\frac{G(\cdots,\xi z_{i1},\cdots,\xi z_{ik_i},z_{i,k_i+1},\cdots,z_{in_i},\cdots)}{\xi^{-\mbf{m}\cdot\mbf{k}-\langle\mbf{n}-\mbf{k},\mbf{k}\rangle}}
\end{align}
is finite for all $\mbf{0}\leq\mbf{k}\leq\mbf{n}$. Here $\langle-,-\rangle$ and $-\cdot-$ are defined as \ref{notation-for-inner-product-on-dimension-vector}.
\end{defn}

It turns out that the above condition on the vertical degree bounding coincides with the definition \ref{Slope-filtration:Definition}.
\begin{lem}
The image of $\mc{A}^{+}_{Q,\leq\mbf{m}}$ in the shuffle algebra $\mc{V}_{Q}$ are the elements in $\mc{S}_{Q}^+\cong\mc{A}^{+}_{Q}$ of slope $\leq\mbf{m}$.
\end{lem}
\begin{proof}
By the injectivity, every element in $\mc{A}^{+}_{Q,\leq\mbf{m}}$ can be described as the shuffle elements in $\mc{S}_{Q}^+$. Morevoer, given an element $F(\cdots,z_{i1},\cdots,z_{in_i},\cdots)\in\mc{A}^{+}_{Q,\leq\mbf{m}}\subset\mc{S}_{Q}$, the torus action $\sigma_{\mbf{k}}$ on $F$ is equivalent to giving the scaling such that:
\begin{align*}
\lim_{\xi\rightarrow\infty}\frac{F(\cdots,\xi z_{i1},\cdots,\xi z_{ik_i},z_{i,k_i+1},\cdots,z_{in_i},\cdots)}{\xi^{\mbf{m}\cdot\mbf{k}+\langle\mbf{k},\mbf{n}-\mbf{k}\rangle}}<\infty.
\end{align*}

Then it is obvious that the vertical degree condition in Definition \ref{Slope-filtration:Definition} is equivalent to the condition \ref{positive-slope-condition}. Thus the proof is finished.
\end{proof}

We denote the corresponding subspace by $\mc{A}_{\leq\mbf{m},Q}^+$ and $\mc{A}_{\leq\mbf{m},Q}^{-}$. It has been proved in \cite{N22} that $\mc{A}_{\leq\mbf{m},Q}^{\pm}$ are subalgebras of $\mc{A}^{\pm}_{Q}$.

If we put the coproduct $\Delta(F)$ over the elements $\mc{F}\in\mc{A}_{Q,\leq\mbf{m}}^{\pm}$, it can be factorised as:
\begin{align}\label{infty-coproduct-on-slope-m-elements}
&\Delta(F)=\Delta_{\mbf{m}}(F)+(\text{anything})\otimes(\text{slope }<\mbf{m}),\qquad F\in\mc{A}_{Q,\leq\mbf{m}}^+\\
&\Delta(G)=\Delta_{\mbf{m}}(G)+(\text{slope }<\mbf{m})\otimes(\text{anything}),\qquad G\in\mc{A}_{Q,\leq\mbf{m}}^{-}.
\end{align}

These coproduct formulas $\Delta_{\mbf{m}}$ can be written in a more concrete way:
\begin{align}\label{definition-of-slope-coproduct-positive}
\Delta_{\mbf{m}}(F)=\sum_{\mbf{0}\leq\mbf{k}\leq\mbf{n}}\lim_{\xi\rightarrow\infty}\frac{h_{\mbf{n}-\mbf{k}}F(\cdots,z_{i1},\cdots,z_{ik_i}\otimes\xi z_{i,k_i+1},\cdots,\xi z_{in_i})}{\xi^{\mbf{m}\cdot(\mbf{n}-\mbf{k})}\cdot\text{lead}[\prod_{1\leq a\leq k_i}^{i\in I}\prod_{k_j<b\leq n_j}^{j\in I}\zeta_{ji}(\frac{\xi z_{jb}}{z_{ia}})]},\qquad\forall F\in\mc{A}_{Q,\mbf{m}|\mbf{n}}^+
\end{align}

\begin{align}\label{definition-of-slope-coproduct-negative }
\Delta_{\mbf{m}}(G)=\sum_{0\leq\mbf{k}\leq\mbf{n}}\lim_{\xi\rightarrow0}\frac{G(\cdots,\xi z_{i1},\cdots,\xi z_{ik_i}\otimes z_{i,k_i+1},\cdots,z_{in_i},\cdots)h_{-\mbf{k}}}{\xi^{-\mbf{m}\cdot\mbf{k}}\cdot\text{lead}[\prod_{1\leq a\leq k_i}^{i\in I}\prod_{k_j<b\leq n_j}^{j\in I}\zeta_{ji}(\frac{\xi z_{ia}}{z_{jb}})]},\qquad\forall G\in\mc{A}_{Q,\mbf{m}|-\mbf{n}}^-.
\end{align}

Also we say that $F\in\mc{A}^+$, respectively $G\in\mc{A}^-$, has naive slope $\leq\mbf{m}$ if:
\begin{align*}
&\text{vdeg}(F)\leq\mbf{m}\cdot\text{hdeg}(F)\\
&\text{vdeg}(G)\geq\mbf{m}\cdot\text{hdeg}(G).
\end{align*}

For the coproduct on $F\in\mc{A}^+$ of slope $\leq\mbf{m}$, it can be written in the following form:
\begin{align*}
\Delta(F)=(\text{anything})\otimes(\text{naive slope}\leq\mbf{m}).
\end{align*}

Similarly for an element $G\in\mc{A}^-$ has slope $\leq\mbf{m}$ we have
\begin{align*}
\Delta(G)=(\text{naive slope}\leq\mbf{m})\otimes(\text{anything}).
\end{align*}

Using the slope filtration, we still denote the subspaces of shuffle elements of slope $\leq\mbf{m}$ by $\mc{A}^{\pm}_{Q,\leq\mbf{m}}\subset\mc{A}^{\pm}_{Q}$. Via computation, one can actually show that these are subalgebras of $\mc{A}^{\pm}_{Q}$.

It is easy to see that the $\mbb{Z}\times\mbb{N}^I$-graded pieces
\begin{align*}
\mc{A}_{Q,\leq\mbf{m}|\pm\mbf{n},\pm d}^{\pm}=\mc{A}_{Q,\pm\mbf{n},\pm d}^{\pm}\cap\mc{A}^{\pm}_{Q,\leq\mbf{m}}
\end{align*}
are finite-dimensional for any $(\mbf{n},d)\in\mbb{N}^I\times\mbb{Z}$ since there are upper (lower) bounds on the exponents of the variables that make up the Laurent polynomials, and with the condition of the fixing of the total homogeneous vertical degree of such a Laurent polynomial, it is obvious that there are only finitely many choices.

\subsubsection{Slope subalgebras and factorisations}
Now we define the \textbf{positive/negative slope subalgebra} of slope $\mbf{m}$ is 
\begin{align*}
\mc{B}^{\pm}_{\mbf{m}}\subset\mc{A}^{\pm}_{Q}
\end{align*}
as the subspace consisting of elements of slope $\leq\mbf{m}$ and naive slope $=\mbf{m}$. Or equivalently, the graded pieces of $\mc{B}^{\pm}_{\mbf{m}}$ can be easily seen that
\begin{align*}
\mc{B}^{\pm}_{\mbf{m}}=\bigoplus_{\mbf{n}\in\mbb{N}^I}^{\mbf{m}\cdot\mbf{n}\in\mbb{Z}}\mc{B}_{\mbf{m}|\pm\mbf{n}}^{\pm}
\end{align*}
with
\begin{align*}
\mc{B}_{\mbf{m}|\pm\mbf{n}}^{\pm}=\mc{A}_{Q,\leq\mbf{m}|\pm\mbf{n},\pm\mbf{m}\cdot\mbf{n}}^{\pm}.
\end{align*}

It has been proved in \cite{N22} that $\mc{B}_{\mbf{m}}^{\pm}$ is a $\mbb{Q}(q,t_e)_{e\in E}$-subalgebra of $\mc{A}^{\pm}_{Q}$. It is not hard to check that this definition of the shuffle algebra coincides with the one in the Definition \ref{Slope-subalgebra-definition-derived
:label}.

We define the extended $\mc{B}_{\mbf{m}}^{\geq,\leq}$ as:
\begin{align}\label{extended-half-slope-subalgebra}
&\mc{B}_{\mbf{m}}^{\geq}=\mc{B}_{\mbf{m}}^{+}\otimes\mbb{F}[h_{i,+0}^{\pm1}]/\text{relation}\\
&\mc{B}_{\mbf{m}}^{\leq}=\mc{B}_{\mbf{m}}^{-}\otimes\mbb{F}[h_{i,-0}^{\pm1}]/\text{relation}.
\end{align}

Here relation means the relation stated in Section \ref{subsection:localised_and_integral_extended_double_kha}.

The coproduct $\Delta_{\mbf{m}}$ can be defined as:
\begin{align*}
\Delta_{\mbf{m}}(h_{i,\pm0})=h_{i,\pm0}\otimes h_{i,\pm0}
\end{align*}
and for any $F\in\mc{B}_{\mbf{m}|\mbf{n}}$ and $G\in\mc{B}_{\mbf{m}|-\mbf{n}}$:
\begin{align}\label{coproduct-m-slope-subalgebra-positive}
\Delta_{\mbf{m}}(F)=\sum_{\mbf{0}\leq\mbf{k}\leq\mbf{n}}\lim_{\xi\rightarrow\infty}\frac{h_{\mbf{n}-\mbf{k}}F(\cdots,z_{i1},\cdots,z_{ik_i}\otimes\xi z_{i,k_{i}+1},\cdots,\xi z_{in_i},\cdots)}{\xi^{\mbf{m}\cdot(\mbf{n}-\mbf{k})}\text{lead}[\prod_{1\leq a\leq k_i}^{i\in I}\prod_{k_j<b\leq n_j}^{j\in I}\zeta_{ji}(\xi z_{jb}/z_{ia})]}\in\mc{B}_{\mbf{m}}^{\geq}\otimes\mc{B}_{\mbf{m}}^{+}
\end{align}
\begin{align}\label{coproduct-m-slope-subalgebra-negative}
\Delta_{\mbf{m}}(G)=\sum_{\mbf{0}\leq\mbf{k}\leq\mbf{n}}\lim_{\xi\rightarrow0}\frac{F(\cdots,\xi z_{i1},\cdots,\xi z_{ik_i}\otimes z_{i,k_{i}+1},\cdots, z_{in_i},\cdots)h_{-\mbf{k}}}{\xi^{-\mbf{m}\cdot\mbf{k}}\text{lead}[\prod_{1\leq a\leq k_i}^{i\in I}\prod_{k_j<b\leq n_j}^{j\in I}\zeta_{ji}(\xi z_{ia}/z_{jb})]}\in\mc{B}_{\mbf{m}}^{-}\otimes\mc{B}_{\mbf{m}}^{\leq}.
\end{align}

Here we have the notation:
\begin{align*}
h_{\pm\mbf{n}}=\prod_{i\in I}h_{i,\pm0}^{n_i}.
\end{align*}

Here $\Delta_{\mbf{m}}$ consists of the leading naive slope terms in formulas in the sense that:
\begin{align*}
\Delta_{\mbf{m}}(F)=\text{component of }\Delta(F)\text{ in }\bigoplus_{\mbf{n}=\mbf{n}_1+\mbf{n}_2}h_{\mbf{n}_2}\mc{A}_{\mbf{n}_1,\mbf{m}\cdot\mbf{n}_1}\otimes\mc{A}_{\mbf{n}_2,\mbf{m}\cdot\mbf{n}_2}
\end{align*}
\begin{align*}
\Delta_{\mbf{m}}(G)=\text{component of }\Delta(G)\text{ in }\bigoplus_{\mbf{n}=\mbf{n}_1+\mbf{n}_2}\mc{A}_{-\mbf{n}_1,-\mbf{m}\cdot\mbf{n}_1}\otimes\mc{A}_{-\mbf{n}_2,-\mbf{m}\cdot\mbf{n}_2}h_{-\mbf{n}_1}.
\end{align*}

\begin{prop}[\cite{N22}]\label{Drinfeld-pairing-to-slope-subalgebra:prop}
The restriction of the pairing $\langle-,-\rangle:\mc{A}^{\geq}_{Q}\otimes\mc{A}^{\leq}_{Q}\rightarrow\mbb{F}$ to $\mc{B}_{\mbf{m}}^{\geq,\leq}$ is also the Drinfeld pairing with resepct to the coproduct $\Delta_{\mbf{m}}$.
\end{prop}

In other words, we can use the pairing $\langle-,-\rangle$ to define the Drinfeld double of $\mc{B}_{\mbf{m}}^{\geq,\leq}$, and we denote the corresponding algebra as $\mc{B}_{\mbf{m}}$. This is called the \textbf{slope subalgebra}. The Drinfeld pairing above induces a quasi-triangular $\mbb{Q}(q,t_e)_{e\in E}$-Hopf algebra structure on $\mc{B}_{\mbf{m}}$, and we denote the corresponding universal $R$-matrices as $R_{\mbf{m}}^+$, and the reduced universal $R$-matrices as $R_{\mbf{m}}'$, i.e. the universal $R$-matirx without the Cartan elements. Here we can choose the Cartan part of $R_{\mbf{m}}$ such that:
\begin{align}\label{decomposition-of-universal-R-matrix}
R_{\mbf{m}}=q^{-\Omega}R_{\mbf{m}}'
\end{align}
when acting on $K(\mbf{w}_1)\otimes K(\mbf{w}_2)$. Here $\Omega=\frac{1}{2}(\mbf{v}_1\cdot\mbf{w}_2+\mbf{v}_2\cdot\mbf{w}_1-\langle\mbf{v}_1,\mbf{v}_2\rangle)$ on the component $K(\mbf{v}_1,\mbf{w}_1)\otimes K(\mbf{v}_2,\mbf{w}_2)$\footnote{If you do the computation on the original Cartan part of the universal $R$-matrix, you can see there ares some term of the form $\mbf{w}_1\cdot\mbf{w}_2$ dropped over here, and we will not need these terms. It will not affect the result.}.

Here we use the convention that the universal $R$-matrix will be written as:
\begin{align}\label{algebraic-upper-triangular-universal-R-matrix-convention}
R_{\mbf{m}}^+:=R_{\mbf{m}}\in\mc{B}_{\mbf{m}}^{\geq}\hat{\otimes}\mc{B}_{\mbf{m}}^{\leq}
\end{align}
and this stands for the lower-triangular one with respect to the coproduct $\Delta_{\mbf{m}}$ as being described before the formula \ref{triangular-rule}.

The slope subalgebras turn out to give a factorisation of the preprojective KHA, the following theorem was proved by Negu\c{t} in \cite{N22}:
\begin{thm}[See Theorem 1.1 in \cite{N22}]\label{Main-theorem-on-slope-factorisation:Theorem}
Fixing $\mbf{m}\in\mbb{Z}^{I}$ and $\bm{\theta}\in\mbb{Q}^I_{+}$, the multiplication map:
\begin{align*}
\bigotimes^{\rightarrow}_{r\in\mbb{Q}}\mc{B}^{\pm}_{\mbf{m}+r\bm{\theta}}\rightarrow\mc{A}^{\pm}_{Q}
\end{align*}
gives an isomorphism. Moreover, we have the following isomorphism:
\begin{align*}
\mc{A}_{Q}^{ext}\cong\bigotimes^{\rightarrow}_{r\in\mbb{Q}\sqcup\{\infty\}}\mc{B}_{\mbf{m}+r\bm{\theta}}
\end{align*}
with $\mc{B}_{\infty}:=\mc{A}^{0}_{Q}$, and the isomorphism preserves the Drinfeld pairing and gives a factorization of the universal $R$-matrix with respect to the Drinfeld coproduct structure:
\begin{align*}
R=\prod_{r\in\mbb{Q}\sqcup\{\infty\}}^{\rightarrow}R'_{\mbf{m}+r\bm{\theta}}
\end{align*}
\end{thm}

\subsubsection{Integral slope subalgebras and integral slope factorisations}
Since we are concentrating on the integral case, we can define the integral analog of the slope subalgebras.

For each integral KHA $\mc{A}^{+,\mbb{Z}}_{Q}$ and $\mc{A}^{+,nilp,\mbb{Z}}_{Q}$, one can define its corresponding integral model for the slope subalgebra as:
\begin{equation}\label{integral-slope-subalgebras}
\mc{B}_{\mbf{m}}^{+,\mbb{Z}}:=\mc{B}_{\mbf{m}}^{+}\cap\mc{A}^{+,\mbb{Z}}_{Q},\qquad\mc{B}_{\mbf{m}}^{+,nilp,\mbb{Z}}:=\mc{B}_{\mbf{m}}^{+}\cap\mc{A}^{+,nilp,\mbb{Z}}_{Q}.
\end{equation}

One can do the similar proof as in \cite{N22} to show that $\mc{B}^{+,\mbb{Z}}_{\mbf{m}}$ and $\mc{B}^{+,nilp,\mbb{Z}}_{\mbf{m}}$ are $\mbb{Z}[q^{\pm1},t_{e}^{\pm1}]_{e\in E}$-subalgebras of $\mc{A}^{+,\mbb{Z}}_{Q}$ and $\mc{A}^{+,nilp,\mbb{Z}}_{Q}$ respectively. It is not hard to check that these coincides with the definition in the Definition \ref{Slope-subalgebra-definition-derived
:label}.

Similar to the factorisation theorem \ref{Main-theorem-on-slope-factorisation:Theorem}, one can also have the factorisation for the integral KHA, and the categorical version has been proved in \cite{PT25}:
\begin{align}\label{slope-factoristaion-integral}
\mc{A}^{+,\mbb{Z}}_{Q}=\bigotimes^{\rightarrow}_{r\in\mbb{Q}}\mc{B}^{+,\mbb{Z}}_{\mbf{m}+r\bm{\theta}},\qquad\mc{A}^{+,nilp,\mbb{Z}}_{Q}=\bigotimes^{\rightarrow}_{r\in\mbb{Q}}\mc{B}^{+,nilp,\mbb{Z}}_{\mbf{m}+r\bm{\theta}}.
\end{align}

\subsection{Primitivity of the slope subalgebra}\label{primitivity:_textbf_primitivity_of_the_slope_subalgebra}

For the Hopf algebra $\mc{B}_{\mbf{m}}$, we say that an element $F,G$ in $\mc{B}_{\mbf{m},\mbf{n}}^{\pm}$ is \textbf{primitive} if the coproduct $\Delta_{\mbf{m}}$ on $F$ can be written as:
\begin{align}\label{primitive-element-definition}
\Delta_{\mbf{m}}(F)=F\otimes\text{Id}+h_{\mbf{n}}\otimes F\text{ or }\Delta_{\mbf{m}}(G)=G\otimes h_{-\mbf{n}}+\text{Id}\otimes G\text{ respectively}.
\end{align}

In fact, here we can show that $\mc{B}_{\mbf{m}}$ is generated by the primitive elements. This will be important when we are trying to do the decomposition of the $K$-theoretic Hall algebras.

\begin{thm}\label{primitivity:Theorem}
For arbitrary slope subalgebra $\mc{B}_{\mbf{m}}$, it is generated by the primitive elements with respect to the coproduct $\Delta_{\mbf{m}}$.
\end{thm}

\begin{proof}
We only give the proof of the statement for the positive half, and the proof for the negative half is similar.

We now prove the theorem by induction, and obviously the elements of the lowest vertical degree are primitive. Let us suppose that for $\mc{B}_{\mbf{m}|\mbf{n}_i}$ with the elements in $\mbf{n}_i<\mbf{n}$ being generated by the primitive elements. We now suppose that given $E_{k}\in\mc{B}_{\mbf{m}|\mbf{n}}^{-}$ such that $\mbf{n}$ can be decomposed into two nonzero vectors such that the corresponding elements can be generated by the primitive elements. Moreover, we will say that an element $E,F\in\mc{B}^{\pm}_{\mbf{m}}$ is \textbf{indecomposable} if it cannot be written as the product of elements in $\mc{B}_{\mbf{m}}^{\pm}$ of lower vertical degree respectively.

By the property of the universal $R$-matrix $R^{+}_{\mbf{m}}$, any element $E\in\mc{B}_{\mbf{m},\mbf{n}}^{+}$, if we write down its coproduct $\Delta_{\mbf{m}}(E)$ as:
\begin{align}
\Delta_{\mbf{m}}(E)=E\otimes\text{Id}+\sum_{\mbf{0}\leq\mbf{k}<\mbf{n}}\sum_{\alpha}h_{\mbf{n}-\mbf{k}}E_{\mbf{k}}^{(\alpha)}\otimes E_{\mbf{n}-\mbf{k}}^{(\alpha)}.
\end{align}

Note that $\mc{B}^{\pm}_{\mbf{m}|\mbf{n}}$ are finite-dimensional, thus we denote the orthogonal basis by $\{E_{\mbf{n}}^{(\alpha)},F_{\mbf{n}}^{(\alpha)}\}_{\alpha\in I_{\mbf{n}}}$. Thus we write down the decomposition of the universal $R$-matrix:

\begin{align*}
R_{\mbf{m}}^{+}=\sum_{\substack{\mbf{n}\in\mbb{N}^I|\mbf{m}\cdot\mbf{n}\in\mbb{Z}}}R_{\mbf{m}|\mbf{n}}^+=q^{\sum_{i\in I}H_i\otimes H_{-i}}\sum_{\substack{\mbf{n}\in\mbb{N}^I|\mbf{m}\cdot\mbf{n}\in\mbb{Z}\\\alpha\in I_{\mbf{n}}}}E_{\mbf{n}}^{(\alpha)}\otimes F_{\mbf{n}}^{(\alpha)}.
\end{align*}

Now we use the coproduct formula $(\Delta_{\mbf{m}}\otimes\text{id})R=R_{13}R_{23}$:
\begin{equation}\label{quasi-triangular-relations}
\begin{aligned}
&(\Delta_{\mbf{m}}\otimes\text{id})R_{\mbf{m}}^+=(\Delta_{\mbf{m}}\otimes\text{id})q^{\sum_{i\in I}H_i\otimes H_{-i}}\sum_{\substack{\mbf{n}\in\mbb{N}^I|\mbf{m}\cdot\mbf{n}\in\mbb{Z}\\\alpha\in I_{\mbf{n}}}}\Delta_{\mbf{m}}(E_{\mbf{n}}^{(\alpha)})\otimes F_{\mbf{n}}^{(\alpha)}\\
&(R_{\mbf{m}}^+)_{13}(R_{\mbf{m}}^+)_{23}=\sum_{\substack{\mbf{n}_1,\mbf{n}_2,\mbf{m}\cdot\mbf{n}_1,\mbf{m}\cdot\mbf{n}_2\in\mbb{Z}\\\alpha_1,\alpha_2\in I_{\mbf{n}_1},I_{\mbf{n}_2}}}q^{\sum_{i\in I}H_i\otimes1\otimes H_{-i}}E_{\mbf{n}_1}^{(\alpha_1)}\otimes q^{\sum_{i\in I}1\otimes H_i\otimes H_{-i}}E_{\mbf{n}_2}^{(\alpha_2)}\otimes F_{\mbf{n}_1}^{(\alpha_1)}F_{\mbf{n}_2}^{(\alpha_2)}.
\end{aligned}
\end{equation}

So if we write down $F^{(\alpha_1)}_{\mbf{n}_1}F^{(\alpha_2)}_{\mbf{n}_2}=\sum_{\alpha}a^{\alpha}_{\alpha_1\alpha_2}F_{\mbf{n}}^{(\alpha)}$, one could write down the coproduct formula for $E_{\mbf{n}}^{(\alpha)}$ as:
\begin{align}
\Delta_{\mbf{m}}(E_{\mbf{n}}^{(\alpha)})=\sum_{\substack{\mbf{n}_1+\mbf{n}_2=\mbf{n}\\\mbf{n}_1,\mbf{n}_2\in\mbb{N}^I}}a^{\alpha}_{\alpha_1\alpha_2}h_{\mbf{n}_2}E_{\mbf{n}_1}^{(\alpha_1)}\otimes E_{\mbf{n}_2}^{(\alpha_2)}.
\end{align}

From this formula, one can see that if $F_{\mbf{n}}^{(\alpha)}$ is indecomposable, it cannot appear in the expansion of the product $F^{(\alpha_1)}_{\mbf{n}_1}F^{(\alpha_2)}_{\mbf{n}_2}$ for arbitrary $F^{(\alpha_1)}_{\mbf{n}_1}$ and $F^{(\alpha_2)}_{\mbf{n}_2}$. Therefore if $F_{\mbf{n}}^{(\alpha)}$ is indecomposable, $E_{\mbf{n}}^{(\alpha)}$ is primitive. But being primitive implies that $E_{\mbf{n}}^{(\alpha)}$ is indecomposable, and then we use the above argument again on $\Delta_{\mbf{m}}(F_{\mbf{n}}^{(\alpha)})$, and we can see that in this case $F_{\mbf{n}}^{(\alpha)}$ is primitive. Thus we have that element in $\mc{B}^{\pm}_{\mbf{m}}$ being indecomposable is equivalent to being primitive. Using the induction again, one obtain that the theorem is true.

\end{proof}

Using these we can define the root subalgebra in the slope subalgebra. 

\begin{defn}\label{Definition-of-wall-subalgebra-in-root-subalgebra:Definition}
Given $w$ a wall dual to a vector $\alpha$ such that $\mbf{m}\cdot\alpha\in\mbb{Z}$, one can define the \textbf{root subalgebra} $\mc{B}_{\mbf{m},w}\subset\mc{B}_{\mbf{m}}$ as the subalgebra of $\mc{B}_{\mbf{m}}$ generated by the primitive elements in $\bigoplus_{k\geq0}\mc{B}^{\pm}_{\mbf{m}|k\alpha}$.
\end{defn}

Similar for the case of the slope subalgebra $\mc{B}_{\mbf{m}}^{\pm}$, it also admits the factorisation property:
\begin{lem}\label{root-factorisation-for-slope:label}
The multiplication map induces the isomorphism
\begin{align}\label{root-factorisation-for-slope}
\bigotimes^{\rightarrow}_{\mbf{m}\in w}\mc{B}_{\mbf{m},w}^{\pm}\cong\mc{B}_{\mbf{m}}^{\pm}
\end{align}
which preserves the Drinfeld bialgebra pairing on both sides. Here the order of the tensor product on right hand side.
\end{lem}
\begin{proof}
For the Drinfeld bialgebra pairing preserving, this follows from the fact that the elements in different root subalgebra $\mc{B}_{\mbf{m},w}^{\pm}$ will not have the same vertical degree.

The surjectivity of the map \ref{root-factorisation-for-slope} is obvious. For the injectivity, note that given arbitrary $\gamma\in\mc{B}^{+}_{\mbf{m}}$ written as the product form $\sum_{s}a_s\gamma_{s_1}\cdots\gamma_{s_n}$ with $\gamma_{s_i}\in\mc{B}_{\mbf{m},w_i}^+$, if it were zero, it means that for arbitrary $\delta\in\mc{B}_{\mbf{m}}^-$, we have that:
\begin{align}
\langle\gamma,\delta\rangle=\sum_{s}a_s\langle\gamma_{s_1}\cdots\gamma_{s_n},\delta\rangle=0,\qquad\forall\delta\in\mc{B}_{\mbf{m}}^-.
\end{align}

Since the Drinfeld bialgebra pairing is preserved by the factorisation, the above expression would be zero if and only if $a_s=0$. Thus the Lemma is proved.

\end{proof}

Similarly one can define the integral forms of the root subalgebra $\mc{B}_{\mbf{m},w}^{+,\mbb{Z}}$ and $\mc{B}_{\mbf{m},w}^{+,nilp,\mbb{Z}}$ as follows:
\begin{align}\label{integral-root-subalgebra}
\mc{B}_{\mbf{m},w}^{+,\mbb{Z}}:=\mc{B}_{\mbf{m},w}^{+}\cap\mc{A}^{+,\mbb{Z}}_{Q},\qquad\mc{B}_{\mbf{m},w}^{+,nilp,\mbb{Z}}:=\mc{B}_{\mbf{m},w}^{+}\cap\mc{A}^{+,nilp,\mbb{Z}}_{Q}.
\end{align}

Later on we will see that the root subalgebra will be treated as the algebraic analog of the wall subalgebra in the MO quantum loop group setting.

\section{\textbf{Stable envelopes and Maulik-Okounkov quantum loop groups}}\label{section:_textbf_stable_envelopes_and_maulik_okounkov_quantum_loop_groups}
In this Section we introduce the Maulik-Okounkov quantum loop group from the $K$-theoretic stable envelopes. For the introduction of the $K$-theoretic stable envelope, one can refer to \cite{O15}\cite{OS22}. In this Section we also give the nilpotent $K$-theoretic stable envelope, which can be thought of as the $K$-theory analog of the nilpotent stable envelope considered in \cite{SV23}.

Given a smooth quasi-projective variety $X$ with a torus $A$ action, we denote $\text{cochar}(A)$ as the lattice of cocharacters over $A$:
\begin{align*}
\sigma:\mbb{C}^*\rightarrow A.
\end{align*}

We denote
\begin{align*}
\mf{a}_{\mbb{R}}:=\text{cochar}(A)\otimes_{\mbb{Z}}\mbb{R}\subset\mf{a}.
\end{align*}

We will say that a vector $\alpha\in\mf{a}_{\mbb{R}}$ is the $A$-\textbf{weight} if $\alpha$ appears in the normal bundle of $X^A\subset X$. The associated hyperplane $\alpha^{\perp}$ will divide the vector space $\mf{a}_{\mbb{R}}$ into finitely many chambers:
\begin{align*}
\mf{a}_{\mbb{R}}\backslash\bigcup_{i}\alpha_{i}^{\perp}=\bigsqcup_{i}\mf{C}_{i}.
\end{align*}

Given a connected component $F\subset X^A$, we denote:
\begin{align*}
\text{Attr}_{\mf{C}}(F)=\{x\in X|\lim_{\mf{C}}(x)\in F\}.
\end{align*}
Here the notation $\lim_{\mf{C}}(x)\in F$ means that choosing the cocharacter $\sigma\in\mf{C}$ in the chamber, we have:
\begin{align*}
\lim_{t\rightarrow0}\sigma(t)\cdot x\in F,\qquad\forall\sigma\in\mf{C}.
\end{align*}

For two fixed-point connected components $Z_i,Z_j\subset X^A$, we say that $Z_i\geq Z_j$ if
\begin{align*}
\overline{\text{Attr}_{\mf{C}}(Z_i)}\cap Z_j\neq\emptyset.
\end{align*}

If we denote $X^A=\sqcup_{i}Z_{i}$, we also denote:
\begin{align*}
\text{Attr}_{\mf{C}}^{\leq}:=\bigsqcup_{i\leq j}\text{Attr}_{\mf{C}}(Z_i)\times Z_j\subset X\times X^A.
\end{align*}

We denote the full attracting set $\text{Attr}^{f}_{\mf{C}}\subset X\times X^A$ as the smallest $A$-invariant subspace such that:
\begin{align*}
(p',p)\in\text{Attr}^f\text{ and }\lim_{t\rightarrow0}\sigma(t)\cdot x=p'\text{ implies that }(x,p)\in\text{Attr}^f_{\mf{C}}.
\end{align*}

By definition, $\text{Attr}_{\mf{C}}^f\subset\text{Attr}_{\mf{C}}^{\leq}$ and it is closed in $X\times X^A$. Fix a connected component $Z\subset X^A$ of $X^A$, we denote:
\begin{align*}
\text{Attr}^{f}_{\mf{C}}(Z):=\text{Attr}^f_{\mf{C}}\cap(Z\times X).
\end{align*}

Alternatively, one can use the language of closed flow lines and half-open flow lines to describe the elements of $\text{Attr}^{f}_{\mf{C}}$.

For the Nakajima quiver variety $X=\mc{M}_{Q}(\mbf{v},\mbf{w})$, we choose $\bm{\theta}=(-1,\cdots,-1)$. If we take the cocharacter $\sigma:\mbb{C}^*\rightarrow\prod_{i\in I}GL(W_i)$ such that $\mbf{w}=\mbf{w}_1+a\mbf{w}_2$, in this case the partial order is written as:
\begin{align*}
\mc{M}_{Q}(\mbf{v}_1,\mbf{w}_1)\times\mc{M}_{Q}(\mbf{v}_2,\mbf{w}_2)\geq\mc{M}_{Q}(\mbf{v}_1-\mbf{n},\mbf{w}_1)\times\mc{M}_{Q}(\mbf{v}_2+\mbf{n},\mbf{w}_2).
\end{align*}

Here we will always fix the stability condition to be $\bm{\theta}=(-1,\cdots,-1)$.

The following lemma has been proved in Lemma 3.11 in \cite{N23}, which reveals the relation between $\text{Attr}^{f}_{\mf{C}}$ and $\mc{M}_{Q}(\mbf{v},\mbf{v}+\mbf{n},\mbf{w})$.
\begin{lem}\label{flow-condition:lemma}
Given $(\tilde{V}_{\bullet}',\tilde{V}_{\bullet}'')\in\mc{M}_{Q}(\mbf{v}'-\mbf{n},\mbf{w}')\times\mc{M}_{Q}(\mbf{v}''+\mbf{n},\mbf{w}'')$ and $(V_{\bullet}',V_{\bullet}'')\in\mc{M}_{Q}(\mbf{v}',\mbf{w}')\times\mc{M}_{Q}(\mbf{v}'',\mbf{w}'')$. Then there is a closed flow from $(\tilde{V}_{\bullet}',\tilde{V}_{\bullet}'')$ to $(V_{\bullet}',V_{\bullet}'')$ if and only if $(\tilde{V}_{\bullet}',V_{\bullet}')\in\mc{M}_{Q}(\mbf{v}'-\mbf{n},\mbf{v}',\mbf{w}')$ and $(V_{\bullet}'',\tilde{V}_{\bullet}'')\in\mc{M}_{Q}(\mbf{v}'',\mbf{v}''+\mbf{n},\mbf{w}')$ and their projection to $\mc{Y}_{\mbf{n}}$ are the same point.
\end{lem}

\subsection{$K$-theoretic stable envelopes}\label{subsection:_k_theoretic_stable_envelopes}

We first review the definition of the $K$-theoretic stable envelopes for the quiver varieties. For details one can see \cite{O15}\cite{OS22}.

Given $X:=\mc{M}_{Q}(\mbf{v},\mbf{w})$ a Nakajima quiver variety and $G:=T_{\mbf{w}}$ given in \ref{torus-action-given-fixed} acting on $\mc{M}_{Q}(\mbf{v},\mbf{w})$. Now we fix the following data:
\begin{itemize}
	\item Given a subtorus $A\subset T_{\mbf{w}}$ in the kernel of $q$. Choose a chamber $\mf{C}$ of the torus $A$, which divides the normal direction to $X^A$ into the attractive and repelling side that determines the support $\text{Attr}^{f}_{\mf{C}}$.
	\item A fractional line bundle $s\in\text{Pic}(X)\otimes\mbb{R}$, which is chosen to be outside of the wall set \ref{wall-set:proposition}.
	\item a choice of the polarisation $T^{1/2}$ for $TX$, i.e.
	\begin{align}\label{polarisation-of-tangent-bundle}
	TX=T^{1/2}\oplus q^{-1}(T^{1/2})^{\vee}\in K_{G}(X).
	\end{align}
\end{itemize}

By definition, the $K$-theoretic stable envelope is a $K$-theory class
\begin{align*}
\text{Stab}_{\mf{C},s,T^{1/2}}\subset K_{G}(X\times X^A).
\end{align*}

supported on $\text{Attr}_{\mf{C}}^f$, such that it induces the morphism
\begin{align*}
\text{Stab}_{\mf{C},s}:K_{G}(X^A)\rightarrow K_{G}(X)
\end{align*}
such that if we write $X^A=\sqcup_{\alpha}F_{\alpha}$ into components:
\begin{itemize}
	\item The diagonal term is given by the structure sheaf of the attractive space:
	\begin{align*}
	\text{Stab}_{\mf{C},s}|_{F_{\alpha}\times F_{\alpha}}=(-1)^{\text{rk }T_{>0}^{1/2}}(\frac{\text{det}(\mc{N}_{-})}{\text{det}T_{\neq0}^{1/2}})^{1/2}\otimes\mc{O}_{\text{Attr}}|_{F_{\alpha}\times F_{\alpha}}=(-1)^{\text{rk }T_{>0}^{1/2}}(\frac{\text{det}(\mc{N}_{-,F_{\alpha}})}{\text{det}T_{\neq0}^{1/2}})^{1/2}(\wedge^*\mc{N}_{-,F_{\alpha}}^{\vee}).
	\end{align*}

	\item The $A$-degree of the stable envelope has the bounding condition for $F_{\beta}\leq F_{\alpha}$:
	\begin{align*}
	\text{deg}_{A}\text{Stab}_{\mf{C},s}|_{F_{\beta}\times F_{\alpha}}+\text{deg}_{A}s|_{F_{\alpha}}\subset\text{deg}_{A}\text{Stab}_{\mf{C},s}|_{F_{\beta}\times F_{\beta}}+\text{deg}_{A}s|_{F_{\beta}}.
	\end{align*}

	Here $\text{deg}_{A}(\mc{F})$ means the Newton polytope of the $K$-theory class $\mc{F}\in K_{G}(F_{\beta}\times F_{\alpha})$ treated as a Laurent polynomial over the group characters of $T$ under the isomorphism $K_{T_{\mbf{w}}}(F_{\beta}\times F_{\alpha})\cong K_{T_{\mbf{w}}/A}(F_{\beta}\times F_{\alpha})\otimes K_{A}(pt)$. We require that for $F_{\beta}<F_{\alpha}$, the inclusion $\subset$ is strict.
\end{itemize}

It is easy to see that the stable envelope map is an isomorphism after localisation. In the case of the Nakajima quiver variety, it is an injective map of $K_{A}(pt)$-modules since $K_{T_{\mbf{w}}}(\mc{M}_{Q}(\mbf{v},\mbf{w}))$ is a free $K_{T_{\mbf{w}}}(pt)$-module.

The uniqueness and existence of the $K$-theoretic stable envelope was given in \cite{AO21} and \cite{O21}. In \cite{AO21}, the consturction is given by the abelinization of the quiver varieties. In \cite{O21}, the construction is given by the stratification of the complement of the attracting set, which is much more general.

The stable envelope has the factorisation property called the triangle lemma \cite{O15}. Given a subtorus $A'\subset A$ with the corresponding chamber $\mf{C}_{A'},\mf{C}_{A}$, we have the following diagram commute:
\begin{equation}\label{triangle-lemma}
\begin{tikzcd}
K_{G}(X^A)\arrow[rr,"\text{Stab}_{\mf{C}_A,s}"description]\arrow[dr,"\text{Stab}_{\mf{C}_{A/A'},s}"description]&&K_{G}(X)\\
&K_{G}(X^{A'})\arrow[ur,"\text{Stab}_{\mf{C}_{A'},s}"description]&
\end{tikzcd}.
\end{equation}

In this paper we will fix the polarisation to be:
\begin{align*}
T^{1/2}\mc{M}_{Q}(\mbf{v},\mbf{w})=\sum_{e=ij\in E}\frac{\mc{V}_j}{t_e\mc{V}_i}-\sum_{i\in I}\frac{\mc{V}_i}{\mc{V}_i}+\sum_{i\in I}\frac{\mc{W}_i}{q\mc{V}_i}
\end{align*}
and throughout the paper, we will write down the stable envelope as $\text{Stab}_{\mf{C},s}$ since we always fix the polarisation. Note that usually the slope $s\in\text{Pic}(X)\otimes\mbb{R}$ should not be "rational" since usually rational points are on the walls $w$, while by construction in \cite{O21} we should avoid the walls to make the definition of the stable envelope unique. Throughout the paper, when we mention the slope point $s$ as a rational point in $\text{Pic}(X)\otimes\mbb{Q}$. We usually mean that we choose a point $s+\epsilon$ that is suitably close to $s$ and outside of the wall.

Moreover, we will often use the torus action $\sigma:\mbb{C}^*\rightarrow A$ such that $\mbf{w}=\mbf{w}_1+a\mbf{w}_2$, and in this case the degree condition can be written as:
\begin{equation}
\begin{aligned}
-\text{max deg}_{A}(\text{Stab}_{\mf{C},\mbf{m}}|_{F_{\beta}\times F_{\beta}})+\mbf{m}\cdot\mbf{l}
\leq\text{deg}_{A}(\text{Stab}_{\mf{C},\mbf{m}}|_{F_{\beta}\times F_{\alpha}})
\leq\text{max deg}_{A}(\text{Stab}_{\mf{C},\mbf{m}}|_{F_{\beta}\times F_{\beta}})+\mbf{m}\cdot\mbf{l}
\end{aligned}
\end{equation}
and here $F_{\alpha}=\mc{M}_{Q}(\mbf{v}',\mbf{w}')\times\mc{M}_{Q}(\mbf{v}'',\mbf{w}'')$, $F_{\beta}=\mc{M}_{Q}(\mbf{v}'-\mbf{l},\mbf{w}')\times\mc{M}_{Q}(\mbf{v}''+\mbf{l},\mbf{w}'')$.

\subsection{Maulik-Okounkov quantum loop groups and wall subalgebras}\label{subsection:maulik_okounkov_quantum_algebra_and_wall_subalgebras}

Let us focus on the case of the quiver varieties $\mc{M}_{Q}(\mbf{v},\mbf{w})$. Choose the framing torus $\sigma:A\rightarrow T_{\mbf{w}}$ and the chamber $\mf{C}$ such that:
\begin{align*}
\mbf{w}=a_1\mbf{w}_1+\cdots+a_k\mbf{w}_k.
\end{align*}

In this case the fixed point is given by:
\begin{align*}
\mc{M}_{Q}(\mbf{v},\mbf{w})^{\sigma}=\bigsqcup_{\mbf{v}_1+\cdots+\mbf{v}_k=\mbf{v}}\mc{M}_{Q}(\mbf{v}_1,\mbf{w}_1)\times\cdots\times \mc{M}_{Q}(\mbf{v}_k,\mbf{w}_k).
\end{align*}

Denote $K(\mbf{w}):=\bigoplus_{\mbf{v}}K_{T_{\mbf{w}}}(\mc{M}_{Q}(\mbf{v},\mbf{w}))_{loc}$, it is easy to see that the stable envelope $\text{Stab}_{s}$ gives the map:
\begin{align*}
\text{Stab}_{\mf{C},s}:K(\mbf{w}_1)\otimes\cdots\otimes K(\mbf{w}_k)\rightarrow K(\mbf{w}_1+\cdots+\mbf{w}_k).
\end{align*}

Using the $K$-theoretic stable envelope, we can define the geometric $R$-matrix as:
\begin{align*}
\mc{R}^{s}_{\mf{C}}:=\text{Stab}_{-\mf{C},s}^{-1}\circ\text{Stab}_{\mf{C},s}:K(\mbf{w}_1)\otimes\cdots\otimes K(\mbf{w}_k)\rightarrow K(\mbf{w}_1)\otimes\cdots\otimes K(\mbf{w}_k).
\end{align*}

Written in the component of the weight subspaces, the geometric $R$-matrix can be written as:
\begin{equation*}
\begin{aligned}
&\mc{R}^{s}_{\mf{C}}:=\text{Stab}_{-\mf{C},s}^{-1}\circ\text{Stab}_{\mf{C},s}:\bigoplus_{\mbf{v}_1+\cdots+\mbf{v}_k=\mbf{v}}K(\mbf{v}_1,\mbf{w}_1)\otimes\cdots\otimes K(\mbf{v}_k,\mbf{w}_k)\\
&\rightarrow \bigoplus_{\mbf{v}_1+\cdots+\mbf{v}_k=\mbf{v}}K(\mbf{v}_1,\mbf{w}_1)\otimes\cdots\otimes K(\mbf{v}_k,\mbf{w}_k).
\end{aligned}
\end{equation*}

From the triangle diagram \ref{triangle-lemma} of the stable envelope, we can further factorise the geometric $R$-matrix into the smaller parts:
\begin{align}\label{abstract-decomposition}
\mc{R}^s_{\mf{C}}=\prod_{1\leq i<j\leq k}\mc{R}^s_{\mf{C}_{ij}}(\frac{a_i}{a_j}),\qquad \mc{R}^s_{\mf{C}_{ij}}(\frac{a_i}{a_j}):K(\mbf{w}_i)\otimes K(\mbf{w}_j)\rightarrow K(\mbf{w}_i)\otimes K(\mbf{w}_j).
\end{align}

Each $\mc{R}^{s}_{\mf{C}_{ij}}(u)$ satisfies the trigonometric Yang-Baxter equation with the spectral parametres:
\begin{align}\label{Yang-Baxter-equation}
\mc{R}^{s}_{\mf{C}_{12}}(\frac{a_1}{a_2})\mc{R}^s_{\mf{C}_{13}}(\frac{a_1}{a_3})\mc{R}^s_{\mf{C}_23}(\frac{a_2}{a_3})=\mc{R}^s_{\mf{C}_{23}}(\frac{a_2}{a_3})\mc{R}^s_{\mf{C}_{13}}(\frac{a_1}{a_3})\mc{R}^{s}_{\mf{C}_{12}}(\frac{a_1}{a_2}).
\end{align}

In the language of the representation theory, we denote $V_{i}(a_i)$ as the modules of type $K(\mbf{w}_i)$ defined above with the spectral parametre $a_i$. The formula \ref{abstract-decomposition} means that:
\begin{align*}
\mc{R}^s_{\bigotimes^{\leftarrow}_{i\in I}V_i(a_i),\bigotimes^{\leftarrow}_{j\in J}V_j(a_j)}=\prod_{i\in I}^{\rightarrow}\prod_{j\in J}^{\leftarrow}\mc{R}^{s}_{V_i,V_j}(\frac{a_i}{a_j}).
\end{align*}

We can also consider the dual module $V_i^*(u_i)$ as the module isomorphic to $V_i(u_i)$ as graded vector space, with the $R$-matrices defined as:
\begin{equation*}
\begin{aligned}
&\mc{R}^s_{V_1^*,V_2}=((\mc{R}^s_{V_1,V_2})^{-1})^{*_1}\\
&\mc{R}^s_{V_1,V_2^*}=((\mc{R}^s_{V_1,V_2})^{-1})^{*_2}\\
&\mc{R}^s_{V_1^*,V_2^*}=(\mc{R}^s_{V_1,V_2})^{*_{12}}
\end{aligned}
\end{equation*}
$*_{k}$ means transpose with respect to the $k$-th factor.

\begin{defn}
The Maulik-Okounkov quantum loop group $U_{q}^{MO}(\hat{\mf{g}}_{Q})$ is the $\mbb{Q}(q,t_e)_{e\in E}$-subalgebra of $\prod_{\mbf{w}}\text{End}(K(\mbf{w}))$ generated by the matrix coefficients of $\mc{R}^{s}_{\mf{C}}$.
\end{defn}

In other words, given an auxillary space $V_0=\bigotimes_{\mbf{w}\in I}K(\mbf{w})$, $V=\bigotimes_{\mbf{w}'\in J}K(\mbf{w}')$ with $I$ and $J$ are finite subsets of the set of dimension vectors, for arbitrary finite rank operator
\begin{align*}
m(a_0)\in\text{End}(V_0)(a_0).
\end{align*}

Now the element of $U_{q}^{MO}(\hat{\mf{g}}_{Q})$ is generated by the following operators:
\begin{align*}
\oint_{a_0=0,\infty}\frac{da_0}{2\pi ia_0}\text{Tr}_{V_0}((1\otimes m(a_0))\mc{R}^{s}_{V,V_0}(\frac{a}{a_0}))\in\text{End}(V(a))
\end{align*}
or on the other hand, it is given by the matrix coefficients:
\begin{align*}
\langle\mbf{y},\mc{R}^{s}_{\mf{C}}(\frac{a}{a_0})\mbf{x}\rangle\in\text{End}(V)(a),\qquad\forall\mbf{y},\mbf{x}\in V_0
\end{align*}
and here the inner product $\langle-,-\rangle$ comes from the perfect pairing \ref{perfect-pairing:label}.

It has been proved in \cite{OS22} that for different choice of $s\in\mbb{Q}^I$, the corresponding MO quantum loop groups $U_{q}^{MO}(\hat{\mf{g}}_{Q})$ are isomorphic to each other. That is the reason why we omit the sign $s$ for the quantum loop groups.

The coproduct structure, antipode map and the counit map can be defined as follows:

Fix the slope point $s\in\mbb{Q}^{|I|}$, and for the coproduct $\Delta_{s}$ on $U_{q}^{MO}(\hat{\mf{g}}_{Q})$ is defined via the conjugation by $\text{Stab}_{\mf{C},s}$, i.e. for $a\in U_{q}^{MO}(\hat{\mf{g}}_{Q})$ as $a:K(\mbf{w})\rightarrow K(\mbf{w})$, $\Delta_{s}(a)$ is defined as:
\begin{equation}\label{coproduct-geometric}
\begin{tikzcd}
K(\mbf{w}_1)\otimes K(\mbf{w}_2)\arrow[r,"\text{Stab}_{\mf{C},s}"]&K(\mbf{w}_1+\mbf{w}_2)\arrow[r,"a"]&K(\mbf{w}_1+\mbf{w}_2)\arrow[r,"\text{Stab}_{\mf{C},s}^{-1}"]&K(\mbf{w}_1)\otimes K(\mbf{w}_2).
\end{tikzcd}
\end{equation}

The antipode map $S_{s}:U_{q}^{MO}(\hat{\mf{g}}_{Q})\rightarrow U_{q}^{MO}(\hat{\mf{g}}_{Q})$ is given by:
\begin{align*}
\oint_{a_0=0,\infty}\frac{da_0}{2\pi ia_0}\text{Tr}_{V_0}((1\otimes m(a_0))\mc{R}^{s}_{V,V_0}(\frac{a}{a_0}))\mapsto\oint_{a_0=0,\infty}\frac{da_0}{2\pi ia_0}\text{Tr}_{V_0}((1\otimes m(a_0))\mc{R}^{s}_{V^*,V_0}(\frac{a}{a_0})).
\end{align*}

The projection of the module $V$ to the trivial module $\mbb{C}$ induce the counit map:
\begin{align*}
\epsilon:U_{q}^{MO}(\hat{\mf{g}}_{Q})\rightarrow\mbb{C}.
\end{align*}

Since $\mc{M}_{Q}(0,\mbf{w})$ is just a point, we denote the vector in $K_{\mbf{0},\mbf{w}}$ as $v_{\varnothing}$, and we call it the \textbf{vacuum vector}. 

\begin{comment}
We define the evaluation map:
\begin{align}\label{evaluation-module}
\text{ev}:U_{q}^{MO}(\hat{\mf{g}}_{Q})\rightarrow\prod_{\mbf{w}}K(\mbf{w}),\qquad F\mapsto Fv_{\varnothing}
\end{align}

\begin{defn}
We define the \textbf{positive half of the Maulik-Okounkov quantum loop groups} $U_{q}^{MO,+}(\hat{\mf{g}}_{Q})$ as the quotient by the kernel of the evaluation map:
\begin{align} 
U_{q}^{MO,+}(\hat{\mf{g}}_{Q}):=U_{q}^{MO}(\hat{\mf{g}}_{Q})/\text{Ker}(ev)
\end{align}
\end{defn}
\end{comment}

\subsubsection{\textbf{Wall subalgebra}}\label{ssub:_textbf_wall_subalgebra}

It is known that and has been proved \cite{OS22} the $K$-theoretic stable envelope $\text{Stab}_{s}$ is locally constant on $s\in\text{Pic}(X)\otimes\mbb{Q}$. It changes as $s$ crosses certain ratioanl hyperplanes:
\begin{prop}[See \cite{OS22} or \cite{Z24-2} for a detailed proof]\label{wall-set:proposition}
The $K$-theoretic stable envelope $\text{Stab}_{\mf{C},s}$ is locally constant on $s$ if and only if $s$ crosses the following hyperplanes.
\begin{align*}
w=\{s\in\text{Pic}(X)\otimes\mbb{R}|(s,\alpha)+n=0,\forall\alpha\in\text{Pic}(X)\}\subset\text{Pic}(X)\otimes\mbb{R}.
\end{align*}
\end{prop}

Now for each quiver variety $\mc{M}_{Q}(\mbf{v},\mbf{w})$, we can associate the wall set $w(\mbf{v},\mbf{w})$, it is clear that via the identification $\text{Pic}(\mc{M}_{Q}(\mbf{v},\mbf{w}))\otimes\mbb{R}=\mbb{R}^n$ for different $\mbf{v}$, we can define the \textbf{wall set} of $\mc{M}_{Q}(\mbf{w}):=\sqcup_{\mbf{v}}\mc{M}_{Q}(\mbf{v},\mbf{w})$ as $w(\mbf{w}):=\sqcup w(\mbf{v},\mbf{w})$. 

In many cases we will also denote the wall $w$ as $(\alpha,n)\in\text{Pic}(X)^+\times\mbb{Z}\cong\mbb{N}^I\times\mbb{Z}$.

Now fix the slope $\mbf{m}$ and the cocharacter $\sigma:\mbb{C}^*\rightarrow T_{\mbf{w}}$ such that $\mbf{w}=\mbf{w}_1+a\mbf{w}_2$, and we denote the corresponding torus as $A$. We choose an ample line bundle $\mc{L}\in\text{Pic}(X)$ with $X=\mc{M}_{Q}(\mbf{v},\mbf{w}_1+\mbf{w}_2)$ and a suitable small positive number $\epsilon$ such that $\mbf{m}-\epsilon\mc{L}$ and $\mbf{m}+\epsilon\mc{L}$ are separated by just one wall $w$ with $\mbf{m}$ being contained in $w$, we define the \textbf{wall $R$-matrices} as:
\begin{equation}\label{wall-R-matrix-definition}
\begin{aligned}
&R_{w}^{\pm}:=\text{Stab}_{\pm\sigma,\mbf{m}+\epsilon\mc{L}}^{-1}\circ\text{Stab}_{\pm\sigma,\mbf{m}-\epsilon\mc{L}}:\bigoplus_{\mbf{v}_1+\mbf{v}_2=\mbf{v}}K_{T_{\mbf{w}}}(\mc{M}_{Q}(\mbf{v}_1,\mbf{w}_1))\otimes K_{T_{\mbf{w}}}(\mc{M}_{Q}(\mbf{v}_2,\mbf{w}_2))\rightarrow\\
&\rightarrow\bigoplus_{\mbf{v}_1+\mbf{v}_2=\mbf{v}}K_{T_{\mbf{w}}}(\mc{M}_{Q}(\mbf{v}_1,\mbf{w}_1))\otimes K_{T_{\mbf{w}}}(\mc{M}_{Q}(\mbf{v}_2,\mbf{w}_2)).
\end{aligned}
\end{equation}

It is an integral $K$-theory class in $K_{T}(X^A\times X^A)$. Note that the choice of $\epsilon$ depends on $\mc{M}_{Q}(\mbf{v},\mbf{w}_1+\mbf{w}_2)$ just to make sure that there is only one wall between $\mbf{m}-\epsilon\mc{L}$ and $\mbf{m}+\epsilon\mc{L}$ corresponding to $w$. By definition it is easy to see that $R_{w}^{+}$ is upper-triangular , and $R_{w}^{-}$ is lower triangular, with respect to the partial ordering on the fixed-point component, i.e. If we decompose $R_{w}^{\pm}=\text{Id}+\sum_{\mbf{n}\in\mbb{N}^I}R_{w,\pm\mbf{n}}^{\pm}$, we have that:
\begin{align}\label{triangular-rule}
R_{w,\pm\mbf{n}}^{\pm}:K_{T_{\mbf{w}_1}}(\mc{M}_{Q}(\mbf{v}_1,\mbf{w}_1))\otimes K_{T_{\mbf{w}_2}}(\mc{M}_{Q}(\mbf{v}_2,\mbf{w}_2))\rightarrow K_{T_{\mbf{w}_1}}(\mc{M}_{Q}(\mbf{v}_1\mp\mbf{n},\mbf{w}_1))\otimes K_{T_{\mbf{w}_2}}(\mc{M}_{Q}(\mbf{v}_2\pm\mbf{n},\mbf{w}_2)).
\end{align}

Note that the definition of $R_{w}^{\pm}$ still depends on the choice of the slope points $\mbf{m}$, but over here we neglect $\mbf{m}$ for simplicity.

If we denote the operator $q^{\Omega}:K_{T_{\mbf{w}}}(X^A)\rightarrow K_{T_{\mbf{w}}}(X^A)$ by $\Omega=\frac{1}{4}\text{codim}(X^A)$. It has been proved in \cite{OS22} that $q^{\Omega}R_{w}^{\pm}$ and $R_{w}^{\pm}q^{\Omega}$ satisfies the Yang-Baxter equation:
\begin{equation}\label{Yang-Baxter-equation-for-wall-R-matrices}
\begin{aligned}
&(q^{\Omega}R_{w}^{\pm})_{12}(q^{\Omega}R_{w}^{\pm})_{13}(q^{\Omega}R_{w}^{\pm})_{23}=(q^{\Omega}R_{w}^{\pm})_{23}(q^{\Omega}R_{w}^{\pm})_{13}(q^{\Omega}R_{w}^{\pm})_{12}\\
&(R_{w}^{\pm}q^{\Omega})_{12}(R_{w}^{\pm}q^{\Omega})_{13}(R_{w}^{\pm}q^{\Omega})_{23}=(R_{w}^{\pm}q^{\Omega})_{23}(R_{w}^{\pm}q^{\Omega})_{13}(R_{w}^{\pm}q^{\Omega})_{12}.
\end{aligned}
\end{equation}

Also one can compute that $q^{\Omega}$ generates the subalgebra $\mbb{Z}[q^{\pm1},t_{e}^{\pm1}]_{e\in E}[h_{i,\pm 0}^{\pm1}]_{i\in I}$ for the Cartan part of the slope subalgebra in \ref{extended-half-slope-subalgebra} and coincides with the one mentioned below in \ref{decomposition-of-universal-R-matrix}.

It has also been proved in \cite{OS22} that the wall $R$-matrices are monomial in spectral parametre $a$:
\begin{align}\label{monomiality-for-wall-R-matrices}
R_{w}^{\pm}|_{F_2\times F_1}=
\begin{cases}
1&F_1=F_2\\
(\cdots)a^{\langle\mu(F_2)-\mu(F_1),\mbf{m}\rangle}&F_1\geq F_2\text{ or }F_1\leq F_2\\
0&\text{Otherwise}.
\end{cases}
\end{align}

Here $(\cdots)$ is some element in $K_{T_{\mbf{w}}/A}(X^A)$, and $\mu$ is a locally constant map $\mu:X^A\rightarrow H_2(X,\mbb{Z})\otimes A^{\wedge}$ defined up to an overall translation such that $\mu(F_1)-\mu(F_2)=[C]\otimes v$ with $C$ an irreducible curve joining $F_1$ and $F_2$ with tangent weight $v$ at $F_1$. Usually it is convenient for us to choose $A$ to be one-dimensional torus such that $A^{\wedge}\cong\mbb{Z}$. 

In the case of the wall $R$-matrices $R_{w}^{\pm}$, $\mu(F_2)-\mu(F_1)$ corresponds to $\pm k\alpha$ with $\alpha$ being the root corresponding to the wall $w$. In this case the inner product $\langle\mu(F_1)-\mu(F_2),\mbf{m}\rangle$ is equal to $\pm k\mbf{m}\cdot\alpha$ with respect to the torus action $\mbf{w}=a\mbf{w}_1+\mbf{w}_2$ and equal to $\mp k\mbf{m}\cdot\alpha$ with respect to the torus action $\mbf{w}=\mbf{w}_1+a\mbf{w}_2$.

Given $V_0=\bigotimes_{\mbf{w}_0\in I}K_{T_{\mbf{w}}}(\mc{M}_{Q}(\mbf{w}_0))$ as before with $I$ a finite subset of dimension vectors in $\mbb{N}^I$, and a finite-rank operator $m\in\text{End}(V_0)$. We define the \textbf{positive half of the wall subalgebra} $U_{q}^{MO,+}(\mf{g}_{w})$ as the $\mbb{Q}(q,t_e)_{e\in E}$-algebra generated by the operators:
\begin{align}\label{positive-wall-generators}
\text{Tr}_{V_0}((m\otimes 1)(R^{+}_{w})_{V_0,V}|_{a_0=1})\in\text{End}(V)
\end{align}
or on the other hand, generated by the matrix coefficients written in the following way:
\begin{align*}
\langle\mbf{y},R^{+}_{w}\mbf{x}\rangle,\forall\mbf{x},\mbf{y}\in V_0.
\end{align*}

Similarly, the \textbf{negative half of the wall subalgebra} $U_{q}^{MO,-}(\mf{g}_{w})$ is defined as the algebra generated by the operators:
\begin{align}\label{negative-wall-generators}
\text{Tr}_{V_0}((m\otimes 1)(R^{-}_{w})_{V_0,V}|_{a_0=1})\in\text{End}(V),\qquad V=K_{T_{\mbf{w}}}(\mc{M}_{Q}(\mbf{w}))
\end{align}
or on the other hand, generated by the matrix coefficients written in the following way:
\begin{align*}
\langle\mbf{y},R^{-}_{w}\mbf{x}\rangle,\forall\mbf{x},\mbf{y}\in V_0.
\end{align*}

Since $R_{w}^{\pm}$ is a strictly upper(lower) triangular matrix, we can see that the algebra $U_{q}^{MO,\pm}(\mf{g}_{w})$ is generated by the operators $m$ such that:
\begin{align*}
m:K_{T_{\mbf{w}}}(\mc{M}_{Q}(\mbf{v},\mbf{w}))\rightarrow K_{T_{\mbf{w}}}(\mc{M}_{Q}(\mbf{v}\pm\mbf{n},\mbf{w})),\qquad\mbf{n}\in\mbb{N}^I.
\end{align*}

Since the operator $q^{\Omega}R_{w}^{\pm}$ satisfy the Yang-Baxter equation. In this way, similarly we can define the non-negative(non-positive) half of the wall subalgebra $U_{q}^{MO,\geq(\leq)}(\mf{g}_{w})$ and also the whole wall subalgebra $U_{q}^{MO}(\mf{g}_{w})$. Obviously we have $U_{q}^{MO,\pm}(\mf{g}_{w})\subset U_{q}^{MO,\geq(\leq)}(\mf{g}_{w})$.

One could define the graded pieces of $U_{q}^{MO}(\mf{g}_w)$ as:
\begin{align*}
U_{q}^{MO,\pm}(\mf{g}_w)=\bigoplus_{\mbf{n}\in\mbb{N}^{I}}U_{q}^{MO,\pm}(\mf{g}_{w})_{\pm\mbf{n}}
\end{align*}
with $a\in U_{q}^{MO,\pm}(\mf{g}_{w})_{\pm\mbf{n}}$ the elements such that $a:K(\mbf{v},\mbf{w})\rightarrow K(\mbf{v}\pm\mbf{n},\mbf{w})$.

In the following context, we will denote $\pm\mbf{n}$ in the graded pieces as the horizontal degree, which is the same as the terminology for the shuffle algebras in \ref{horizontal-vertical-degree}.

\subsubsection{Integral wall subalgebra}
Since the wall $R$-matrices $R_{w}^{\pm}$ are the integral $K$-theory, we can still use the above definition of generators \ref{positive-wall-generators} and \ref{negative-wall-generators} in the integral $K$-theory $\text{End}_{K_{T_{\mbf{w}}}(pt)}(K_{T_{\mbf{w}}}(\mc{M}_{Q}(\mbf{w})))$. Similarly we can define the integral wall subalgebra $U_{q}^{MO,\mbb{Z}}(\mf{g}_{w})$ as the $\mbb{Z}[q^{\pm1},t_{e}^{\pm1}]_{e\in E}$-subalgebra of $\prod_{\mbf{w}}\text{End}_{K_{T_{\mbf{w}}}(pt)}(K_{T_{\mbf{w}}}(\mc{M}_{Q}(\mbf{w})))$ generated by the integral version of the generators in the formulas \ref{positive-wall-generators} and \ref{negative-wall-generators}. Also \ref{positive-wall-generators} will give an integral positive wall subalgebra $U_{q}^{MO,+,\mbb{Z}}(\mf{g}_{w})$, and \ref{negative-wall-generators} will give an integral negative wall subalgebra $U_{q}^{MO,-,\mbb{Z}}(\mf{g}_{w})$.

In this case the integral wall subalgebra $U_{q}^{MO,\mbb{Z}}(\mf{g}_{w})$ admits the decomposition:
\begin{align*}
U_{q}^{MO,\mbb{Z}}(\mf{g}_{w})=U_{q}^{MO,+,\mbb{Z}}(\mf{g}_{w})\otimes U_{q}^{MO,0,\mbb{Z}}(\mf{g}_{w})\otimes U_{q}^{MO,-,\mbb{Z}}(\mf{g}_{w})
\end{align*}
and here $U_{q}^{MO,0,\mbb{Z}}(\mf{g}_{w})$ is the subalgebra generated by $q^{\Omega}$. Each positive and negative half are actually $\mbb{N}^I$-graded:
\begin{align*}
U_{q}^{MO,\pm,\mbb{Z}}(\mf{g}_{w})=\bigoplus_{\mbf{n}\in\mbb{N}^I}U^{MO,\pm,\mbb{Z}}(\mf{g}_{w})_{\pm\mbf{n}}
\end{align*}
such that elements in $U^{MO,\pm,\mbb{Z}}(\mf{g}_{w})_{\pm\mbf{n}}$ send $K_{T_{\mbf{w}}}(\mc{M}_{Q}(\mbf{v},\mbf{w}))\rightarrow K_{T_{\mbf{w}}}(\mc{M}_{Q}(\mbf{v}\pm\mbf{n},\mbf{w}))$. 

\textbf{Remark. }Note that in general the wall subalgebra $U_{q}^{MO,\mbb{Z}}(\mf{g}_{w})$ is also dependent on the choice of the slope $\mbf{m}$. Thus rigorously one should write down the wall subalgebra as $U_{q}^{MO,\mbb{Z}}(\mf{g}_{\mbf{m},w})$.

Also note that for the special slope points $\mbf{m}$, it might be possible that there might be multiple walls between $\mbf{m}-\epsilon\mc{L}$ and $\mbf{m}+\epsilon\mc{L}$. In this case the corresponding subalgebra will be denoted by $U_{q}^{MO,\mbb{Z}}(\mf{g}_{\mbf{m}})$. This algebra can be thought of as the algebra generated by the wall subalgebra $U_{q}^{MO}(\mf{g}_{w})$ such that $\mbf{m}$ contains the wall $w$.

\subsection{Factorisation of geometric $R$-matrices and integral Maulik-Okounkov quantum loop groups}\label{subsection:factorisation_of_geometric_r_matrices_and_integral_maulik_okounkov_quantum_affine_algebras}
\subsubsection{KT-type factorisation for the geometric $R$-matrix}\label{subsubsection:factorisation_of_geometric_r_matrices_and_integral_maulik_okounkov_quantum_affine_algebras}
Fix the stable envelope $\text{Stab}_{\sigma,\mbf{m}}$ and $\text{Stab}_{\sigma,\infty}$, we can have the following factorisation of $\text{Stab}_{\pm,\mbf{m}}$ near $a=0,\infty$:
\begin{equation}\label{factorisation-stable-envelopes-negative}
\begin{aligned}
\text{Stab}_{\sigma,\mbf{m}}=&\text{Stab}_{\sigma,-\infty}\cdots\text{Stab}_{\sigma,\mbf{m}_{-2}}\text{Stab}_{\sigma,\mbf{m}_{-2}}^{-1}\text{Stab}_{\sigma,\mbf{m}_{-1}}\text{Stab}_{\sigma,\mbf{m}_{-1}}^{-1}\text{Stab}_{\sigma,\mbf{m}}\\
=&\text{Stab}_{\sigma,-\infty}\cdots R_{\mbf{m}_{-2},\mbf{m}_{-1}}^+R_{\mbf{m}_{-1},\mbf{m}}^+
\end{aligned}
\end{equation}

\begin{equation}\label{factorisation-stable-envelopes-positive}
\begin{aligned}
\text{Stab}_{-\sigma,\mbf{m}}=&\text{Stab}_{-\sigma,\infty}\cdots\text{Stab}_{-\sigma,\mbf{m}_2}\text{Stab}_{-\sigma,\mbf{m}_2}^{-1}\text{Stab}_{-\sigma,\mbf{m}_1}\text{Stab}_{-\sigma,\mbf{m}_1}^{-1}\text{Stab}_{-\sigma,\mbf{m}}\\
=&\text{Stab}_{-\sigma,\infty}\cdots R_{\mbf{m}_2,\mbf{m}_1}^{-}R_{\mbf{m}_1,\mbf{m}}^{-}.
\end{aligned}
\end{equation}

Here $\mbf{m}_i$ with $i<0$ are the points between $-\infty$. $\mbf{m}_i$ with $i>0$ are the points between $\infty$ and $\mbf{m}$. Also here $R_{\mbf{m}_1,\mbf{m}_2}^{\pm}=\text{Stab}_{\pm\sigma,\mbf{m}_1}^{-1}\text{Stab}_{\pm\sigma,\mbf{m}_2}$ is the wall $R$-matrix. For simplicity we always choose generic slope points $\mbf{m}_i$ such that there is only one wall between $\mbf{m}_1$ and $\mbf{m}_2$. In this case we use $R_{w}^{\pm}$ as $R_{\mbf{m}_1,\mbf{m}_2}^{\pm}$. 

Note that this notation does not mean that $R_{w}^{\pm}$ only depends on the wall $w$, but we still use the notation for simplicity.

These two factorisations \ref{factorisation-stable-envelopes-negative} and \ref{factorisation-stable-envelopes-positive} give the factorisation of the geometric $R$-matrix:
\begin{align}\label{factorisation-geometry}
\mc{R}^{s}(a)=\prod_{i<s}^{\leftarrow}R_{w_i}^{-}\mc{R}^{\infty}\prod_{i\geq s}^{\leftarrow}R_{w_i}^{+}
\end{align}
and here $\mc{R}^{\infty}$ is defined as:
\begin{align}\label{infinite-slope-R-matrix}
\mc{R}^{\infty}:=\text{Stab}_{-\sigma,\infty}^{-1}\circ\text{Stab}_{\sigma,\infty}=(\text{polarisation line bundle})\frac{\wedge^*\mc{N}^+_{X^{\sigma}/X}}{\wedge^*\mc{N}^-_{X^{\sigma}/X}}.
\end{align}
In the case such that $\sigma$ corresponds to $\mbf{w}=\mbf{w}'+a\mbf{w}''$, $X=\mc{M}_{Q}(\mbf{v},\mbf{w})$ and $F=\bigsqcup_{\mbf{v}'+\mbf{v}''=\mbf{v}}\mc{M}_{Q}(\mbf{v}',\mbf{w}')\times\mc{M}_{Q}(\mbf{v}'',\mbf{w}'')$. $\mc{R}^{\infty}$ can be written using the formula \ref{Normal-bundle-formula} and \ref{Positive-half-normal-bundle-formula}.

This has been proved in \cite{OS22} that this factorisation is well-defined in the topology of the Laurent formal power series in the spectral parametre $a$ around $0$ and $\infty$. We also understand $\mc{R}^{\infty}$ as the Laurent power series expansion of its rational form. In this way the formula \ref{factorisation-geometry} is in the formal completion $K_{T_{\mbf{w}}}(X^{A}\times X^A)[[a^{\pm1}]]$.

\subsubsection{Integral Maulik-Okounkov quantum loop groups}\label{subsubsection:integral_maulik_okounkov_quantum_affine_algebras}
The above Laurent series expansion implies that one can also use these integral $K$-theory coefficients to define an integral form of the Maulik-Okounkov quantum loop groups.
\begin{defn}
The integral form Maulik-Okounkov quantum loop group $U_q^{MO,\mbb{Z}}(\hat{\mf{g}}_{Q})$ is a $\mbb{Z}[q^{\pm1},t_e^{\pm1}]$-subalgebra of $\prod_{\mbf{w}}\text{End}_{K_{T_{\mbf{w}}}}(K_{T_{\mbf{w}}}(\mc{M}_{Q}(\mbf{w})))$ generated by the matrix coefficients of the geometric $R$-matrix in the factorised form as in \ref{factorisation-geometry}.
\end{defn}

From the definition one can also define the integral form Maulik-Okounkov quantum loop group $U_q^{MO,\mbb{Z}}(\hat{\mf{g}}_{Q})$ as the algebra generated by the Laurent expansion of the geometric $R$-matrix as in \ref{factorisation-geometry}. 

Here we denote $U_q^{MO,\mbb{Z},\pm}(\hat{\mf{g}}_{Q})$ are the subalgebra of positive and negative parts of $U_q^{MO,\mbb{Z}}(\hat{\mf{g}}_{Q})$ generated by the matrix coefficients of $\oint_{a_0=0,\infty}\frac{da_0}{2\pi ia_0}a_0^k\text{Tr}_{V_0}((1\otimes m)R^{\pm}_{w})$ for the arbitrary walls $w$. $U_q^{MO,\mbb{Z},0}(\hat{\mf{g}}_{Q})$ is the subalgebra generated by the Laurent expansion of $\mc{R}^{\infty}$, i.e. generated by the tautological classes.

\begin{prop}\label{triangular-decomposition:label}
The integral form MO quantum loop group admits the triangular decomposition
\begin{equation}\label{triangular-decomposition}
U_q^{MO,\mbb{Z}}(\hat{\mf{g}}_{Q})\cong U_q^{MO,\mbb{Z},\geq}(\hat{\mf{g}}_{Q})\otimes_{U_q^{MO,\mbb{Z},0}(\hat{\mf{g}}_{Q})}U_q^{MO,\mbb{Z},\leq}(\hat{\mf{g}}_{Q})
\end{equation}
here $U_{q}^{MO,\mbb{Z},\geq}(\hat{\mf{g}}_{Q}):=U_{q}^{MO,\mbb{Z},+}(\mf{g}_{Q})U_{q}^{MO,\mbb{Z},0}(\mf{g}_{Q})$ and $U_{q}^{MO,\mbb{Z},\leq}(\hat{\mf{g}}_{Q}):=U_{q}^{MO,\mbb{Z},0}(\mf{g}_{Q})U_{q}^{MO,\mbb{Z},-}(\mf{g}_{Q})$, and the isomorphism is given by the multiplication map on right hand side.
\end{prop}
\begin{proof}
Using the factorisation property \ref{factorisation-geometry} for the geometric $R$-matrix $\mc{R}^m$, and the matrix elements can be written as follows:
\begin{equation}
\begin{aligned}
&\oint\frac{da_0}{2\pi ia_0}a_0^l\text{Tr}_{V_0}((1\otimes m(a_0))\mc{R}^{\mbf{m}}_{V,V_0}(\frac{a}{a_0}))\\
=&\oint\frac{da_0}{2\pi ia_0}a_0^l\text{Tr}_{V_0}((1\otimes m(a_0))\prod_{i<\mbf{m}}^{\leftarrow}R_{w_i}^{-}\mc{R}^{\infty}\prod_{i\geq\mbf{m}}^{\leftarrow}R_{w_i}^{+}).
\end{aligned}
\end{equation}

From the integral one can see that each choice of $m(a_0)$ will give a formula of the elements in $U_{q}^{MO,\mbb{Z}}(\hat{\mf{g}}_{Q})$ as:
\begin{align}\label{PBW-generators}
\sum_{I}a_{I}E_IH_IF_I,\qquad E_I\in U_{q}^{MO,\mbb{Z},+}(\hat{\mf{g}}_{Q})_{<\mbf{m}},H_I\in U_{q}^{MO,\mbb{Z},0}(\hat{\mf{g}}_{Q}),F_I\in U_{q}^{MO,\mbb{Z},-}(\hat{\mf{g}}_{Q})_{\geq\mbf{m}}
\end{align}
and since $U_q^{MO,\mbb{Z},+}(\hat{\mf{g}}_{Q})_{<\mbf{m}}U_{q}^{MO,\mbb{Z},0}(\hat{\mf{g}}_{Q})$ and $U_q^{MO,\mbb{Z},0}(\hat{\mf{g}}_{Q})_{<\mbf{m}}U_{q}^{MO,\mbb{Z},-}(\hat{\mf{g}}_{Q})$ is the same as $U_q^{MO,\mbb{Z},\geq}(\hat{\mf{g}}_{Q}), U_q^{MO,\mbb{Z},\leq}(\hat{\mf{g}}_{Q})$ respectively. The proof is finished.

\end{proof}

\begin{comment}
\begin{lem}
$U_{q}^{MO,+}(\hat{\mf{g}}_{Q})$ is generated by the positive half of the wall subalgebras $U_{q}^{MO,+}(\mf{g}_{w})$ for arbitrary walls $w$. .
\end{lem}
\begin{proof}
Note that the elements of $U_{q}^{MO}(\hat{\mf{g}}_{Q})$ are of the form:
\begin{align}
\oint_{a_0=0,\infty}\frac{da_0}{2\pi ia_0}\text{Tr}_{V_0}((1\otimes \mc{M}_{Q}(a_0))\mc{R}^{s}_{V,V_0}(\frac{a}{a_0}))\in\text{End}(V(a))
\end{align}
Now via the mapping $a\mapsto av_{\varnothing}$ we have that:
\begin{equation}
\begin{aligned}
&\oint_{a_0=0,\infty}\frac{da_0}{2\pi ia_0}\text{Tr}_{V_0}((1\otimes \mc{M}_{Q}(a_0))\mc{R}^{s}_{V,V_0}(\frac{a}{a_0}))v_{\varnothing}\\
=&\oint_{a_0=0,\infty}\frac{da_0}{2\pi ia_0}\text{Tr}_{V_0}((1\otimes \mc{M}_{Q}(a_0))\prod_{i<0}^{\leftarrow}R_{w_i}^{-}R_{\infty}\prod_{i\geq0}^{\leftarrow}R_{w_i}^{+})v_{\varnothing}\\
=&\oint_{a_0=0,\infty}\frac{da_0}{2\pi ia_0}\text{Tr}_{V_0}((1\otimes \mc{M}_{Q}(a_0))\prod_{i<0}^{\leftarrow}R_{w_i}^{-}v_{\varnothing}
\end{aligned}
\end{equation}

which implies that the lemma.
\end{proof}
\end{comment}

The MO quantum loops groups $U_{q}^{MO,\mbb{Z}}(\hat{\mf{g}}_{Q})$ can be thought of as an integral form of the quantum affine algebras with central charge being trivial of the quiver type $Q$. Moreover, one can think of the geometric $R$-matrix $\mc{R}^{s}(a)$ as the evaluation of the \textbf{universal R-matrix} $R^{s,MO}\in U_{q}^{MO,\mbb{Z}}(\hat{\mf{g}}_{Q})\hat{\otimes}U_{q}^{MO,\mbb{Z}}(\hat{\mf{g}}_{Q})$ with respect to the coproduct $\Delta^{MO}_{s}$ which satisfies the following properties:
\begin{itemize}
	\item It satisfies the Yang-Baxter equation:
	\begin{align}
	R^{s,MO}_{12}R^{s,MO}_{13}R^{s,MO}_{23}=R^{s,MO}_{23}R^{s,MO}_{13}R^{s,MO}_{12}\in U_{q}^{MO,\mbb{Z}}(\hat{\mf{g}}_{Q})\hat{\otimes}U_{q}^{MO,\mbb{Z}}(\hat{\mf{g}}_{Q})\hat{\otimes}U_{q}^{MO,\mbb{Z}}(\hat{\mf{g}}_{Q}).
	\end{align}
	\item It is admits the inverse and satisfies the unitarity condition:
	\begin{align}
	R^{s,MO}_{21}=(R^{s,MO})^{-1}\in U_{q}^{MO,\mbb{Z}}(\hat{\mf{g}}_{Q})\hat{\otimes}U_{q}^{MO,\mbb{Z}}(\hat{\mf{g}}_{Q}).
	\end{align}
\end{itemize}

Similarly, one can also think of $q^{\pm\Omega}R_{w}^{\pm}$ as the \textbf{universal R-matrix} of the integral wall subalgebra $U_{q}^{MO,\mbb{Z}}(\mf{g}_{w})$ satisfies the Yang-Baxter equation as written in \ref{Yang-Baxter-equation-for-wall-R-matrices} such that the reduced part $R_{w}^{\pm}\in U_{q}^{MO,\mbb{Z},\mp}(\mf{g}_{w})\hat{\otimes}U_{q}^{MO,\mbb{Z},\pm}(\mf{g}_{w})$ is upper-triangular or lower-triangular respectively.

Similar to Theorem \ref{Main-theorem-on-slope-factorisation:Theorem}, we also have the factorisation for the positive and negative half of the MO quantum loop group after localisation:
\begin{prop}\label{factorisation-for-wall-subalgebras:label}
As the $\mbb{Z}[q^{\pm1},t_e^{\pm1}]_{e\in E}$-modules, the multiplication map gives the isomorphism:
\begin{equation}\label{multiplication-map}
\begin{tikzcd}
\bigotimes_{t\in\mbb{Q},\mbf{m}+t\bm{\theta}\in w}^{\rightarrow}U_{q}^{MO,\pm,\mbb{Z}}(\mf{g}_{\mbf{m}+t\bm{\theta},w})\arrow[r,"\cong"]&U_{q}^{MO,\pm,\mbb{Z}}(\hat{\mf{g}}_{Q}).
\end{tikzcd}
\end{equation}
\end{prop}
\begin{proof}
We will prove the proposition for the positive half, and the proof for the negative half is similar. 

It is enough to prove it after localisation. Note that since $U_{q}^{MO,\pm}(\hat{\mf{g}}_{Q})$ is generated by the matrix coefficients of $R_{\mbf{m},\infty}^{+}:=\text{Stab}_{\infty}^{-1}\circ\text{Stab}_{\mbf{m}}=\prod^{\rightarrow}_{w\leq\mbf{m}}R_{w}^{\pm}$. Thus the image of the multiplication map is just:
\begin{align}
\langle v',R_{\mbf{m},\infty}^{+}v\rangle\in\text{End}(K(\mbf{w}'))
\end{align}
for some vectors $v,v'\in K(\mbf{w})$, and here the pairing is given by \ref{perfect-pairing:label}. This gives the surjectivity of the map.

For the injectivity, it is equivalent to prove the following: Without loss of generality, we fix the vector $v\in K(\mbf{w})$. If for arbitrary vectors $v'\in K(\mbf{w})$ and $w,w'\in K(\mbf{w}')$, we have
\begin{align}\label{explicit-matrix-coefficients-for-positive-half}
\langle v'\otimes w',R_{\mbf{m},\infty}^{+}(v\otimes w)\rangle=0
\end{align}
then $v=0$.

Note that by definition of $R_{\mbf{m},\infty}^{+}$, the above formula \ref{explicit-matrix-coefficients-for-positive-half} can be written as:
\begin{align}
\langle\text{Stab}_{\infty,-\mf{C},T^{1/2}_{op}}(v'\otimes w'),\text{Stab}_{\mbf{m},\mf{C},T^{1/2}}(v\otimes w)\rangle=0
\end{align}
and now since the stable envelope is an isomorphism after localisation, the left-hand side of the bracket can be represented by arbitrary vectors $u'\in K(\mbf{w}+\mbf{w}')$. While we know that the bracket is a perfect pairing by \ref{perfect-pairing:label}. This implies that $\text{Stab}_{\mbf{m},\mf{C},T^{1/2}}(v\otimes w)=0$, which implies that $v=0$.

\end{proof}

For the positive/negative half $U_{q}^{MO,\pm}(\hat{\mf{g}}_{Q})$, note that they are generated by the matrix coefficients of $R_{\mbf{m},\infty}^{\pm}=\text{Stab}_{\pm\sigma,\infty}^{-1}\circ\text{Stab}_{\pm\sigma,\mbf{m}}$ for arbitrary $\mbf{m}\in\mbb{Q}^I$. Moreover, the generators for $U_{q}^{MO,\pm}(\hat{\mf{g}}_{Q})$ can be simpler:
\begin{lem}\label{generators-for-positive-negative-half:label}
$U_{q}^{MO,\pm}(\hat{\mf{g}}_{Q})$ is generated by $\langle v_{\emptyset},R_{\mbf{m},\infty}^{+}v\rangle$ and $\langle v,(R_{\mbf{m},\infty}^{-})^{-1}v_{\emptyset}\rangle$ for arbitrary $v\in K(\mbf{w})$ respectively.
\end{lem}
\begin{proof}
We are showing the proof for the positive, and the proof for the negative half is similar.

Given $E\in U_{q}^{MO,+}(\hat{\mf{g}}_{Q})$, and we suppose that it can be represented by $\langle v',\mc{R}^{\mbf{m}}v\rangle$ for some vectors $v,v'\in K(\mbf{w})$. By Theorem \ref{surjectivity-noetherian-lemma-localised:theorem}, we can express $v'$ as the covector $v_{\emptyset}\circ f$ for some $f\in\mc{A}^{-}_{Q}$ and therefore the above element can be written as:
\begin{align*}
\langle v_{\emptyset},(f\otimes\text{Id})\mc{R}^{\mbf{m}}v\rangle.
\end{align*}

Then we can also represent $f$ as $\langle w',(\mc{R}^{\mbf{m}})w\rangle$ for some vectors $w,w'\in K(\mbf{w}')$. Thus we can further expand the above formula as: 
\begin{equation*}
\begin{aligned}
&\langle  v_{\emptyset}\otimes w',(\mc{R}^{\mbf{m}})_{12}(\mc{R}^{\mbf{m}})_{13}v\otimes w\rangle\\
=&\langle v_{\emptyset}\otimes(-)\otimes w'(\mc{R}^{\mbf{m}})_{13}(\mc{R}^{\mbf{m}})_{12}v\otimes(-)\otimes w\rangle\\
=&\langle v_{\emptyset}\otimes(-)\otimes w'(1\otimes\Delta_{\mbf{m}})(\mc{R}^{\mbf{m}})v\otimes(-)\otimes w\rangle\\
=&\langle v_{\emptyset}\otimes(-)\otimes w'(1\otimes\text{Stab}_{\mbf{m}}^{-1})(\mc{R}^{\mbf{m}})(1\otimes\text{Stab}_{\mbf{m}})v\otimes(-)\otimes w\rangle.
\end{aligned}
\end{equation*}

Thus we can use the stable envelope to replace the vectors in the bracket $(-)$ by some new vectors $\text{Stab}_{\mbf{m}}((-)\otimes w)$ and covectors $(-)\otimes w'(\text{Stab}_{\mbf{m}}^{-1})$. In this way one can see that it is now reduced to compute the matrix coefficients of the form:
\begin{align}
\langle v_{\emptyset},\mc{R}^{\mbf{m}}v\rangle.
\end{align}

Using the factorisation property, we have that:
\begin{equation}
\begin{aligned}
\langle v_{\emptyset},\mc{R}^{\mbf{m}}v\rangle=&\langle v_{\emptyset},\prod^{\leftarrow}_{w<\mbf{m}}R_{w}^{-}\mc{R}^{\infty}\prod_{w\geq\mbf{m}}R_{w}^+v\rangle\\
=&\text{(tautological classes)}\circ\langle v_{\emptyset},\prod_{w\geq\mbf{m}}R_{w}^+v\rangle.
\end{aligned}
\end{equation}

Thus if we forget the tautological classes, we can see that the part of $\langle v_{\emptyset},\prod_{w\geq\mbf{m}}R_{w}^+v\rangle$ gives the generators of $U_{q}^{MO,+}(\hat{\mf{g}}_{Q})$.

\end{proof}

\subsection{Freeness of wall subalgebras}\label{subsection:freeness_of_wall_subalgebras}
In this section we prove the freeness of the wall subalgebras, and this is one of the key aspects of the MO quantum loop groups.
\begin{thm}\label{freeness-of-wall-subalgebra:theorem}
The wall subalgebra $U_{q}^{MO,\mbb{Z}}(\mf{g}_{w})$ is an $\mbb{N}^{I}$-graded finitely generated free $\mbb{Z}[q^{\pm1},t_{e}^{\pm1}]_{e\in E}$-module
\end{thm}
\begin{proof}
The proof follows the strategy from \cite{MO19}. It is $\mbb{N}^{I}$-graded finitely generated since the wall $R$-matrices are integral $K$-theory and $K_{T_{\mbf{w}}}(\mc{M}_{Q}(\mbf{v},\mbf{w}))$ is a finitely-generated free $K_{T_{\mbf{w}}}(pt)$-module.  Then note that imitating the proof of Theorem \ref{primitivity:Theorem}, one can first show the following the wall subalgebra is also generated by the primitive elements in the sense of \ref{primitive-element-definition}.
\begin{prop}\label{primitivity-of-wall-subalgebra:proposition}
$U_{q}^{MO,\mbb{Z}}(\mf{g}_{w})$ is generated by the primitive elements.
\end{prop}
\textbf{Remark.} The proof of the above Proposition can be reduced to the localised case, since primitivity is independent of being localised or not.

Moreover since $U_{q}^{MO,\mbb{Z}}(\mf{g}_{w})$
has the triangular decomposition:
\begin{align*}
U_{q}^{MO,\mbb{Z}}(\mf{g}_{w})=U_{q}^{MO,+,\mbb{Z}}(\mf{g}_{w})\otimes U_{q}^{MO,0,\mbb{Z}}(\mf{g}_{w})\otimes U_{q}^{MO,-,\mbb{Z}}(\mf{g}_{w}).
\end{align*}
Without loss of generality, we will only show that $U_{q}^{MO,+,\mbb{Z}}(\mf{g}_{w})$ is a free $\mbb{Z}[q^{\pm1},t_{e}^{\pm1}]_{e\in E}$-module. Since $U_{q}^{MO,\mbb{Z}}(\mf{g}_{w})$ is graded by $\mbf{n}\in\mbb{N}^{I}$, it remains to show that the graded piece $U_{q}^{MO,+,\mbb{Z}}(\mf{g}_{w})_{\mbf{n}}$ is a free $\mbb{Z}[q^{\pm1},t_{e}^{\pm1}]_{e\in E}$-module.

We denote $U_{q}^{MO,prim,\mbb{Z}}(\mf{g}_{w})$ the submodule of $U_{q}^{MO,\mbb{Z}}(\mf{g}_{w})$ as a $\mbb{Z}[q^{\pm1},t_{e}^{\pm1}]_{e\in E}$-module. Obviously it can be graded as $U_{q}^{MO,prim,\mbb{Z}}(\mf{g}_{w})_{\mbf{n}}$ with $\mbf{n}\in\mbb{N}^I$.

First given $E\in U_{q}^{MO,+,prim,\mbb{Z}}(\mf{g}_{w})_{\mbf{v}}$, using the identity:
\begin{align*}
R_{w}^{-}q^{\Omega}\Delta_{\mbf{m}}^{MO,op}(E)=\Delta_{\mbf{m}}^{MO}(E)R_{w}^{-}q^{\Omega}.
\end{align*}
Since $\Delta^{MO,op}_{\mbf{m}}(E)=E\otimes h_{\mbf{v}}+\text{Id}\otimes E$, we have that:
\begin{align*}
R_{w}^-q^{\Omega}(E\otimes h_{\mbf{v}}+\text{Id}\otimes E)=(h_{\mbf{v}}\otimes E+E\otimes\text{Id})R_{w}^-q^{\Omega}.
\end{align*}
Using the decomposition $R_{w}^{-}=\text{Id}+\sum_{\mbf{v}\in\mbb{N}^I}R_{w,-\mbf{v}}^-$, we have that:
\begin{align}\label{commutation-relation-for-primitive-elements}
[(1\otimes E),R_{w,-\mbf{v}}^{-}]=E\otimes (h_{\mbf{v}}-h_{-\mbf{v}}).
\end{align}

\begin{lem}\label{first-evaluation-injective-lemma:lemma}
The evaluation map 
\begin{align*}
U_{q}^{MO,+,prim,\mbb{Z}}(\mf{g}_{w})_{\mbf{v}}\otimes K_{T_{\mbf{w}}}(pt)\rightarrow K_{T_{\mbf{w}}}(\mc{M}_{Q}(\mbf{v},\mbf{w})),\qquad E\mapsto Ev_{\emptyset}
\end{align*}
is an injective map of $\mbb{Z}[q^{\pm1},t_{e}^{\pm1}]_{e\in E}$-modules.
\end{lem}
\begin{proof}
This follows from \ref{commutation-relation-for-primitive-elements} that given $E\in U_{q}^{MO,+,prim,\mbb{Z}}(g_{w})_{\mbf{v}}$ and $F\in U_{q}^{MO,-,prim,\mbb{Z}}(g_{w})_{-\mbf{v}}$, we have that:
\begin{align*}
FEv_{\emptyset}=(h_{-\mbf{v}}-h_{\mbf{v}})v_{\emptyset}\neq0.
\end{align*}
\end{proof}

\begin{lem}\label{second-evaluation-injective-lemma:lemma}
For the dimension vector $\mbf{w}\in\mbb{N}^I$ which is bigger or equal to $\mbf{v}$ the evaluation map
\begin{align*}
U_{q}^{MO,+,\mbb{Z}}(\mf{g}_{w})_{\mbf{v}}\otimes K_{T_{\mbf{w}}}(pt)\rightarrow K_{T_{\mbf{w}}}(\mc{M}_{Q}(\mbf{v},\mbf{w})),\qquad E\mapsto Ev_{\emptyset,\mbf{w}}
\end{align*}
is an injective map of $\mbb{Z}[q^{\pm1},t_{e}^{\pm1}]_{e\in E}$-modules.
\end{lem}
\begin{proof}
This can be proved using the induction on the horizontal degree $\mbf{v}\in\mbb{N}^I$. For $\mbf{v}$ being of the minimal degree $\mbf{v}_0$, since in this case $U_{q}^{MO,+,\mbb{Z}}(\mf{g}_{w})_{\mbf{v}_0}$ consists of primitive elements, it is injective by Lemma \ref{first-evaluation-injective-lemma:lemma}.

By Lemma \ref{primitivity-of-wall-subalgebra:proposition}, one can write down $E$ as follows:
\begin{align}
E=\sum_{\mbf{l}}a_{\mbf{l}}E_{l_1}\cdots E_{l_k},\qquad E_{l_i}\in U_{q}^{MO,prim,+,\mbb{Z}}(\mf{g}_{w}).
\end{align}

Then we consider the coproduct operation $(\Delta_{\mbf{m}}^{MO})^{q}(E)$ on $E$ such that $q$ is large enough (i.e. $q>\sum_{\mbf{l}}\lvert\mbf{l}\rvert$). In this way one can write down $(\Delta_{\mbf{m}}^{MO})^{q}(E)$ in the following way:
\begin{align}\label{large-coproduct-formula}
(\Delta_{\mbf{m}}^{MO})^{q}(E)=\sum_{i=1}^q h_{\mbf{v}+\mbf{v}_i}^{\otimes(i-1)}\otimes E\otimes\text{Id}^{\otimes(q-i)}+(\cdots).
\end{align}
Here $(\cdots)$ stands for resting terms in $\bigoplus_{\mbf{n}_j<\mbf{v}+\mbf{v}_i}U_{q}^{MO,prim,+,\mbb{Z}}(\mf{g}_{w})_{\mbf{n}_1}\otimes\cdots\otimes U_{q}^{MO,prim,+,\mbb{Z}}(\mf{g}_{w})_{\mbf{n}_q}$. One can see that even if for many $\mbf{w}\in\mbb{N}^I$ such that $Ev_{\emptyset,\mbf{w}}=0$, one can have the resting term in \ref{large-coproduct-formula} living in 
$$\bigoplus_{\substack{\mbf{n}_1+\cdots+\mbf{n}_q=\mbf{v}+\mbf{v}_i\\\mbf{n}_j<\mbf{v}+\mbf{v}_i}}U_{q}^{MO,+,\mbb{Z}}(\mf{g}_{w})_{\mbf{n}_1}\otimes\cdots\otimes U_{q}^{MO,+,\mbb{Z}}(\mf{g}_{w})_{\mbf{n}_q}$$ 
Using the induction, we have the embedding:
\begin{equation*}
\begin{aligned}
&\bigoplus_{\substack{\mbf{n}_1+\cdots+\mbf{n}_q=\mbf{v}+\mbf{v}_i\\\mbf{n}_j<\mbf{v}+\mbf{v}_i}}U_{q}^{MO,+,\mbb{Z}}(\mf{g}_{w})_{\mbf{n}_1}\otimes\cdots\otimes U_{q}^{MO,+,\mbb{Z}}(\mf{g}_{w})_{\mbf{n}_q}\\
\hookrightarrow&\bigoplus_{\substack{\mbf{n}_1+\cdots+\mbf{n}_q=\mbf{v}+\mbf{v}_i\\\mbf{n}_j<\mbf{v}+\mbf{v}_i}}K_{T_{\mbf{w}_1}}(\mc{M}_Q(\mbf{v}_1,\mbf{w}_1))\otimes\cdots\otimes K_{T_{\mbf{w}_q}}(\mc{M}_Q(\mbf{v}_q,\mbf{w}_q))
\end{aligned}
\end{equation*}
and we can choose $\mbf{w}_1,\cdots,\mbf{w}_q$ to make the resting term in \ref{large-coproduct-formula} acting on the vacuum to be nonzero. Thus 
now if we composite with the stable envelope map $\text{Stab}_{\mbf{m},\mf{C}}:\bigoplus_{\substack{\mbf{n}_1+\cdots+\mbf{n}_q=\mbf{v}+\mbf{v}_i\\\mbf{n}_j<\mbf{v}+\mbf{v}_i}}K_{T_{\mbf{w}_1}}(\mc{M}_Q(\mbf{v}_1,\mbf{w}_1))\otimes\cdots\otimes K_{T_{\mbf{w}_q}}(\mc{M}_Q(\mbf{v}_q,\mbf{w}_q))\rightarrow K_{T_{\mbf{w}}}(\mc{M}_{Q}(\mbf{v},\mbf{w}))$, the left-hand side of \ref{large-coproduct-formula} acting on $v_{\emptyset,\mbf{w}'}$ with $\mbf{w}'=\mbf{w}_1+\cdots+\mbf{w}_q$ will be $Ev_{\emptyset,\mbf{w}'}$, which will be nonzero since the stable envelope map is injective. 

Moreover, if we choose the framing dimension $\mbf{w}$ injective map given in Lemma \ref{first-evaluation-injective-lemma:lemma} as $\mbf{w}=\mbf{v}$, we get the injective map with $\mbf{w}=\mbf{v}$.
Thus we have finished the proof.
\end{proof}

Then we consider the operator $P_{w,\mbf{v}}:=q^{\Omega}R_{w,\mbf{w}}^-$ restricted to the component:
\begin{align*}
K_{T_{\mbf{w}}}(\mc{M}_{Q}(\mbf{0},\mbf{w}))\otimes K_{T_{\mbf{w}}}(\mc{M}_{Q}(\mbf{v},\mbf{w}))\rightarrow K_{T_{\mbf{w}}}(\mc{M}_{Q}(\mbf{v},\mbf{w}))\otimes K_{T_{\mbf{w}}}(\mc{M}_{Q}(\mbf{0},\mbf{w})).
\end{align*}

Using the Yang-Baxter equation \ref{Yang-Baxter-equation-for-wall-R-matrices} for $q^{\Omega}R_{w}^+$, one has the following result:
\begin{lem}\label{projection-operator:lemma}
One has the following relations:
\begin{align*}
P_{w,\mbf{v}}^2=q^{-\Omega}P_{w,\mbf{v}}\in\text{End}_{K_{T_{\mbf{w}}}(pt)}(K_{T_{\mbf{w}}}(\mc{M}_{Q}(\mbf{v},\mbf{w})))
\end{align*}
\end{lem}

Now if we do the sum over arbitrary dimension vector $\mbf{w}$, one has the diagonal operator $ P_{w,\mbf{v}}$ over $\bigoplus_{\mbf{w}}K_{T_{\mbf{w}}}(\mc{M}_{Q}(\mbf{v},\mbf{w}))$. Thus the following result is obvious:
\begin{lem}\label{freeness-of-positive-half:label}
$U^{MO,+,\mbb{Z}}_q(\mf{g}_{w})_{\mbf{v}}$ is isomorphic to the image of $P_{w,\mbf{v}}$ in $\bigoplus_{\mbf{w}}K_{T_{\mbf{w}}}(\mc{M}_{Q}(\mbf{v},\mbf{w}))$. Moreover, $U^{MO,+,\mbb{Z}}_q(\mf{g}_{w})_{\mbf{v}}$ is a projective $\mbb{Z}[q^{\pm1},t_{e}^{\pm1}]_{e\in E}$-module.
\end{lem}
\begin{proof}
Since $U^{MO,+,\mbb{Z}}_q(\mf{g}_{w})_{\mbf{v}}$ is generated by the matrix coefficients of $R_{w,\mbf{v}}^{+}$, this follows from Lemma \ref{second-evaluation-injective-lemma:lemma} and Lemma \ref{projection-operator:lemma} and the fact that $K_{T_{\mbf{w}}}(\mc{M}_{Q}(\mbf{v},\mbf{w}))$ is a free $K_{T_{\mbf{w}}}(pt)$-module.
\end{proof}
Combining these facts, we conclude that $U^{MO,+}_{q}(\mf{g}_{w})$ is an $\mbb{N}^I$-graded projective $\mbb{Z}[q^{\pm1},t_e^{\pm1}]$-module, and hence an $\mbb{N}^I$-graded free $\mbb{Z}[q^{\pm1},t_e^{\pm1}]$-module. Similar proof also applies for $U^{MO,-}_{q}(\mf{g}_{w})$. Thus the proof is finished.

\end{proof}

\subsection{Relations with the double of the preprojective $K$-theoretic Hall algebra}

In this subsection we prove that the stable envelope $\text{Stab}_{\pm\sigma,\infty}$ of the infinite slope intertwines the Drinfeld coproduct:
\begin{thm}\label{Commutativity-of-infty-coproduct-with-infty-stable-envelopes}
Denote $\text{Stab}_{\infty}:=\text{Stab}_{\sigma,\infty}$ and given $\forall F\in \mc{A}_{Q}$, the following diagram commute:
\begin{equation}
\begin{tikzcd}
K(\mbf{w}_1)\otimes K(\mbf{w}_2)\arrow[r,"\text{Stab}_{\infty}"]\arrow[d,"\Delta(F)"]&K(\mbf{w}_1+\mbf{w}_2)\arrow[d,"F"]\\
K(\mbf{w}_1)\otimes K(\mbf{w}_2)\arrow[r,"\text{Stab}_{\infty}"]&K(\mbf{w}_1+\mbf{w}_2).
\end{tikzcd}
\end{equation}
\end{thm}
\begin{proof}

The strategy of the proof follows from \cite{N15}. Since $\mc{A}_{Q}$ is generated by $\{e_i(z),f_i(z),\psi_i^{\pm}(z)\}_{i\in I}$. For now we use $\text{Stab}_{\infty}$, and we only need to prove the commutativity of the following diagram:
\begin{equation*}
\begin{tikzcd}
K(\mbf{w}_1)\otimes K(\mbf{w}_2)\arrow[r,"\text{Stab}_{\infty}"]\arrow[d,"\Delta(e_i(z))"]&K(\mbf{w}_1+\mbf{w}_2)\arrow[d,"e_i(z)"]\\
K(\mbf{w}_1)\otimes K(\mbf{w}_2)\arrow[r,"\text{Stab}_{\infty}"]&K(\mbf{w}_1+\mbf{w}_2)
\end{tikzcd}
\begin{tikzcd}
K(\mbf{w}_1)\otimes K(\mbf{w}_2)\arrow[r,"\text{Stab}_{\infty}"]\arrow[d,"\Delta(f_i(z))"]&K(\mbf{w}_1+\mbf{w}_2)\arrow[d,"f_i(z)"]\\
K(\mbf{w}_1)\otimes K(\mbf{w}_2)\arrow[r,"\text{Stab}_{\infty}"]&K(\mbf{w}_1+\mbf{w}_2).
\end{tikzcd}
\end{equation*}

In other words:
\begin{align*}
e_i(z)\cdot\text{Stab}_{\infty}(p_{\bm{\lambda}_1}\otimes p_{\bm{\lambda}_2})=\text{Stab}_{\infty}(\Delta(e_i(z))(p_{\bm{\lambda}_1}\otimes p_{\bm{\lambda}_2})).
\end{align*}

For simplicity we only prove the theorem for $e_{i}(z)$ and $\psi_{i}^{\pm}(z)$.

Recall that the action of $e_{i}(z)$ can be written as:
\begin{equation*}
\begin{aligned}
&e_{i}(z)\cdot p(\mbf{X}_{\mbf{v}})=\int\frac{dz_i}{2\pi\sqrt{-1}z_i}\delta(\frac{z_{i1}}{z})p(\mbf{X}_{\mbf{v}+\mbf{e}_i}-z_{i1})\tilde{\zeta}(\frac{z_{i1}}{\mbf{X}_{\mbf{v}+\mbf{e}_i}})\wedge^*(\frac{z_{i1}q}{\mbf{W}})\\
=&\tilde{\zeta}(\frac{z}{X_{\mbf{v}+\mbf{e}_i}})\wedge^*(\frac{zq}{\mbf{W}})p(\mbf{X}_{\mbf{v}+\mbf{e}_i}-z).
\end{aligned}
\end{equation*}

While we have the following normal bundle formula in terms of the tautological classes:

\begin{equation}\label{Normal-bundle-formula}
\begin{aligned}
\text{Nor}_{\mc{M}_{Q}(\mbf{v}',\mbf{w}')\times\mc{M}_{Q}(\mbf{v}'',\mbf{w}'')}(\mc{M}_{Q}(\mbf{v},\mbf{w}))=\sum_{e=ij\in E}(\frac{V_{j}'}{t_eV_{i}''}+\frac{t_eV_{i}'}{qV_{j}''}+\frac{V_{j}''}{t_eV_{i}'}+\frac{t_eV_{i}''}{qV_{j}'})\\
-\sum_{i\in I}(1+\frac{1}{q})(\frac{V_i'}{V_i''}+\frac{V_{i}''}{V_{i}'})+\sum_{i\in I}(\frac{V_i'}{W_{i}''}+\frac{W_{i}'}{qV_{i}''}+\frac{V_{i}''}{W_{i}'}+\frac{W_{i}''}{qV_{i}'}).
\end{aligned}
\end{equation}

Thus the negative half is written as:
\begin{equation}\label{Positive-half-normal-bundle-formula}
\begin{aligned}
\text{Nor}_{\mc{M}_{Q}(\mbf{v}',\mbf{w}')\times\mc{M}_{Q}(\mbf{v}'',\mbf{w}'')}^-(\mc{M}_{Q}(\mbf{v},\mbf{w}))=\sum_{e=ij\in E}(\frac{V_{j}''}{t_eV_{i}'}+\frac{t_eV_{i}''}{qV_{j}'})
-\sum_{i\in I}(1+\frac{1}{q})(\frac{V_i''}{V_i'})+\sum_{i\in I}(\frac{V_i''}{W_{i}'}+\frac{W_{i}''}{qV_{i}'}).
\end{aligned}
\end{equation}

Following the strategy in \cite{N23}, we choose a suitable polarisation such that the stable envelope with infinite slope $\text{Stab}_{\infty}|_{F\times F}$ with $F=\mc{M}_{Q}(\mbf{v}',\mbf{w}')\times\mc{M}_{Q}(\mbf{v}'',\mbf{w}'')$ can be written as :
\begin{equation}\label{formula-for-infty-stable-envelope}
\begin{aligned}
&\text{Stab}_{\infty}|_{F\times F}=q^{\frac{\mbf{w}''\cdot\mbf{v}'-\langle\mbf{v}'',\mbf{v}'\rangle}{2}}\frac{\prod_{e=ij\in E}\wedge^*(\frac{t_eV_i'}{V_j''})\wedge^*(\frac{qV_i''}{t_eV_j'})}{\prod_{i\in I}\wedge^*(\frac{V_i'}{V_i''})\wedge^*(\frac{V_i''}{qV_i'})}\prod_{i\in I}\wedge^*(\frac{V_i'}{W_i''})\wedge^*(\frac{qV_i''}{W_i'})\\
=&q^{\frac{\mbf{w}''\cdot\mbf{v}'-\langle\mbf{v}'',\mbf{v}'\rangle}{2}}\tilde{\zeta}(\frac{V_i'}{V_i''})\prod_{i\in I}\wedge^*(\frac{qV_i'}{W_{i}''})\wedge^*(\frac{V_{i}''}{W_{i}'}).
\end{aligned}
\end{equation}
Since the Cartan element $q^{\frac{\mbf{w}''\cdot\mbf{v}'-\langle\mbf{v}'',\mbf{v}'\rangle}{2}}$ will be eliminated in the computation, in the following computation we will ignore it.
Thus we have that:
\begin{align*}
\text{Stab}_{\infty}(p_{1}(\mbf{X}_{\mbf{v}_1})\otimes p_{2}(\mbf{X}_{\mbf{v}_2}))=\text{Sym}(\text{Stab}_{\infty}|_{F\times F}p_{1}(\mbf{X}_{\mbf{v}_1})p_{2}(\mbf{X}_{\mbf{v}_2})).
\end{align*}

Thus we have that:
\begin{equation}\label{stable-envelope-formula-infty-1}
\begin{aligned}
&e_{i}(z)\cdot\text{Stab}_{\infty}(p_{1}(\mbf{X}_{\mbf{v}_1})\otimes p_{2}(\mbf{X}_{\mbf{v}_2}))\\
=&\tilde{\zeta}(\frac{z}{X_{\mbf{v}_1+\mbf{v}_2+\mbf{e}_i}})\wedge^*(\frac{zq}{W_1})\wedge^*(\frac{zq}{W_2})\text{Sym}(\tilde{\zeta}(\frac{X_{\mbf{v}_1}}{X_{\mbf{v}_2+\mbf{e}_i}-z})\wedge^*(\frac{qX_{\mbf{v}_1}}{W_2})\wedge^*(\frac{(X_{\mbf{v}_2+\mbf{e}_i}-z)}{W_1})p_{1}(\mbf{X}_{\mbf{v}_1})p_{2}(\mbf{X}_{\mbf{v}_2+\mbf{e}_i}-z))\\
&+\tilde{\zeta}(\frac{z}{X_{\mbf{v}_1+\mbf{v}_2+\mbf{e}_i}})\wedge^*(\frac{zq}{W_1})\wedge^*(\frac{zq}{W_2})\text{Sym}(\tilde{\zeta}(\frac{X_{\mbf{v}_1+\mbf{e}_i}-z}{X_{\mbf{v}_2}})\prod_{i\in I}\wedge^*(\frac{q(X_{\mbf{v}_1+\mbf{e}_i}-z)}{W_2})\wedge^*(\frac{X_{\mbf{v}_2}}{W_1})p_{1}(\mbf{X}_{\mbf{v}_1+\mbf{e}_i}-z)p_{2}(\mbf{X}_{\mbf{v}_2})).
\end{aligned}
\end{equation}

Meanwhile:
\begin{equation}\label{stable-envelope-formula-infty-2}
\begin{aligned}
&\text{Stab}_{\infty}(\Delta(e_{i}(z))(p_{1}(\mbf{X}_{\mbf{v}_1})\otimes p_{2}(\mbf{X}_{\mbf{v}_2})))\\
=&\text{Stab}_{\infty}(e_i(z)p_{1}(\mbf{X}_{\mbf{v}_1})\otimes p_{2}(\mbf{X}_{\mbf{v}_2})+h_{i}^+(z)p_{1}(\mbf{X}_{\mbf{v}_1})\otimes e_i(z)p_{2}(\mbf{X}_{\mbf{v}_2}))\\
=&\text{Stab}_{\infty}(\tilde{\zeta}(\frac{z}{X_{\mbf{v}_1+\mbf{e}_i}})\wedge^*(\frac{zq}{W_1})p_1(X_{\mbf{v}_1+\mbf{e}_i}-z)\otimes p_2(X_{\mbf{v}_2}))\\
&+\text{Stab}_{\infty}(\frac{\tilde{\zeta}(\frac{z}{X_{\mbf{v}_1}})}{\tilde{\zeta}(\frac{X_{\mbf{v}_1}}{z})}\frac{\wedge^*(\frac{zq}{W_1})}{\wedge^*(\frac{z}{W_1})}\tilde{\zeta}(\frac{z}{X_{\mbf{v}_2+\mbf{e}_i}})\wedge^*(\frac{zq}{W_2})p_1(X_{\mbf{v}_1})\otimes p_2(X_{\mbf{v}_2+\mbf{e}_i}-z)).
\end{aligned}
\end{equation}
Now using the definition of $\text{Stab}_{\infty}$ in \ref{formula-for-infty-stable-envelope}, one can calculate that the formula \ref{stable-envelope-formula-infty-2} matches \ref{stable-envelope-formula-infty-1}.

For the Cartan current $\psi_{i}^{\pm}(z)$, by doing the computation on both sides:
\begin{equation*}
\begin{aligned}
\psi_{i}^{\pm}(z)\text{Stab}_{\infty}(p_{1}(X_{\mbf{v}_1})\otimes p_2(X_{\mbf{v}_2}))=\frac{\tilde{\zeta}(\frac{z}{X_{\mbf{v}_1+\mbf{v}_2}})}{\tilde{\zeta}(\frac{X_{\mbf{v}_1+\mbf{v}_2}}{z})}\frac{\wedge^*(\frac{zq}{W_1})\wedge^*(\frac{zq}{W_2})}{\wedge^*(\frac{z}{W_1})\wedge^*(\frac{z}{W_2})}\text{Stab}_{\infty}(p_{1}(X_{\mbf{v}_1})\otimes p_2(X_{\mbf{v}_2})).
\end{aligned}
\end{equation*}

\begin{equation*}
\begin{aligned}
&\text{Stab}_{\infty}(\Delta(\psi_{i}^{\pm}(z))p_{1}(X_{\mbf{v}_1})\otimes p_2(X_{\mbf{v}_2}))=\text{Stab}_{\infty}(\psi_{i}^{\pm}(z)\otimes \psi_{i}^{\pm}(z)(p_{1}(X_{\mbf{v}_1})\otimes p_2(X_{\mbf{v}_2})))\\
=&\text{Stab}_{\infty}(\frac{\tilde{\zeta}(\frac{z}{X_{\mbf{v}_1}})}{\tilde{\zeta}(\frac{X_{\mbf{v}_1}}{z})}\frac{\tilde{\zeta}(\frac{z}{X_{\mbf{v}_2}})}{\tilde{\zeta}(\frac{X_{\mbf{v}_2}}{z})}\frac{\wedge^*(\frac{zq}{W_1})\wedge^*(\frac{zq}{W_2})}{\wedge^*(\frac{z}{W_1})\wedge^*(\frac{z}{W_2})}(p_{1}(X_{\mbf{v}_1})\otimes p_2(X_{\mbf{v}_2}))).
\end{aligned}
\end{equation*}

Thus both sides coincides by the definition of $\text{Stab}_{\infty}$ in \ref{formula-for-infty-stable-envelope}.
\end{proof}

\textbf{Remark.} One aspect of the Theorem \ref{Commutativity-of-infty-coproduct-with-infty-stable-envelopes} implies that we can write down the expression $\Delta(F)(p_{\bm{\lambda_1}}\otimes p_{\bm{\lambda_2}})$ not only in the formal power series of the tautological classes, but also we can pack them into the rational function of the tautological classes via $\text{Stab}_{\infty}^{-1}(F)\text{Stab}_{\infty}(p_{\bm{\lambda_1}}\otimes p_{\bm{\lambda_2}})$.
\begin{cor}\label{rationality-of-drinfeld-coproduct:corollary}
For the action of $\Delta(F)$ on $K(\mbf{w}_1)\otimes K(\mbf{w}_2)$, it can be written as the rational function of the tautological classes.
\end{cor}

Since for now we can see the Drinfeld coproduct corresponds to the stable envelope of the infinite slope. In the following context we will denote the Drinfel coproduct as $\Delta_{\infty}$ or $\Delta$.

\subsubsection{Integral version}
The Theorem \ref{Commutativity-of-infty-coproduct-with-infty-stable-envelopes} can also be lifted to the integral version:
\begin{thm}\label{Integral-KHA-coproduct:label}
Given arbitrary $F\in\mc{A}^{+,\mbb{Z}}_{Q}$, the following diagrams commute:
\begin{equation*}
\begin{tikzcd}
K_{T}(\mc{M}_{Q}(\mbf{w}_1))\otimes K_{T}(\mc{M}_{Q}(\mbf{w}_2))\arrow[r,"\text{Stab}_{\infty}"]\arrow[d,"\Delta_{\infty}(F)"]&K_{T}(\mc{M}_{Q}(\mbf{w}_1+\mbf{w}_2))\arrow[d,"F"]\\
K_{T}(\mc{M}_{Q}(\mbf{w}_1))\otimes K_{T}(\mc{M}_{Q}(\mbf{w}_2))\arrow[r,"\text{Stab}_{\infty}"]&K_{T}(\mc{M}_{Q}(\mbf{w}_1+\mbf{w}_2).
\end{tikzcd}
\end{equation*}
Moreover, if we denote $\mc{A}^{\geq,\mbb{Z}}_{Q}:=\mc{A}^{+,\mbb{Z}}_{Q}\otimes\mc{A}^{0,\mbb{Z}}_{Q}$, we have that:
\begin{equation*}
\Delta_{\infty}(\mc{A}^{\geq,\mbb{Z}}_{Q})\subset\mc{A}^{\geq,\mbb{Z}}_{Q}\hat{\otimes}\mc{A}^{\geq,\mbb{Z}}_{Q}.
\end{equation*}
\end{thm}
\begin{proof}
Since we have already known the result for the localised case in Theorem \ref{Commutativity-of-infty-coproduct-with-infty-stable-envelopes}, it is equivalent to prove that given $F\in\mc{A}^{+,\mbb{Z}}_{Q}$, we have that:
\begin{align}
\text{Stab}_{\infty}^{-1}(F)\text{Stab}_{\infty}\in\mc{A}^{\geq,\mbb{Z}}_{Q}\hat{\otimes}\mc{A}^{\geq,\mbb{Z}}_{Q}.
\end{align}

For now we can write down the above equation as:
\begin{equation*}
\begin{aligned}
&\text{Stab}_{\infty}^{-1}(F)\text{Stab}_{\infty}(\alpha\otimes\beta)\\
=&(i^*\text{Stab}_{\infty})^{-1}i^*F(i^*)^{-1}i^*\text{Stab}_{\infty}(\alpha\otimes\beta)
\end{aligned}
\end{equation*}
and here $(i^*)^{-1}$ is defined via the equivariant localisation as:
\begin{align*}
(i^*)^{-1}=i_*(\frac{1}{\wedge^*\mc{N}_{F}^{\vee}}).
\end{align*}

Then we consider the following commutative diagram:
\begin{equation*}
\begin{tikzcd}
&\mc{M}_{Q}(\mbf{v},\mbf{v}+\mbf{n},\mbf{w})\arrow[dl,"p\times\pi_{-}"]\arrow[dr,"\pi_{+}"]&\\
\mc{Y}_{\mbf{n}}\times\mc{M}_{Q}(\mbf{v},\mbf{w})&\mc{M}_{Q}(\mbf{v},\mbf{v}+\mbf{n},\mbf{w})^A\arrow[u,hook,"i"]\arrow[dl,"p\times\pi_{-}"]\arrow[dr,"\pi_{+}"]&\mc{M}_{Q}(\mbf{v}+\mbf{n},\mbf{w})\\
\mc{Y}_{\mbf{n}}^A\times\mc{M}_{Q}(\mbf{v},\mbf{w})^A\arrow[u,hook,"i"]&&\mc{M}_{Q}(\mbf{v}+\mbf{n},\mbf{w})^A\arrow[u,hook,"i"].
\end{tikzcd}
\end{equation*}
Rewinding the definition, we have that:
\begin{equation*}
\begin{aligned}
&(i^*\text{Stab}_{\infty})^{-1}i^*F(i^*)^{-1}i^*\text{Stab}_{\infty}(\alpha\otimes\beta)\\
=&\frac{1}{\wedge^{*}\mc{N}^{-}_{F'/X'}}i^*[(\pi_{+})_*(\text{sdet}(p\times\pi_{-})^{!})(\mc{F}\boxtimes i_*((\frac{\wedge^*\mc{N}^{-}_{F/X}}{\wedge^*\mc{N}^{\vee}_{F/X}})\alpha\otimes\beta)]\\
=&\frac{1}{\wedge^{*}\mc{N}^{-}_{F'/X'}}(\pi_{+})_*[i^*[(\text{sdet}(p\times\pi_{-})^{!})(\mc{F}\boxtimes i_*((\frac{\wedge^*\mc{N}^{-}_{F/X}}{\wedge^*\mc{N}^{\vee}_{F/X}})\alpha\otimes\beta)]]\\
=&\frac{1}{\wedge^{*}\mc{N}^{-}_{F'/X'}}(\pi_{+})_*[i^*(\text{sdet})i^{!}(p\times\pi_{-})^{!}(\mc{F}\boxtimes i_*((\frac{\wedge^*\mc{N}^{-}_{F/X}}{\wedge^*\mc{N}^{\vee}_{F/X}})\alpha\otimes\beta))]\\
=&\frac{1}{\wedge^{*}\mc{N}^{-}_{F'/X'}}(\pi_{+})_*[i^*(\text{sdet})(p\times\pi_{-})^{!}i^{!}(\mc{F}\boxtimes i_*((\frac{\wedge^*\mc{N}^{-}_{F/X}}{\wedge^*\mc{N}^{\vee}_{F/X}})\alpha\otimes\beta))]\\
=&\frac{1}{\wedge^{*}\mc{N}^{-}_{F'/X'}}(\pi_{+})_*[i^*(\text{sdet})(p\times\pi_{-})^{!}(i^{!}\mc{F}\boxtimes (\wedge^*\mc{N}^{-}_{F/X})(\alpha\otimes\beta))].
\end{aligned}
\end{equation*}
In the language of the action of $\mc{A}^{ext,\mbb{Z}}_{Q}\hat{\otimes}\mc{A}^{ext,\mbb{Z}}_{Q}$ on $K_{T}(\mc{M}_{Q}(\mbf{w}_1))\otimes K_{T}(\mc{M}_{Q}(\mbf{w}_2))$, the above computation is equivalent to the following:
\begin{align*}
\frac{1}{\wedge^{*}\mc{N}^{-}_{F'/X'}}i^!(\mc{F})\cdot(\wedge^*\mc{N}^{-}_{F/X})\cdot(\alpha\otimes\beta).
\end{align*}
Obviously we have that $i^!(\mc{F})\in\mc{A}^{+,\mbb{Z}}_{Q}\otimes\mc{A}^{+,\mbb{Z}}_{Q}$. Since the Laurent expansion of $\frac{1}{\wedge^{*}\mc{N}^{-}_{F'/X'}}$ can be written as the formal series of tautological classes, and thus one can conclude that:
\begin{align*}
\Delta_{\infty}(F)\in\mc{A}^{\geq,\mbb{Z}}_{Q}\hat{\otimes}\mc{A}^{\geq,\mbb{Z}}_{Q}.
\end{align*}
\end{proof}

Using the similar proof, one can also state the similar theorem for the integral nilpotent KHA:
\begin{thm}\label{Integral-nilpotent-KHA-coproduct:label}
Given arbitrary $F\in\mc{A}^{+,nilp,\mbb{Z}}_{Q}$, the following diagrams commute:
\begin{equation*}
\begin{tikzcd}
K_{T}(\mc{M}_{Q}(\mbf{w}_1))\otimes K_{T}(\mc{M}_{Q}(\mbf{w}_2))\arrow[r,"\text{Stab}_{\infty}"]\arrow[d,"\Delta_{\infty}(F)"]&K_{T}(\mc{M}_{Q}(\mbf{w}_1+\mbf{w}_2))\arrow[d,"F"]\\
K_{T}(\mc{M}_{Q}(\mbf{w}_1))\otimes K_{T}(\mc{M}_{Q}(\mbf{w}_2))\arrow[r,"\text{Stab}_{\infty}"]&K_{T}(\mc{M}_{Q}(\mbf{w}_1+\mbf{w}_2).
\end{tikzcd}
\end{equation*}
Moreover, if we denote $\mc{A}^{\geq,nilp,\mbb{Z}}_{Q}:=\mc{A}^{+,nilp,\mbb{Z}}_{Q}\otimes\mc{A}^{0,\mbb{Z}}_{Q}$, we have that:
\begin{equation*}
\Delta_{\infty}(\mc{A}^{\geq,nilp,\mbb{Z}}_{Q})\subset\mc{A}^{\geq,nilp,\mbb{Z}}_{Q}\hat{\otimes}\mc{A}^{\geq,nilp,\mbb{Z}}_{Q}.
\end{equation*}
\end{thm}

\subsection{Integrality for the stable basis}

In this subsection we describe the attracting and repelling subspace for the torus action. We still fix $X:=\mc{M}_{Q}(\mbf{v},\mbf{w})$ and $A$ a one-dimensional torus such that $\mbf{w}=\mbf{w}_1+a\mbf{w}_2$.	

By the definition of the stable envelope $\text{Stab}_{\sigma,\mbf{m}}$, it is an integral $K$-theory class in $K_{T}(X^A\times X)$ such that it sends $\alpha\in K_{T}(X^A)$ to $\pi_{1*}(\text{Stab}_{\mbf{m}}\cdot\pi_2^*(\alpha))$ via the following correspondences:
\begin{equation*}
\begin{tikzcd}
&X^A\times X\arrow[dl,"\pi_1"]\arrow[dr,"\pi_2"]&\\
X&&X^A.
\end{tikzcd}
\end{equation*}

Also we fix a $Y=\mc{M}_{Q}(\mbf{v}',\mbf{w})$, and now suppose that $F:K_{T_{\mbf{w}}}(X)\rightarrow K_{T_{\mbf{w}}}(Y)$ is a Lagrangian correspondence in $X\times Y$, so for now $\text{Stab}_{\mbf{m}}^{-1}\cdot F\cdot\text{Stab}_{\mbf{m}}$ can be written as the following diagram:
\begin{equation*}
\begin{tikzcd}
&&W\subset X\times Y\arrow[dl,"\pi_1"]\arrow[dr,"\pi_2"]&&\\
&X\arrow[rr,"F"]&&Y&\\
X\times X^A\arrow[ur]\arrow[dr]&&&&Y^A\times Y\arrow[ul]\arrow[dl]\\
&X^A\arrow[uu,"\text{Stab}_{\mbf{m}}"]\arrow[rr,"\text{Stab}_{\mbf{m}}^{-1}\cdot F\cdot\text{Stab}_{\mbf{m}}"]&&Y^A\arrow[uu,"\text{Stab}_{\mbf{m}}"]&\\
&&W^A\subset X^A\times Y^A.\arrow[ul,"\pi_1^A"]\arrow[ur,"\pi_2^A"]&&
\end{tikzcd}
\end{equation*}

Thus if we want $\text{Stab}_{\mbf{m}}^{-1}\cdot F\cdot\text{Stab}_{\mbf{m}}$ to be still integral over the equivariant parametres in $A$, we need to know whether the correspondence $F$ takes supported on $\text{Attr}^f_{\sigma}$ to the classes supported on the attracting set $\text{Attr}^f_{\sigma}$, which implies that:
\begin{align*}
\pi_2(\pi_1^{-1}(\text{Attr}^f_{\sigma}))\subset\text{Attr}^f.
\end{align*}

We use the following result from Negu\c{t} \cite{N23}:
\begin{prop}\label{Description-of-full-attracting-set:label}
If we choose $\sigma:\mbb{C}^*\rightarrow A$ such that $\mbf{w}=\mbf{w}_1+a\mbf{w}_2$, the full attracting subvariety $\text{Attr}^f_{\sigma}\subset X^A\times X$ parametrises triples of framed double quiver representations $(V_{\bullet}',V_{\bullet}'',V_{\bullet})\in \mc{M}_{Q}(\mbf{v}',\mbf{w}')\times \mc{M}_{Q}(\mbf{v}'',\mbf{w}'')\times \mc{M}_{Q}(\mbf{v},\mbf{w})$ such that there exist linear maps
\begin{equation*}
\begin{tikzcd}
V_{\bullet}'\arrow[r,"f"]&V_{\bullet}\arrow[r,"g"]&V_{\bullet}''
\end{tikzcd}
\end{equation*}
such that the following conditions hold:
\begin{itemize}
	\item The composition $g\circ f=0$.
	\item The maps $f$ and $g$ commutes with the $X,Y$ maps, and also commute with the $A,B$ maps via the split long exact sequence
	\begin{equation*}
	\begin{tikzcd}
	W_{\bullet}'\arrow[hook,r]&W_{\bullet}\arrow[twoheadrightarrow,r]&W_{\bullet}''.
	\end{tikzcd}
	\end{equation*}
	\item Letting $\tilde{V}_{\bullet}'=\text{Im}(f)$ and $\tilde{V}_{\bullet}''=V_{\bullet}/\text{Im}(f)$, we require the existence of filtrations
	\begin{equation*}
	\begin{tikzcd}
	V_{\bullet}'=E_{\bullet}^0\arrow[twoheadrightarrow,r]&E_{\bullet}^1\arrow[twoheadrightarrow,r]&\cdots\arrow[twoheadrightarrow,r]&E_{\bullet}^{k-1}\arrow[twoheadrightarrow,r]&E_{\bullet}^k=\tilde{V}_{\bullet}'
	\end{tikzcd}
	\end{equation*}
	\begin{equation*}
	\begin{tikzcd}
	\tilde{V}_{\bullet}''=F_{\bullet}^k\arrow[twoheadrightarrow,r]&F_{\bullet}^{k-1}\arrow[twoheadrightarrow,r]&\cdots\arrow[twoheadrightarrow,r]&F_{\bullet}^1\arrow[twoheadrightarrow,r]&F_{\bullet}^0=V_{\bullet}''.
	\end{tikzcd}
	\end{equation*}
	such that the kernels of the maps $E_{\bullet}^l\rightarrow E_{\bullet}^{l+1}$ and $F_{\bullet}^{l+1}\rightarrow F_{\bullet}^{l}$ are isomorphic.
\end{itemize}
\end{prop}

\begin{prop}\label{Integrality-for-slope-coproduct:Proposition}
Given any $F\in\mc{A}_{Q}^+$, the operator $\text{Stab}_{\mbf{m}}^{-1}F\text{Stab}_{\mbf{m}}$ is an integral $K$-theory class over which is Laurent polynomial over the equivariant variable $a$.
\end{prop}
\begin{proof}
Note that $\mc{A}_{Q}$ is generated by the elements in $\mc{A}^{+}_{Q}$, $\mc{A}^0_{Q}$ and $\mc{A}^{-}_{Q}$. For elements in $\mc{A}^{0}_{Q}$, since it is generated by the tautological classes on $X^A\times X^A$, which are the class supported on the fixed locus, thus it is in $\text{Attr}^f_{\mf{C}}$.

Moreover, since the proof for $\mc{A}^{\pm}_{Q}$ are similar, so we will only focus on $\mc{A}^{+}_{Q}$.

Note that since $F\in\mc{A}^{+}_{Q}$ is the linear combination of the class $e_{i_1,d_1}*\cdots*e_{i_n,d_n}$, it is only left to prove the integrality for the generators $e_{i,d}$.

By definition, $e_{i,d}$ is represented by the quasiprojective scheme $\mc{N}_{\mbf{v},\mbf{v}+\mbf{e}_i,\mbf{w}}$. Points of $\mc{N}_{\mbf{v},\mbf{v}+\mbf{e}_i,\mbf{w}}$ are quadruples of linear maps that preserve a collection of quotients $\{V_j^+\rightarrow V_j^-\}$ of codimension $\delta_{ij}$. Thus we need to prove that if $\{V_j^-\}\in\text{Attr}^f$, then the vector spaces $\{V_j^+\}$ in the definition of $\mc{N}_{\mbf{v},\mbf{v}+\mbf{e}_i,\mbf{w}}$ also lies in $\text{Attr}^f$. Now we fix the splitting $((V_{\bullet}^-)',(V_{\bullet}^-)'')$ for the representation $V_{\bullet}^-$. We want to make the following diagram:
\begin{equation*}
\begin{tikzcd}
0\arrow[r]&\mbb{C}^{\delta_{\blt i}}\arrow[d,"f_{i}"]\arrow[r]&V_{\blt}^{+,'}\arrow[r]\arrow[d,"f_{+}"]&V_{\blt}^{-,'}\arrow[r]\arrow[d,"f_{-}"]&0\\
0\arrow[r]&\mbb{C}^{\delta_{\blt i}}\arrow[d,"g_{i}"]\arrow[r]&V_{\blt}^{+}\arrow[d,"g_{+}"]\arrow[r]&V_{\blt}^{-}\arrow[r]\arrow[d,"g_{-}"]&0\\
0\arrow[r]&0\arrow[r]&V_{\blt}^{+,''}\arrow[r]&V_{\blt}^{-,''}\arrow[r]&0
\end{tikzcd}
\end{equation*}
such that the middle verticle sequence will satisfy the condition for the elements in $\text{Attr}^f$. Now note that the chain $\mbb{C}^{\delta_{\blt i}}\rightarrow\mbb{C}^{\delta_{blt i}}\rightarrow0$ satisfies the condition listed for $\text{Attr}^f$ if and only if $f_i$ is an isomorphism. This means that given the filtration for $V_{\blt}^-$:
\begin{equation*}
\begin{tikzcd}
V_{\bullet}^{- '}=E_{\bullet}^{0,-}\arrow[twoheadrightarrow,r]&E_{\bullet}^{1,-}\arrow[twoheadrightarrow,r]&\cdots\arrow[twoheadrightarrow,r]&E_{\bullet}^{k-1,-}\arrow[twoheadrightarrow,r]&E_{\bullet}^{k,-}=\tilde{V}_{\bullet}^{- '}=\text{Im}(f_-)
\end{tikzcd}
\end{equation*}
\begin{equation*}
\begin{tikzcd}
V_{\blt}^-/\text{Im}(f_-)=\tilde{V}_{\bullet}^{- ''}=F_{\bullet}^{k,-}\arrow[twoheadrightarrow,r]&F_{\bullet}^{k-1,-}\arrow[twoheadrightarrow,r]&\cdots\arrow[twoheadrightarrow,r]&F_{\bullet}^{1,-}\arrow[twoheadrightarrow,r]&F_{\bullet}^{0,-}=V_{\bullet}^{-''}.
\end{tikzcd}
\end{equation*}

By the above diagram we have that $F_{\blt}^{l,-}\cong F_{\blt}^{l,+}$, and $V_{j}^{+'}=V_{j}^{-'}\oplus\mbb{C}^{\delta_{ij}}$. We define that $f_{+}:V_{\blt}^{+'}\rightarrow V_{\blt}^{+}$ to be such that $f_{+}(V_{j}^{+'})\cong f_{-}(V_{j}^{-'})\oplus\mbb{C}$, i.e. we have the following short exact sequence of quiver representations:
\begin{equation*}
\begin{tikzcd}
0\arrow[r]&\mbb{C}^{\delta_{\blt i}}\arrow[r]&\text{Im}(f_+)\arrow[r]&\text{Im}(f_-)\arrow[r]&0.
\end{tikzcd}
\end{equation*}

\end{proof}

As the result, for the class $\text{Stab}_{\sigma,\mbf{m}}^{-1}\cdot F\text{Stab}_{\sigma,\mbf{m}}$, it might be localised over the flavor parametres $q$ and $t_e$, but it is a Laurent polynomial over the equivariant variable $a$.

\subsection{Nilpotent $K$-theoretic stable envelopes}\label{subsection:nilpotent_k_theoretic_stable_envelopes}
Another important counterpart for the $K$-theoretic stable envelope of the Nakajima quiver varieties is the $K$-theoretic stable envelope for the nilpotent quiver varieties. This has been introduced in \cite{SV23} in the cohomological case.

Let us now define what is the $K$-theoretic stable envelope for the nilpotent quiver variety $\mc{L}_{Q}(\mbf{v},\mbf{w})$.

First given a stable envelope class $[\text{Stab}_{\mf{C},s,T^{1/2}}]\in K_{T}(M_{Q}(\mbf{v},\mbf{w})^A\times M_{Q}(\mbf{v},\mbf{w}))$, by the perfect pairing \ref{perfect-pairing:label}, we have the corresponding dual class
\begin{align*}
[\text{Stab}_{\mf{C},s,T^{1/2}}^{\mc{L}}]\in K_{T}(\mc{L}_{Q}(\mbf{v},\mbf{w})^A\times\mc{L}_{Q}(\mbf{v},\mbf{w}))
\end{align*}
corresponding to $[\text{Stab}_{\mf{C},s,T^{1/2}}]$.
It is supported on $\text{Attr}_{\mf{C}}^f\cap(\mc{L}_{Q}(\mbf{v},\mbf{w})^A\times\mc{L}_{Q}(\mbf{v},\mbf{w}))$. Moreover, the projection map $\text{Attr}_{f,\mf{C}}\cap(\mc{L}_{Q}(\mbf{v},\mbf{w})^A\times\mc{L}_{Q}(\mbf{v},\mbf{w}))\rightarrow\mc{L}_{Q}(\mbf{v},\mbf{w})^A$ is proper. Thus we can define the \textbf{nilpotent stable envelope} as the convolution by:
\begin{align*}
\text{Stab}_{\mf{C},s,T^{1/2}}^{\mc{L},\vee}:K_{T}(\mc{L}_{Q}(\mbf{v},\mbf{w}))\rightarrow K_{T}(\mc{L}_{Q}(\mbf{v},\mbf{w})^A).
\end{align*}

Similar to the case as in the stable envelopes, it is an isomorphism after localisation. It satisfies the condition as written in \ref{subsection:_k_theoretic_stable_envelopes}.

The following lemma implies why we define the nilpotent stable envelope on the other way round. 

\begin{lem}\label{transpose-of-stable-envelopes:label}
The nilpotent stable envelope $\text{Stab}_{\mf{C},s,T^{1/2}}^{\mc{L},\vee}$ is the transpose of $\text{Stab}_{\mf{C},s,T^{1/2}}$ under the perfect pairing \ref{perfect-pairing:label}.
\end{lem}
\begin{proof}
This follows from the perfect pairing \ref{perfect-pairing:label}.
\begin{comment}
The following diagram is useful for the computation:
\begin{equation*}
\begin{tikzcd}
&\mc{L}_{Q}(\mbf{v},\mbf{w})^A\times\mc{L}_{Q}(\mbf{v},\mbf{w})\arrow[dl,"\pi_{1,\mc{L}}"]\arrow[dr,"\pi_{2,\mc{L}}"]\arrow[dd,hook,"i_A\times i"]&\\
\mc{L}_{Q}(\mbf{v},\mbf{w})^A\arrow[dd,"i_A"]&&\mc{L}_{Q}(\mbf{v},\mbf{w})\arrow[dd,"i"]\\
&M_{Q}(\mbf{v},\mbf{w})^A\times M_{Q}(\mbf{v},\mbf{w})\arrow[dl,"\pi_1"]\arrow[dr,"\pi_2"]&\\
M_{Q}(\mbf{v},\mbf{w})^A&&M_{Q}(\mbf{v},\mbf{w})
\end{tikzcd}
\end{equation*}

This can be via the following computation:
\begin{equation*}
\begin{aligned}
&p_*(\mc{F}\otimes i_{*}\text{Stab}_{\mf{C},s}^{\mc{L},\vee}(\mc{G}))=p_{*}(\mc{F}\otimes i_*(\pi_{2,\mc{L}})_*(\text{Stab}_{\mf{C},s}^{\mc{L},\vee}\otimes\pi_{1,\mc{L}}^*\mc{G}))\\
=&p_*(\mc{F}\otimes(\pi_2)_*(i_A\times i)_*(\text{Stab}_{\mf{C},s}^{\mc{L},\vee}\otimes\pi_{1,\mc{L}}^*\mc{G}))\\
=&p_*(\mc{F}\otimes(\pi_2)_*(i_A\times i)_*((i_A\times i)^*(\text{Stab}_{\mf{C},s})\otimes\pi_{1,\mc{L}}^*\mc{G}))\\
=&p_*(\mc{F}\otimes(\pi_2)_*(\text{Stab}_{\mf{C},s}\otimes (i_A\times i)_*\pi_{1,\mc{L}}^*\mc{G}))\\
=&p_*(\mc{F}\otimes\pi_{2*}(\text{Stab}_{\mf{C},s}\otimes\pi_1^*i_{A*}\mc{G}))\\
=&(\mc{F},\text{Stab}_{\mf{C},s,T^{1/2}}^{op}(i_{A*}\mc{G}))\\
=&(\text{Stab}_{\mf{C},s,T^{1/2}}(\mc{F}),i_{A*}\mc{G})
\end{aligned}
\end{equation*}
\end{comment}
\end{proof}

\subsection{Nilpotent Maulik-Okounkov quantum loop groups}\label{subsection:nilpotent_maulik_okounkov_quantum_loop_groups}

Similar to the stable envelope, the nilpotent stable envelope $\text{Stab}_{\mf{C},s}^{\mc{L},\vee}$ is an isomorphism after the localisation. In this one can define the nilpotent geometric $R$-matrix:
\begin{align}\label{nilpotent-R-matrix}
\mc{R}^{\mc{L},s}_{\mf{C}_1,\mf{C}_2}:=\text{Stab}_{\mf{C}_1,s}^{\mc{L},\vee}\circ(\text{Stab}_{\mf{C}_2,s}^{\mc{L},\vee})^{-1}:K_{T}(\mc{L}_{Q}(\mbf{v},\mbf{w})^A)_{loc}\rightarrow K_{T}(\mc{L}_{Q}(\mbf{v},\mbf{w}))_{loc}.
\end{align}
and here the localisation is over the equivariant variables $K_{G_{\mbf{w}}}(pt)_{loc}$. 

From the definition one can see the following lemma:
\begin{lem}
$\mc{R}^{\mc{L},s}_{\mf{C}_1,\mf{C}_2}$ is the transpose of $(\mc{R}^s_{\mf{C}_1,\mf{C}_2})^{-1}$.
\end{lem}

\begin{defn}
The \textbf{nilpotent Maulik-Okounkov quantum loop group} $U_{q}^{MO,nilp}(\hat{\mf{g}}_{Q})$ is an algebra over $\mbb{Z}[q^{\pm1},t_{e}^{\pm1}]_{e\in E}$ generated by the matrix coefficients of the nilpotent geometric $R$-matrix $\mc{R}^{\mc{L},s}_{\mf{C}_1,\mf{C}_2}$.
\end{defn}

Now we use the following definition
\begin{align*}
K(\mbf{w})^{\vee}:=\bigoplus_{\mbf{v}\in\mbb{N}^I}K_{T}(\mc{L}_{Q}(\mbf{v},\mbf{w}))_{loc}
\end{align*}
and by definition it is easy to see that the nilpotent MO quantum loop group $U_{q}^{MO,nilp}(\hat{\mf{g}}_{Q})$ has the following embedding of algebras:
\begin{align*}
U_{q}^{MO,nilp}(\hat{\mf{g}}_{Q})\hookrightarrow\prod_{\mbf{w}}\text{End}(K(\mbf{w})^{\vee}).
\end{align*}

Now we denote $U_{q}^{MO,nilp}(\hat{\mf{g}}_{Q})_{loc}:=U_{q}^{MO,nilp}(\hat{\mf{g}}_{Q})\otimes\mbb{Q}(q,t_e)_{e\in E}$.
\begin{lem}\label{Isomorphism-of-MO-and-nilpotent-MO-after-localisation:lemma}
There is an isomorphism of $\mbb{Q}(q,t_e)_{e\in E}$-algebras $U_{q}^{MO,nilp}(\hat{\mf{g}}_{Q})_{loc}\cong U_{q}^{MO}(\hat{\mf{g}}_{Q})_{loc}$.
\end{lem}
\begin{proof}
By definition the elements of $U_{q}^{MO,nilp}(\hat{\mf{g}}_{Q})$ can be written as:
\begin{align}\label{generators-of-nilpotent}
\oint_{a_0=0,\infty}\text{Tr}_{V}((1\otimes m(a_0)^{\vee})\mc{R}^{\mc{L},s}_{\mf{C}}(\frac{a}{a_0})),\qquad m(a_0)^{\vee}\in\text{End}(K(\mbf{w})^{\vee})(a_0).
\end{align}

By the natural pairing there is an isomorphism of graded $\mbb{Q}(q,t_e)_{e\in E}$-modules $K(\mbf{w})^{\vee}\cong K(\mbf{w})$, there is an isomorphism of graded $\mbb{Q}(q,t_e)_{e\in E}$-modules $\text{End}(K(\mbf{w})^{\vee})\cong\text{End}(K(\mbf{w}))$. Now since $\mc{R}^{\mc{L},s}_{\mf{C}}$ is the transpose of $\mc{R}^{s}_{\mf{C}}$. There one can see that 
\begin{align*}
\oint_{a_0=0,\infty}\text{Tr}_{V}((1\otimes m(a_0))\mc{R}^{s}_{\mf{C}}(\frac{a}{a_0})),\qquad m(a_0)\in\text{End}(K(\mbf{w}))(a_0)
\end{align*}
is dual to the element \ref{generators-of-nilpotent}.
\end{proof}

\subsection{Factorisation property for the nilpotent geometric $R$-matrix}
\subsubsection{Nilpotent wall $R$-matrices and nilpotent wall subalgebras}
Similar to the case of the stable envelope, the nilpotent stable envelope is locally constant for the slope point with the same wall set as in Proposition \ref{wall-set:proposition}.

Now similarly as the original situation, given $X:=\mc{L}_{Q}(\mbf{v},\mbf{w})$, and fix the slope $\mbf{m}$ and the cocharacter $\sigma:\mbb{C}^*\rightarrow A_{\mbf{w}}$ such that $\mbf{w}=\mbf{w}_1+a\mbf{w}_2$. Choose an ample line bundle $\mc{L}\in\text{Pic}(X)$ and a suitable small number $\epsilon$ such that $\mbf{m}$ and $\mbf{m}+\epsilon\mc{L}$ are separated by just one wall $w$. The \textbf{nilpotent wall $R$-matrices} are:
\begin{align*}
R_{w}^{\pm,\mc{L}}:=\text{Stab}^{\mc{L}}_{\pm\sigma,\mbf{m}+\epsilon\mc{L}}\circ(\text{Stab}^{\mc{L}}_{\pm\sigma,\mbf{m}})^{-1}\in\text{End}_{K_{T_{\mbf{w}}}(pt)}(K_{T_{\mbf{w}}}(\mc{L}_{Q}(\mbf{v},\mbf{w})^A)).
\end{align*}

Still the same, this is an integral $K$-theory class in $K_{T_{\mbf{w}}}(\mc{L}_{Q}(\mbf{v},\mbf{w})^A\times\mc{L}_{Q}(\mbf{v},\mbf{w})^A)$. Similarly $R_{w}^{+\mc{L}}$ is upper-triangular and $R_{w}^{-,\mc{L}}$ is lower-triangular. Also similar to the result in \ref{monomiality-for-wall-R-matrices}, the nilpotent wall $R$-matrices are monomial in the spectral parametre $a$. Moreover, by the transposition property, we have that:
\begin{align}\label{transpose-of-upper-and-lower-triangular}
R_{w}^{\pm,\mc{L}}=(R_{w}^{\pm})^T
\end{align}
which means that we switches the upper-triangularity and lower-triangularity respectively.

In this way one can also similarly define the nilpotent wall subalgebra $U_{q}^{MO,nilp,\mbb{Z}}(\mf{g}_{w})$ as the $\mbb{Z}[q^{\pm1},t_{e}^{\pm1}]$-algebra generated by the matrix coefficients of the wall $R$-matrix $q^{\Omega}R_{w}^{\mp,\mc{L}}$. Correspondingly, one can also define the positive half $U_{q}^{MO,nilp,+,\mbb{Z}}(\mf{g}_{w})$ and the negative half $U_{q}^{MO,nilp,-,\mbb{Z}}(\mf{g}_{w})$.

\subsubsection{Freeness of nilpotent wall subalgebra}

Imitating the proof in Theorem \ref{subsection:freeness_of_wall_subalgebras}, one can have the following similar result as Theorem \ref{freeness-of-wall-subalgebra:theorem}:

\begin{thm}\label{freeness-of-nilpotent-wall-subalgebra:theorem}
The wall nilpotent subalgebra $U_{q}^{MO,nilp,\mbb{Z}}(\mf{g}_{w})$ is a free $\mbb{Z}[q^{\pm1},t_{e}^{\pm1}]_{e\in E}$-module. Moreover, its positive half and the negative half are generated by the primitive elements as defined in \ref{primitive-element-definition}.
\end{thm}

Over here we denote $U_{q}^{MO,nilp,prim,\pm,\mbb{Z}}(\mf{g}_{w})$ as the $\mbb{Z}[q^{\pm1},t_{e}^{\pm1}]$-submodule of $U_{q}^{MO,nilp,\pm,\mbb{Z}}(\mf{g}_{w})$ consisting of the primitive elements of $U_{q}^{MO,nilp,prim,\pm,\mbb{Z}}(\mf{g}_{w})$, which is also a free $\mbb{Z}[q^{\pm1},t_{e}^{\pm1}]$-module.

\subsubsection{Factorisation property}
Now we fix the nilpotent stable envelope $\text{Stab}_{\sigma,\mbf{m}}^{\mc{L}}$ and $\text{Stab}_{\sigma,\infty}^{\mc{L}}$, similar to the factorisation in \ref{subsubsection:factorisation_of_geometric_r_matrices_and_integral_maulik_okounkov_quantum_affine_algebras}, we have the factorisation of the geometric $R$-matrices written as:
\begin{align}\label{nilpotent-factorisation}
\mc{R}^{s,\mc{L}}(a)=\prod^{\leftarrow}_{i<0}R_{w_i}^{-,\mc{L}}\mc{R}^{\infty,\mc{L}}\prod^{\leftarrow}_{i\geq0}R_{w_i}^{+,\mc{L}}
\end{align}
and we should understand this factorisation as the formal power series expansion of the nilpotent $R$-matrix near $a=0,\infty$.

\subsection{Isomorphism of the integral form}
Similar to the case of the usual MO quantum loop group, the above factorisation property ensures that we can define the integral version of the nilpotent MO quantum loop group.
\begin{defn}
The \textbf{integral nilpotent MO quantum loop group} $U_{q}^{MO,nilp,\mbb{Z}}$ is a $\mbb{Z}[q^{\pm1},t_{e}^{\pm1}]_{e\in E}$-subalgebra of $\prod_{\mbf{w}}\text{End}_{K_{T_{\mbf{w}}}(pt)}(K_{T_{\mbf{w}}}(\mc{L}_{Q}(\mbf{w})))$ generated by matrix coefficients of the nilpotent geometric $R$-matrix with the factorisation \ref{nilpotent-factorisation}.
\end{defn}

Now we prove the following the isomorphism for the integral form:
\begin{thm}\label{transpose-map-anti-isomorphism:label}
The transpose map $(-)^{T}:\prod_{\mbf{w}}K_{T_{\mbf{w}}}(\mc{M}_{Q}(\mbf{w}))\rightarrow\prod_{\mbf{w}}K_{T_{\mbf{w}}}(\mc{L}_{Q}(\mbf{w}))$ induces the anti-isomorphism of $\mbb{Z}[q^{\pm1},t_{e}^{\pm1}]$-algebras $(-)^T:U_{q}^{MO,\mbb{Z}}(\hat{\mf{g}}_{Q})\cong U_{q}^{MO,nilp,\mbb{Z}}(\hat{\mf{g}}_{Q})$.
\end{thm}
\begin{proof}
It is known that the operator in $U_{q}^{MO,\mbb{Z}}(\hat{\mf{g}}_{Q})$ can be written as:
\begin{align*}
\oint_{a_0=0,\infty}\frac{da_0}{2\pi ia_0}\text{Tr}_{V_0}((1\otimes m(a_0))\mc{R}^s_{V,V_0}(\frac{a}{a_0})).
\end{align*}

By Lemma \ref{transpose-of-stable-envelopes:label} and the transpose map, we can write down the above equation in the following way:
\begin{align*}
\oint_{a_0=0,\infty}\text{Tr}_{V_0^*}((1\otimes m^T(a_0))\mc{R}^{s,\mc{L}}_{V_0^*,V^*}(\frac{a}{a_0}))
\end{align*}
which is an element in $U_{q}^{MO,nilp,\mbb{Z}}(\hat{\mf{g}}_{Q})$, and it is easy to check that this is an $\mbb{Z}[q^{\pm1},t_{e}^{\pm1}]$-algebra anti-homomorphism:
\begin{equation*}
\begin{aligned}
&(-)^T:U_{q}^{MO,\mbb{Z}}(\hat{\mf{g}}_{Q})\rightarrow U_{q}^{MO,nilp,\mbb{Z}}(\hat{\mf{g}}_{Q})\\
&\oint_{a_0=0,\infty}\frac{da_0}{2\pi ia_0}\text{Tr}_{V_0}((1\otimes m(a_0))\mc{R}^s_{V,V_0}(\frac{a}{a_0}))\mapsto\oint_{a_0=0,\infty}\text{Tr}_{V_0^*}((1\otimes m^T(a_0))\mc{R}^{s,\mc{L}}_{V_0^*,V^*}(\frac{a}{a_0})).
\end{aligned}
\end{equation*}

The surjectivity comes from the fact that $\mc{R}^{s}_{V,V_0}=\mc{R}^s_{V_0^*,V^*}$, and the injectivity comes from the perfect pairing \ref{perfect-pairing:label}.
\end{proof}

More precisely, one can also have the following anti-isomorphism of the wall subalgebra and the nilpotent wall subalgebra using the similar proof:
\begin{prop}\label{transpose-map-anti-isomorphism-for-wall-subalgebra:label}
The transpose map $(-)^{T}:\prod_{\mbf{w}}K_{T_{\mbf{w}}}(\mc{M}_{Q}(\mbf{w}))\rightarrow\prod_{\mbf{w}}K_{T_{\mbf{w}}}(\mc{L}_{Q}(\mbf{w}))$ induces the anti-isomorphism of $\mbb{Z}[q^{\pm1},t_{e}^{\pm1}]$-algebras $(-)^T:U_{q}^{MO,\mbb{Z}}(\mf{g}_{w})\cong U_{q}^{MO,nilp,\mbb{Z}}(\mf{g}_{w})$.
\end{prop}

\subsection{Hopf algebra structure}
One can define the Hopf algebra structure on $U_{q}^{MO,nilp}(\hat{\mf{g}}_{Q})$ similarly as the definition in Section \ref{subsection:maulik_okounkov_quantum_algebra_and_wall_subalgebras}. For example, the coproduct $\Delta_{s}^{MO}$ is defined as:
\begin{align}\label{nilpotent-geometric-coproduct-definition}
\Delta_{s}^{MO,\mc{L}}(F):=\text{Stab}^{\mc{L},\vee}_{\mf{C},s}(F)(\text{Stab}^{\mc{L},\vee}_{\mf{C},s})^{-1}\in U_{q}^{MO,nilp}(\hat{\mf{g}}_{Q})\otimes U_{q}^{MO,nilp}(\hat{\mf{g}}_{Q}),\qquad F\in U_{q}^{MO,nilp}(\hat{\mf{g}}_{Q}).
\end{align}

Note that since the nilpotent stable envelope is the transpose of the original stable envelope. This means that after the transposition, we have that:
\begin{align}\label{geometric-coproduct-transpose}
\Delta_s^{MO,\mc{L}}=(\Delta_{s}^{MO})^T.
\end{align}

The following proposition is easy to prove:
\begin{prop}\label{transpose-Hopf-algebra:proposition}
The anti-isomorphism between nilpotent MO quantum loop group $U_{q}^{MO,nilp}(\hat{\mf{g}}_{Q})_{loc}$ and MO quantum loop groups $U_{q}^{MO}(\hat{\mf{g}}_{Q})_{loc}$ intertwines the Hopf algebra sturcture on the respective slope point $s\in\mbb{R}^I$.
\end{prop}
\begin{proof}
This follows from Theorem \ref{transpose-map-anti-isomorphism:label} and Lemma \ref{transpose-of-stable-envelopes:label}.
\end{proof}

\section{\textbf{Injective map as the localised form}}\label{section:_textbf_isomorphism_as_the_hopf_algebras}

For simplicity of the notation, we will use the following notation to stand for the following objects:
\begin{equation}\label{notation-for-fixed-point-components}
\begin{aligned}
&F_{\alpha}=\mc{M}_{Q}(\mbf{v}_1,\mbf{w}_1)\times \mc{M}_{Q}(\mbf{v}_2,\mbf{w}_2),\qquad F_{\beta}=\mc{M}_{Q}(\mbf{v}_1+\mbf{k},\mbf{w}_1)\times \mc{M}_{Q}(\mbf{v}_2+\mbf{n}-\mbf{k},\mbf{w}_2)\\
&F_{\alpha'}=\mc{M}_{Q}(\mbf{v}_1-\bm{\delta},\mbf{w}_1)\times \mc{M}_{Q}(\mbf{v}_2+\bm{\delta},\mbf{w}_2),\qquad F_{\beta'}=\mc{M}_{Q}(\mbf{v}_1-\bm{\delta}+\mbf{l},\mbf{w}_1)\times \mc{M}_{Q}(\mbf{v}_2+\bm{\delta}+\mbf{n}-\mbf{l},\mbf{w}_2).
\end{aligned}
\end{equation}

Also in many cases of the section, we will use the torus action $\mbf{w}=\mbf{w}_1+a\mbf{w}_2$. But the stable envelope $\text{Stab}_{\mbf{m}}$ we are using here are for the torus action $\mbf{w}=a'\mbf{w}_1+\mbf{w}_2$, and the reader can think of $\mbf{a}$ as $(\mbf{a}')^{-1}$. The reason for such notations is for the sake of the convenience of counting the vertical degree on the slope subalgebras $\mc{B}_{\mbf{m}}$.

\subsection{Degree bounding}\label{sub:degree_bounding}

Recall that the construction of the geometric action is of $\mc{A}^+_{Q}$ on $K(\mbf{w})$ is given by the following: Given $F\in\mc{A}_{\mbf{n}}$ and $G\in\mc{A}_{-\mbf{n}}$, we have that:
\begin{align}\label{positive-action}
F\cdot p(\mbf{X}_{\mbf{v}})=\frac{1}{\mbf{n}!}\int^{+}\frac{F(\mbf{Z}_{\mbf{n}})}{\tilde{\zeta}(\frac{\mbf{Z}_{\mbf{n}}}{\mbf{Z}_{\mbf{n}}})}p(\mbf{X}_{\mbf{v}+\mbf{n}}-\mbf{Z}_{\mbf{n}})\tilde{\zeta}(\frac{\mbf{Z}_{\mbf{n}}}{\mbf{X}_{\mbf{v}+\mbf{n}}})\wedge^*(\frac{\mbf{Z}_{\mbf{n}}q}{\mbf{W}})
\end{align}

\begin{align}\label{negative-action}
G\cdot p(\mbf{X}_{\mbf{v}})=\frac{1}{\mbf{n}!}\int^{-}\frac{G(\mbf{Z}_{\mbf{n}})}{\tilde{\zeta}(\frac{\mbf{Z}_{\mbf{n}}}{\mbf{Z}_{\mbf{n}}})}p(\mbf{X}_{\mbf{v}-\mbf{n}}+\mbf{Z}_{\mbf{n}})\tilde{\zeta}(\frac{\mbf{X}_{\mbf{v}-\mbf{n}}}{\mbf{Z}_{\mbf{n}}})^{-1}\wedge^*(\frac{\mbf{Z}_{\mbf{n}}}{\mbf{W}}).
\end{align}

Let us first concentrate on the case when $F\in\mc{A}_{\mbf{n},Q}^+$.

We can write down the formula in an explicit way:
\begin{equation}\label{explicit-formula-stable}
\begin{aligned}
&\Delta_{\infty}(F)(p_{\bm{\lambda}_1}(X_{\mbf{v}_1})\otimes p_{\bm{\lambda_2}}(X_{\mbf{v}_2}))\\
=&\sum_{[0\leq k_i\leq n_i]_{i\in I}}\frac{\prod_{k_j<b\leq n_j}^{j\in I}h_{j}^{+}(z_{jb})F(\cdots,z_{i1},\cdots,z_{ik_i}\otimes z_{i,k_i+1},\cdots,z_{in_i},\cdots)}{\prod_{1\leq a\leq k_i}^{i\in I}\prod_{k_j<b\leq n_j}^{j\in I}\zeta_{ji}(z_{jb}/z_{ia})}(p_{\bm{\lambda}_1}(X_{\mbf{v}_1})\otimes p_{\bm{\lambda_2}}(X_{\mbf{v}_2}))\\
=&\sum_{[0\leq k_i\leq n_i]_{i\in I}}\frac{1}{\mbf{k}!(\mbf{n}-\mbf{k})!}\int^{+}\int^{+}\frac{\prod^{j\in I}_{k_j<b\leq n_j}h_{j}^+(Z_{\mbf{n}-\mbf{k}})F(Z_{\mbf{k}}\otimes Z_{\mbf{n}-\mbf{k}})}{\tilde{\zeta}(\frac{Z_{\mbf{k}}}{Z_{\mbf{k}}})\tilde{\zeta}(\frac{Z_{\mbf{n}-\mbf{k}}}{Z_{\mbf{n}-\mbf{k}}})\zeta(\frac{Z_{\mbf{n}-\mbf{k}}}{Z_{\mbf{k}}})}\\
&(p_{\bm{\lambda}_1}(X_{\mbf{v}_1+\mbf{k}}-Z_{\mbf{k}})\otimes p_{\bm{\lambda}_2}(X_{\mbf{v}_2+\mbf{n}-\mbf{k}}-Z_{\mbf{n}-\mbf{k}}))\tilde{\zeta}(\frac{Z_{\mbf{k}}}{X_{\mbf{v}_1+\mbf{k}}})\tilde{\zeta}(\frac{Z_{\mbf{n}-\mbf{k}}}{X_{\mbf{v}_2+\mbf{n}-\mbf{k}}})\wedge^*(\frac{Z_{\mbf{k}}q}{\mbf{W}_1})\wedge^*(\frac{Z_{\mbf{n}-\mbf{k}}q}{\mbf{W}_2}).
\end{aligned}
\end{equation}

Note that the value of $\Delta_{\infty}(F)(p_{\bm{\lambda}_1}\otimes p_{\bm{\lambda}_2})$ lies in $K(\mbf{w}_1)\otimes K(\mbf{w}_2)$. If we take the framing such that $\mbf{w}_1+a\mbf{w}_2$, it is a rational function over $a$ by Theorem \ref{Commutativity-of-infty-coproduct-with-infty-stable-envelopes} or Corollary \ref{rationality-of-drinfeld-coproduct:corollary}. We denote $\Delta_{\infty}(F)(p_{\bm{\lambda}_1}\otimes p_{\bm{\lambda}_2})(a)$ as the image of $\Delta_{\infty}(F)(p_{\bm{\lambda}_1}\otimes p_{\bm{\lambda}_2})$ after evaluating to $a$. We denote:
\begin{align*}
\text{max deg}_{A}(\Delta_{\infty}(F)|_{F_{\alpha}\times F_{\beta}}):=\text{Leading order of }\Delta_{\infty}(F)|_{F_{\alpha}\times F_{\beta}}(a)\text{ as }a\rightarrow\infty
\end{align*}

\begin{align*}
\text{min deg}_{A}(\Delta_{\infty}(F)|_{F_{\alpha}\times F_{\beta}}):=\text{Leading order of }\Delta_{\infty}(F)|_{F_{\alpha}\times F_{\beta}}(a)\text{ as }a\rightarrow0.
\end{align*}

\begin{lem}\label{positive-half-infinite-degree-bounding:lemma}
If $F\in\mc{B}_{\mbf{m},\mbf{n}}^+$, one has that
\begin{align*}
\text{max deg}_{A}\Delta_{\infty}(F)|_{F_{\alpha}\times F_{\beta}}\leq\mbf{m}\cdot(\mbf{n}-\mbf{k})
\end{align*}
and
\begin{align*}
\text{min deg}_{A}\Delta_{\infty}(F)|_{F_{\alpha}\times F_{\beta}}\geq\mbf{m}\cdot(\mbf{n}-\mbf{k}).
\end{align*}

\end{lem}
\begin{proof}
Using the coproduct formula, the degree bounding for $\Delta_{\infty}(F)|_{F_{\alpha}\times F_{\beta}}$ can be computed via the residue calculation.

Recall that the coproduct formula comes from the integral \ref{explicit-formula-stable}. We first analyse which part of residue will contribute to the lowest or highest $A$-weight component.

First note that since the $A$-weights are given by the tautological classes of $X_{\mbf{v}_2+\mbf{n}-\mbf{k}}$. In this way the term $p_{\bm{\lambda}_1}(X_{\mbf{v}_1+\mbf{k}}-Z_{\mbf{k}})\otimes p_{\bm{\lambda}_2}(X_{\mbf{v}_2+\mbf{n}-\mbf{k}}-Z_{\mbf{n}-\mbf{k}})$ can be written as the expansion of $\gamma_{(\bm{\mu}_1,\bm{\mu}_2)}p_{\bm{\mu}_1}(X_{\mbf{v}_1+\mbf{k}})\otimes p_{\bm{\mu}_2}(X_{\mbf{v}_2+\mbf{n}-\mbf{k}})$  with $\gamma_{(\bm{\mu}_1,\bm{\mu}_2)}$ the Laurent polynomial of $Z_{\mbf{k}}$ and $Z_{\mbf{n}-\mbf{k}}$. There one can see that the lowest and highest $A$-degree part is given by th part such that all the variables $Z_{\mbf{n}-\mbf{k}}$ take the residue over the poles over the tautological classes.

The residue that contributed to the $A$-degree in the integral has the poles of the form:
\begin{align*}
z_{ia}=x_{ib},\qquad z_{ia}=q^{-1}x_{ib}
\end{align*}

\begin{align*}
z_{ia}=t_{e}^{-1}x_{jb},\qquad z_{ia}=t_e^{-1}qx_{jb}.
\end{align*}

As we can see in the integral \ref{explicit-formula-stable}, these poles are of the simple poles if the torus character $t_e$ are generically different. Using the residue formula, one can see that the $A$-degree contribution is given by the term:
\begin{align*}
\frac{\prod_{k_j<b\leq n_j}^{j\in I}h_j^+(Z_{\mbf{n}-\mbf{k}})}{\zeta(\frac{Z_{\mbf{n}-\mbf{k}}}{Z_{\mbf{k}}})}F(Z_{\mbf{k}}\otimes Z_{\mbf{n}-\mbf{k}})=\frac{\tilde{\zeta}(\frac{Z_{\mbf{n}-\mbf{k}}}{X_{\mbf{v}+\mbf{n}}})}{\tilde{\zeta}(\frac{X_{\mbf{v}+\mbf{n}}}{Z_{\mbf{n}-\mbf{k}}})\zeta(\frac{Z_{\mbf{n}-\mbf{k}}}{Z_{\mbf{k}}})}\frac{\wedge^*(\frac{Z_{\mbf{n}-\mbf{k}}q}{W})}{\wedge^*(\frac{Z_{\mbf{n}-\mbf{k}}}{W})}F(Z_{\mbf{k}}\otimes Z_{\mbf{n}-\mbf{k}}).
\end{align*}

Thus the maximal $A$-degree can be computed via computing the following limit:
\begin{align*}
\lim_{\xi\rightarrow\infty}\frac{\tilde{\zeta}(\frac{\xi Z_{\mbf{n}-\mbf{k}}}{X_{\mbf{v}+\mbf{n}-\mbf{k}}})}{\tilde{\zeta}(\frac{X_{\mbf{v}+\mbf{n}-\mbf{k}}}{\xi Z_{\mbf{n}-\mbf{k}}})\zeta(\frac{\xi Z_{\mbf{n}-\mbf{k}}}{Z_{\mbf{k}}})}\frac{\wedge^*(\frac{\xi Z_{\mbf{n}-\mbf{k}}q}{W})}{\wedge^*(\frac{\xi Z_{\mbf{n}-\mbf{k}}}{W})}F(Z_{\mbf{k}}\otimes\xi Z_{\mbf{n}-\mbf{k}})|_{z_{in}=(\cdots)x_{(\cdots)}}.
\end{align*}

In this case one can see that only the following gives the contribution to the $A$-degree:
\begin{align*}
\lim_{\xi\rightarrow\infty}\frac{F(Z_{\mbf{k}}\otimes \xi Z_{\mbf{n}-\mbf{k}})}{\zeta(\xi Z_{\mbf{n}-\mbf{k}}/Z_{\mbf{k}})}.
\end{align*}

Now since we have the degree condition in the definition of slope subalgebras in Section \ref{sub:slope_subalgebras_and_factorizations_of_r_matrices} that:
\begin{align*}
\text{max deg}_{A}\text{Stab}_{\infty}^{-1}F\text{Stab}_{\infty}|_{F_{\alpha}\times F_{\beta}}\leq\langle\mbf{k},\mbf{n}-\mbf{k}\rangle+\mbf{m}\cdot(\mbf{n}-\mbf{k})-\langle\mbf{k},\mbf{n}-\mbf{k}\rangle=\mbf{m}\cdot(\mbf{n}-\mbf{k}).
\end{align*}

This implies that 
\begin{align*}
\text{deg}_{A}\lim_{\xi\rightarrow\infty}\frac{F(Z_{\mbf{k}}\otimes \xi Z_{\mbf{n}-\mbf{k}})}{\zeta(\xi Z_{\mbf{n}-\mbf{k}}/Z_{\mbf{k}})}\leq\mbf{m}\cdot(\mbf{n}-\mbf{k}).
\end{align*}

\begin{comment}
We want to compute the $A$-degree of $\gamma_{(\bm{\lambda}_1,\bm{\lambda}_2)}^{(\bm{\mu}_1,\bm{\mu}_2)}$, by definition we have that:
\begin{equation}
\begin{aligned}
\text{max deg}_{A}(i^*\theta_{(\bm{\mu}_1,\bm{\mu}_2)})\leq&\text{max}_{\alpha\geq\alpha'}(\text{deg}_{A}(R_{\infty,\mbf{m}}|_{F_{\alpha}\times F_{\alpha'}}))+\text{max}_{\beta,\alpha'}(\text{deg}_{A}(F|_{F_{\alpha'}\times F_{\beta}}))\\
\leq&\langle\mu(F_{\alpha'})-\mu(F_{\alpha}),\mbf{m}\rangle+\langle\mu(F_{\beta})-\mu(F_{\alpha'}),\mbf{m}\rangle\\
\leq&\langle\mu(F_{\beta})-\mu(F_{\alpha}),\mbf{m}\rangle
\end{aligned}
\end{equation}
\end{comment}

For the minimal degree, one just need to compute the degree of
\begin{align*}
\lim_{\xi\rightarrow0}\frac{F(Z_{\mbf{k}}\otimes \xi Z_{\mbf{n}-\mbf{k}})}{\zeta(\xi Z_{\mbf{n}-\mbf{k}}/Z_{\mbf{k}})}
\end{align*}
and this can be computed such that:
\begin{align*}
\text{deg}_{A}\lim_{\xi\rightarrow0}\frac{F(Z_{\mbf{k}}\otimes \xi Z_{\mbf{n}-\mbf{k}})}{\zeta(\xi Z_{\mbf{n}-\mbf{k}}/Z_{\mbf{k}})}\geq\mbf{m}\cdot\mbf{n}-\mbf{m}\cdot\mbf{k}-\langle\mbf{k},\mbf{n}-\mbf{k}\rangle+\langle\mbf{k},\mbf{n}-\mbf{k}\rangle=\mbf{m}\cdot(\mbf{n}-\mbf{k}).
\end{align*}
In this case one has
\begin{align*}
\text{min deg}_{A}\text{Stab}_{\infty}^{-1}F\text{Stab}_{\infty}=\mbf{m}\cdot(\mbf{n}-\mbf{k})
\end{align*}

\end{proof}

For the negative half $G\in\mc{A}_{Q,-\mbf{n}}^{-}$, one can also do the similar analysis as follows. Here we will list the main conclusions and the sketch of the proof since the proof is totally similar to the case of the positive half.

From now on we define the $A$-degree on $G\in\mc{A}_{-\mbf{n}}$ via $\mbf{w}=a\mbf{w}_1+\mbf{w}_2$. In this case the $A$-degree for $\Delta_{\infty}(G)$ is defined for the torus action of the form $\mbf{w}=a\mbf{w}_1+\mbf{w}_2$, and we still denote it by $\text{deg}_{A}\Delta_{\infty}(G)$.

Similar to the positive half case in Lemma \ref{positive-half-infinite-degree-bounding:lemma}, one can also have the degree bounding for the negative half:
\begin{lem}
If $G\in\mc{B}_{\mbf{m},-\mbf{n}}^{-}$, one has that:
\begin{equation*}
\begin{aligned}
&\text{max deg}_{A}\Delta_{\infty}(G)|_{F_{\beta}\times F_{\alpha}}\leq-\mbf{m}\cdot\mbf{k}\\
&\text{min deg}_{A}\Delta_{\infty}(G)|_{F_{\beta}\times F_{\alpha}}\geq-\mbf{m}\cdot\mbf{k}.
\end{aligned}
\end{equation*}
\end{lem}
\begin{proof}
One can use the action map \ref{negative-action} and give the similar computation of $\Delta_{\infty}(G)|_{F_{\beta}\times F_{\alpha}}$ as \ref{explicit-formula-stable}. Doing the similar analysis as in the proof of Lemma \ref{positive-half-infinite-degree-bounding:lemma}. One can see that only the following term contributes to the computation of the maximal $A$-degree:
\begin{align*}
\lim_{\xi\rightarrow0}\frac{G(\xi Z_{\mbf{k}}\otimes Z_{\mbf{n}-\mbf{k}})}{\zeta(Z_{\mbf{n-k}}/\xi Z_{\mbf{k}})}
\end{align*}
and the minimal $A$-degree:
\begin{align*}
\lim_{\xi\rightarrow\infty}\frac{G(\xi Z_{\mbf{k}}\otimes Z_{\mbf{n}-\mbf{k}})}{\zeta(Z_{\mbf{n-k}}/\xi Z_{\mbf{k}})}.
\end{align*}
Thus we have that:-
\begin{align*}
\text{max deg}_{A}\text{Stab}_{\infty}^{-1}F\text{Stab}_{\infty}|_{F_{\beta}\times F_{\alpha}}\leq-\mbf{m}\cdot\mbf{k}-\langle\mbf{k},\mbf{n}-\mbf{k}\rangle+\langle\mbf{k},\mbf{n}-\mbf{k}\rangle=-\mbf{m}\cdot\mbf{k}
\end{align*}
and
\begin{align*}
\text{min deg}_{A}\text{Stab}_{\infty}^{-1}G\text{Stab}_{\infty}|_{F_{\beta}\times F_{\alpha}}\geq-\mbf{m}\cdot\mbf{n}+\mbf{m}\cdot(\mbf{n}-\mbf{k})+\langle\mbf{k},\mbf{n}-\mbf{k}\rangle-\langle\mbf{k},\mbf{n}-\mbf{k}\rangle=-\mbf{m}\cdot\mbf{k}
\end{align*}
and thus the proof is finished.
\end{proof}

\subsection{Hopf embedding of slope subalgebras}

Here we state the first main theorem, which will help us identify the root subalgebra defined in \ref{Definition-of-wall-subalgebra-in-root-subalgebra:Definition} with the wall subalgebra.

\begin{thm}\label{Hopf-embedding}
For arbitrary $\mbf{m}\in\mbb{Q}^I$, when restricted to $\mc{B}_{\mbf{m},w}$. There is a Hopf algebra embedding
\begin{align*}
(\mc{B}_{\mbf{m},w},\Delta_{\mbf{m}},S_{\mbf{m}},\eta,\epsilon)\hookrightarrow (U_{q}^{MO}(\mf{g}_{w}),\Delta_{\mbf{m}}^{MO},S_{\mbf{m}}^{MO},\eta,\epsilon).
\end{align*}
\end{thm}

\begin{proof}
The intertwining properties for unit map and counit map $\eta$ and $\epsilon$ are easy to check. The intertwining property for $S_{\mbf{m}}$ and $S_{\mbf{m}}^{MO}$ comes from the intertwining property for $\Delta_{\mbf{m}}$ and $\Delta_{\mbf{m}}^{MO}$. Thus we need to prove that the following diagram commute:
\begin{equation}\label{coproduct-commutativity}
\begin{tikzcd}
K(\mbf{w}_1)\otimes K(\mbf{w}_2)\arrow[d,"\Delta_{\mbf{m}}(F)"]\arrow[r,"\text{Stab}_{\mbf{m}}"]&K(\mbf{w}_1+\mbf{w}_2)\arrow[d,"F"]\\
K(\mbf{w}_1)\otimes K(\mbf{w}_2)\arrow[r,"\text{Stab}_{\mbf{m}}"]&K(\mbf{w}_1+\mbf{w}_2).
\end{tikzcd}
\end{equation}

i.e.
\begin{align*}
F\circ\text{Stab}_{\mbf{m}}(p_{\bm{\lambda}_1}\otimes p_{\bm{\lambda}_2})=\text{Stab}_{\mbf{m}}(\Delta_{\mbf{m}}(F)(p_{\bm{\lambda}_1}\otimes p_{\bm{\lambda}_2})).
\end{align*}

\textbf{Step I: Theorem for $\mc{B}_{\mbf{m}}$}

Let us first prove the statement for $F\in\mc{B}_{\mbf{m},\mbf{n}}^+$ for arbitrary $\mbf{m}\in\mbb{Q}^I$

If $F\in\mc{B}_{\mbf{m}}^+$, we need to prove that $\text{Stab}_{\mbf{m}}^{-1}F\text{Stab}_{\mbf{m}}$ satisfies the degree bounding that:
\begin{align*}
\text{deg}_{A}\text{Stab}_{\mbf{m}}^{-1}F\text{Stab}_{\mbf{m}}|_{F_{\alpha}\times F_{\beta}}=\langle\mu(F_{\beta})-\mu(F_{\alpha}),\mbf{m}\rangle=\mbf{m}\cdot(\mbf{n}-\mbf{k}).
\end{align*}

Now we do the following transformation:
\begin{align*}
\text{Stab}_{\mbf{m}}^{-1}F\text{Stab}_{\mbf{m}}=(R_{\mbf{m},\infty}^{+})^{-1}\text{Stab}_{\infty}^{-1}F\text{Stab}_{\infty}R_{\mbf{m},\infty}^{+},\qquad R_{\mbf{m},\infty}^{+}=\text{Stab}_{\infty}^{-1}\circ\text{Stab}_{\mbf{m}}.
\end{align*}

Note that since the matrix $R_{\mbf{m},\infty}^{+}$ is strictly upper-triangular, the upper-triangular part $R_{\mbf{m},\infty}^{+}|_{F_{\alpha}\times F_{\beta}}$ has the degree strictly smaller than $-\langle\mu(F_{\alpha})-\mu(F_{\beta}),\mbf{m}\rangle$. Therefore given the decomposition of $\text{Stab}_{\mbf{m}}^{-1}F\text{Stab}_{\mbf{m}}(a)|_{F_{\alpha}\times F_{\beta}}$ into the following component:
\begin{align*}
(R_{\mbf{m},\infty}^{+})^{-1}|_{F_{\beta'}\times F_{\beta}}(\text{Stab}_{\infty}^{-1}F\text{Stab}_{\infty})|_{F_{\alpha'}\times F_{\beta'}}\circ R_{\mbf{m},\infty}^{-}|_{F_{\alpha}\times F_{\alpha'}}.
\end{align*}

We have that its $A$-degree is given by:
\begin{equation*}
\begin{aligned}
&\text{deg}_{A}((R_{\mbf{m},\infty}^{+})^{-1}|_{F_{\beta'}\times F_{\beta}}(\text{Stab}_{\infty}^{-1}F\text{Stab}_{\infty})|_{F_{\alpha'}\times F_{\beta'}}\circ R_{\mbf{m},\infty}^{+}|_{F_{\alpha}\times F_{\alpha'}})\\
<&\mbf{m}\cdot(\mbf{l}-\mbf{k}-\bm{\delta})+\mbf{m}\cdot(\mbf{n}-\mbf{l})+\mbf{m}\cdot\bm{\delta}\\
=&\mbf{m}\cdot(\mbf{n}-\mbf{k})=\langle\mu(F_{\beta})-\mu(F_{\alpha}),\mbf{m}\rangle.
\end{aligned}
\end{equation*}

Thus on the diagonal part of $R_{\mbf{m},\infty}^{+}$, we have the degree given as:
\begin{equation*}
\begin{aligned}
&\text{deg}_{A}((R_{\mbf{m},\infty}^{+})^{-1}|_{F_{\beta}\times F_{\beta}}(\text{Stab}_{\infty}^{-1}F\text{Stab}_{\infty})|_{F_{\alpha}\times F_{\beta}}\circ R_{\mbf{m},\infty}^{+}|_{F_{\alpha}\times F_{\alpha}})\\
\leq&\mbf{m}\cdot(\mbf{n}-\mbf{k})=\langle\mu(F_{\beta})-\mu(F_{\alpha}),\mbf{m}\rangle.
\end{aligned}
\end{equation*}

Similarly, for the vice versa, we have that:
\begin{equation*}
\begin{aligned}
&\text{deg}_{A}((R_{\mbf{m},\infty}^{+})^{-1}|_{F_{\beta}\times F_{\beta}}(\text{Stab}_{\infty}^{-1}F\text{Stab}_{\infty})|_{F_{\alpha}\times F_{\beta}}\circ R_{\mbf{m},\infty}^+|_{F_{\alpha}\times F_{\alpha}})\\
\geq&\mbf{m}\cdot(\mbf{n}-\mbf{k})=\langle\mu(F_{\beta})-\mu(F_{\alpha}),\mbf{m}\rangle
\end{aligned}
\end{equation*}

\begin{equation*}
\begin{aligned}
&\text{deg}_{A}((R_{\mbf{m},\infty}^{+})^{-1}|_{F_{\beta'}\times F_{\beta}}(\text{Stab}_{\infty}^{-1}F\text{Stab}_{\infty})|_{F_{\alpha'}\times F_{\beta'}}\circ R_{\mbf{m},\infty}^+|_{F_{\alpha}\times F_{\alpha'}})\\
\geq&\mbf{m}\cdot(\bm{\delta}+\mbf{l}-\mbf{k})+\mbf{m}\cdot(\mbf{n}-\mbf{l})-\mbf{m}\cdot\bm{\delta}\\
=&\mbf{m}\cdot(\mbf{n}-\mbf{k})=\langle\mu(F_{\beta})-\mu(F_{\alpha}),\mbf{m}\rangle.
\end{aligned}
\end{equation*}

This implies that for the part that the upper-triangular part of $R_{\mbf{m},\infty}^+$ makes contribution is forced to be zero. We conclude that:
\begin{align*}
\text{deg}_{A}(\text{Stab}_{\mbf{m}}^{-1}F\text{Stab}_{\mbf{m}}(a)|_{F_{\alpha}\times F_{\beta}})=\mbf{m}\cdot(\mbf{n}-\mbf{k}).
\end{align*}

Using the fact that $\text{Stab}_{\mbf{m}}^{-1}F\text{Stab}_{\mbf{m}}(a)|_{F_{\alpha}\times F_{\beta}}$ is a Laurent polynomial in $a$ by Proposition \ref{Integrality-for-slope-coproduct:Proposition}, we conclude that $\text{Stab}_{\mbf{m}}^{-1}F\text{Stab}_{\mbf{m}}(a)|_{F_{\alpha}\times F_{\beta}}$ is a monomial in $a$.

Since the left hand side has its each component as the monomial in $a$, the identity can be written in the following way:
\begin{align*}
\text{Stab}_{\mbf{m}}^{-1}F\text{Stab}_{\mbf{m}}(a)|_{F_{\alpha}\times F_{\beta}}=\lim_{\xi\rightarrow\infty}\frac{1}{\xi^{\mbf{m}\cdot(\mbf{n}-\mbf{k})}}R_{\mbf{m},\infty}^{-1}\text{Stab}_{\infty}^{-1}F\text{Stab}_{\infty}R_{\mbf{m},\infty}(\xi a)|_{F_{\alpha}\times F_{\beta}}.
\end{align*}

Now we use the following trick, and this has been proved in Exercises 10.2.14 of \cite{O15} with simple modifications:
\begin{lem}\label{block-diagonal-lemma:lemma}
Given a slope $\mbf{m}$ such that after some integral translations $\mc{L}\in\text{Pic}(\mc{M}_{Q}(\mbf{v},\mbf{w}))$, $\mbf{m}-\mc{L}$ lies in the intersection of a small neighborhood of $0$ with the anti-ample cone $-C_{ample}\subset\text{Pic}(\mc{M}_{Q}(\mbf{v},\mbf{w}))\otimes\mbb{R}$, we have that:
\begin{align*}
\lim_{a\rightarrow\infty}a^{-\langle\mbf{m},-\rangle-\langle T^{1/2},\sigma\rangle}i^*\text{Stab}_{\mbf{m}}
\end{align*}
is a block-diagonal operator. Here $a^{-\langle\mbf{m},-\rangle}$ means that when restricted to the component $F_{\alpha}\times F_{\beta}$, the degree is given by $a^{-\langle\mbf{m},\mu(F_{\beta})-\mu(F_{\alpha})\rangle}$
\end{lem}

The lemma implies that $a^{-\langle\mbf{m},-\rangle}R_{\mbf{m},\infty}^+$ will be an identity operator as $a\rightarrow\infty$ under the scaling.

By the above computation, it implies that:
\begin{align}\label{definition-of-wall-coproduct-positive}
\text{Stab}_{\mbf{m}}^{-1}F\text{Stab}_{\mbf{m}}(a)|_{F_{\alpha}\times F_{\beta}}=\lim_{\xi\rightarrow\infty}\frac{1}{\xi^{\mbf{m}\cdot(\mbf{n}-\mbf{k})}}\text{Stab}_{\infty}^{-1}F\text{Stab}_{\infty}(\xi a)|_{F_{\alpha}\times F_{\beta}}.
\end{align}

Since $\text{Stab}_{\infty}^{-1}F\text{Stab}_{\infty}=\Delta_{\infty}(F)$, by the definition of the coproduct, and we know that from \ref{definition-of-slope-coproduct-positive}:
\begin{align*}
\Delta_{\mbf{m}}(F)|_{F_{\alpha}\times F_{\beta}}=\lim_{\xi\rightarrow\infty}\frac{1}{\xi^{\mbf{m}\cdot(\mbf{n}-\mbf{k})}}\Delta_{\infty}(F)|_{F_{\alpha}\times F_{\beta}}.
\end{align*}
Now using Theorem \ref{Commutativity-of-infty-coproduct-with-infty-stable-envelopes}, we can reach to the conclusion:
\begin{align*}
\Delta_{\mbf{m}}(F)|_{F_{\alpha}\times F_{\beta}}=\Delta_{\mbf{m}}^{MO}(F)|_{F_{\alpha}\times F_{\beta}}.
\end{align*}

For the case that $G\in\mc{B}_{\mbf{m},-\mbf{n}}^-$, one can still do the similar calculations as above. That one can find out the similar result as in \ref{definition-of-wall-coproduct-positive}:
\begin{align*}
\text{Stab}_{\mbf{m}}^{-1}(G)\text{Stab}_{\mbf{m}}|_{F_{\beta}\times F_{\alpha}}=\lim_{\xi\rightarrow0}\frac{1}{\xi^{-\mbf{m}\cdot\mbf{k}}}\text{Stab}_{\infty}^{-1}(G)\text{Stab}_{\infty}|_{F_{\beta}\times F_{\alpha}}.
\end{align*}

By the definition of the slope coproduct for negative part elements as in \ref{definition-of-slope-coproduct-negative }, we can obtain that:
\begin{align*}
\Delta_{\mbf{m}}(G)|_{F_{\beta}\times F_{\alpha}}=\Delta_{\mbf{m}}^{MO}(G)|_{F_{\beta}\times F_{\alpha}}.
\end{align*}

\textbf{Step II: Injectivity as Hopf algebras}\label{Step III:Injectivity as Hopf algebras:Step-proof}

Now we prove the injectivity of $\mc{B}_{\mbf{m},w}$ into $U_{q}^{MO}(\mf{g}_{w})$ as a Hopf algebra. It is known that for the elements in $U_{q}^{MO}(\mf{g}_{w})$, it is generated by the matrix coefficients of the wall $R$-matrix $R_{w}^{\pm,MO}$. The matrix coefficients of $R_{w}^{\pm,MO}$ can be written in the following way:
\begin{equation*}
\begin{tikzcd}
(R^{\pm,e,f}_{w})_{\mbf{w}^{aux},\mbf{w}}:&K(\mbf{0},\mbf{w}^{aux})_{loc}\otimes K(\mbf{w})_{loc}\arrow[r,"e\otimes\text{Id}"]&K(\mbf{w}^{aux})_{loc}\otimes K(\mbf{w})\\
&\arrow[r,"(R_{w}^{\pm,MO})_{\mbf{w}^{aux},\mbf{w}}"]&K(\mbf{w}^{aux})_{loc}\otimes K(\mbf{w})\arrow[r,"f\otimes\text{Id}"]&K(\mbf{0},\mbf{w}^{aux})_{loc}\otimes K(\mbf{w})_{loc}.
\end{tikzcd}
\end{equation*}
Here $e$ and $f$ are arbitrary elements in $\mc{A}_{Q}^{+}$ and $\mc{A}_{Q}^{-}$. 

Now we take $R_{w}^{+,MO}$ and consider $e\in\mc{B}_{\mbf{m},w}^+$, and we consider $(R_{w}^{-,e,1})_{\mbf{w}^{aux},\mbf{w}}$. We can further assume that $e$ is a primitive element in $\mc{B}_{\mbf{m},w}^+$. In this case the coproduct $\Delta_{\mbf{m}}$ on $e$ is written as:
\begin{align*}
\Delta_{\mbf{m}}(e)=e\otimes\text{Id}+h\otimes e.
\end{align*}
This implies that $e\otimes\text{Id}=\Delta_{\mbf{m}}(e)-h\otimes e$, and by the above computation, we know that $\Delta_{\mbf{m}}(e)=(\Delta_{\mbf{m}}^{MO})(e)$. In this case we have that:
\begin{equation*}
\begin{aligned}
&\langle v_{\emptyset},(R_{w}^{+,MO})_{\mbf{w}^{aux},\mbf{w}}(e\otimes\text{Id})v_{\emptyset}\rangle=\langle v_{\emptyset},(R_{w}^{+,MO})_{\mbf{w}^{aux},\mbf{w}}(\Delta_{\mbf{m}}^{MO}(e)-h_{\mbf{v}}\otimes e)v_{\emptyset}\rangle\\
=&\langle v_{\emptyset},\Delta_{\mbf{m}}^{MO,op}(e)(R_{w}^{+,MO})_{\mbf{w}^{aux},\mbf{w}}v_{\emptyset}\rangle-h_{\mbf{v}}e=h_{\mbf{v}}(1-h_{\mbf{v}})e.
\end{aligned}
\end{equation*}

Now since the map $\mc{A}_{Q}\mapsto\prod_{\mbf{w}}\text{End}(K(\mbf{w}))$ is injective, this implies that $e\in U_{q}^{MO}(\mf{g}_{w})$. Thus it gives the injective map $\mc{B}_{\mbf{m},w}\hookrightarrow U_{q}^{MO}(\mf{g}_{w})$.

\end{proof}

As an application, we will give another proof of the injectivity from $\mc{A}^{ext}_{Q}$ to $U_{q}^{MO}(\hat{\mf{g}}_Q)$:
\begin{prop}\label{injeticivty-from-KHA-to-MO:label}
There is an injective map of $\mbb{Q}(q,t_e)_{e\in E}$-algebras
\begin{align}
\mc{A}^{ext}_{Q}\hookrightarrow U_{q}^{MO}(\hat{\mf{g}}_{Q}).
\end{align}
\end{prop}
\textbf{Remark.} One can also refer to the proof of the Proposition in \cite{N23}.
\begin{proof}
This is equivalent to prove the injectivity of the following three pieces:
\begin{align}
\mc{A}^{\pm}_{Q}\hookrightarrow U_{q}^{MO,\pm}(\hat{\mf{g}}_{Q}),\qquad\mc{A}^{0}_{Q}\hookrightarrow U_{q}^{MO,0}(\hat{\mf{g}}_{Q}).
\end{align}

For the injectivity of the Cartan part, this is clear from the geometric action map in subsubsection \ref{subsub:localised_action_from_mc_a}. Thus we only need to prove the injectivity for the positive and negative pieces.

Recall by by Theorem \ref{Hopf-embedding}, we have the injective map $\mc{B}_{\mbf{m},w}^{\pm}\hookrightarrow U_{q}^{MO,\pm}(\mf{g}_{w})$. Recall that by the slope factorisation for KHA in Theorem \ref{Main-theorem-on-slope-factorisation:Theorem} and Lemma \ref{root-factorisation-for-slope:label} and wall factorisation for MO quantum loop groups in Proposition \ref{factorisation-for-wall-subalgebras:label}, the injective map above induce the injective map of the following form:
\begin{align}
\bigotimes^{\rightarrow}_{t\in\mbb{Q},\mbf{m}+t\bm{\theta}\in w}\mc{B}^{\pm}_{\mbf{m}+t\bm{\theta},w}\hookrightarrow\bigotimes^{\rightarrow}_{t\in\mbb{Q},\mbf{m}+t\bm{\theta}\in w}U_q^{MO,\pm}(\mf{g}_{w,\mbf{m}+t\bm{\theta}}).
\end{align}
which is exactly the injective map above.
\end{proof}

\section{Isomorphism as the integral form}\label{section:isomorphism_as_the_integral_form}
We define the integral form $\mc{A}^{ext,\mbb{Z}}_{Q}$ of the double of the extended KHA as follows:
\begin{align}\label{integral-form-double-KHA}
\mc{A}^{ext,\mbb{Z}}_{Q}:=\mc{A}^{+,\mbb{Z}}_{Q}\otimes\mc{A}^{0,\mbb{Z}}_{Q}\otimes\mc{A}^{-,\mbb{Z}}_{Q}.
\end{align}
Here $\mc{A}^{0,\mbb{Z}}_{Q}$ is the polynomial ring with the integral coefficients:
\begin{align*}
\mc{A}^{0,\mbb{Z}}_{Q}:=\mbb{Z}[q^{\pm1},t_{e}^{\pm1}]_{e\in E}[a_{i,\pm d},b_{i,\pm d},q^{\pm\frac{v_i}{2}},q^{\pm\frac{w_i}{2}}]_{i\in I,d\geq1}.
\end{align*}

Here for the positive and negative parts, we choose the following model:
\begin{align*}
\mc{A}^{+,\mbb{Z}}_{Q}:=\mc{A}^{+,\mbb{Z}}_{Q},\qquad\mc{A}^{-,\mbb{Z}}_{Q}:=(\mc{A}^{+,nilp,\mbb{Z}}_{Q})^{op}.
\end{align*}

The algebra structure of $\mc{A}^{ext,\mbb{Z}}_{Q}$ may be thought of as coming from the algebra structure of $\mc{A}^{ext}_{Q}$ and restrict the algebra embedding in \ref{injectivity-of-localised-double-KHA:theorem} to the $\mbb{Z}[q^{\pm1},t_e^{\pm1}]$-module case, which means that we have the following $\mbb{Z}[q^{\pm1},t_e^{\pm1}]$-algebra embedding:
\begin{align*}
\mc{A}^{ext,\mbb{Z}}_{Q}\hookrightarrow\prod_{\mbf{w}}\text{End}(K_{T_{\mbf{w}}}(\mc{M}_{Q}(\mbf{w}))).
\end{align*}

In this section we are going to prove the main result of the paper:
\begin{thm}\label{Main-theorem-on-integral-form:theorem}
The Maulik-Okounkov quantum loop group $U_{q}^{MO,\mbb{Z}}(\hat{\mf{g}}_{Q})$ admits the triangular decomposition:
\begin{align*}
U_{q}^{MO,\mbb{Z}}(\hat{\mf{g}}_{Q})\cong U_{q}^{MO,\mbb{Z},+}(\hat{\mf{g}}_{Q})\otimes U_{q}^{MO,\mbb{Z},0}(\hat{\mf{g}}_{Q})\otimes U_{q}^{MO,\mbb{Z},-}(\hat{\mf{g}}_{Q})
\end{align*}
such that as $\mbb{N}^I$-graded $\mbb{Z}[q^{\pm1},t_{e}^{\pm1}]_{e\in E}$-algebras, the negative half $U_{q}^{MO,\mbb{Z},-}(\hat{\mf{g}}_{Q})$ is isomorphic to $(\mc{A}^{+,\mbb{Z},nilp}_{Q})^{op}$ the opposite algebra of the nilpotent $K$-theoretic Hall algebra, and the positive half is isomorphic to $\mc{A}^{+,\mbb{Z}}_{Q}$ the preprojective $K$-theoretic Hall algebra. The Cartan part $U_{q}^{MO,\mbb{Z},0}$ is isomorphic to $\mc{A}^{0,\mbb{Z}}_{0}$.

In other words, we have the isomorphism of $\mbb{Z}[q^{\pm},t_e^{\pm1}]_{e\in E}$-algebras:
\begin{align*}
\mc{A}^{ext,\mbb{Z}}_{Q}\cong U_{q}^{MO,\mbb{Z}}(\hat{\mf{g}}_{Q}).
\end{align*}

Moreover, the above isomorphisms intertwine the action over $K_{T_{\mbf{w}}}(\mc{M}_{Q}(\mbf{w}))$.
\end{thm}

\subsection{Isomorphism of the positive integral half}\label{sub:isomorphism_of_the_positive_integral_half}
In this subsection we prove that the positive half of the integral MO quantum loop group $U_{q}^{MO,+,\mbb{Z}}(\hat{\mf{g}}_{Q})$ is isomorphic to the preprojective KHA $\mc{A}^{+,\mbb{Z}}_{Q}$.
\begin{thm}\label{isomorphism-of-integral-positive-half:theorem}
There exists an isomorphism of $\mbb{Z}[q^{\pm1},t_{e}^{\pm1}]_{e\in E}$-algebras between the positive half of Maulik-Okounkov quantum loop group and the preprojective KHA
\begin{align}
U_{q}^{MO,+,\mbb{Z}}(\hat{\mf{g}}_{Q})\cong\mc{A}^{+,\mbb{Z}}_{Q}
\end{align}
which intertwines the action over $K_{T_{\mbf{w}}}(\mc{M}_{Q}(\mbf{w}))$
\end{thm}

\begin{proof}
Since the preprojective KHA is a torsion-free $\mbb{Z}[q^{\pm1},t_e^{\pm1}]_{e\in E}$-module, the Proposition \ref{injeticivty-from-KHA-to-MO:label} can be lifted to the integral version and thus we have the injective map of $\mbb{Z}[q^{\pm1},t_e^{\pm1}]_{e\in E}$-algebras $\mc{A}^{+,\mbb{Z}}_{Q}\hookrightarrow U_{q}^{MO,+,\mbb{Z}}(\hat{\mf{g}}_{Q})$.

Now we need to introduce some further filtrations for both MO quantum loop groups and KHA:

Fix two slope points $\mbf{m},\mbf{m}'\in\mbb{Q}^I$, we denote $U_{q}^{MO,\pm,\mbb{Z}}(\hat{\mf{g}}_{Q})_{[\mbf{m},\mbf{m}']}$ the $\mbb{Z}[q^{\pm1},t_e^{\pm1}]_{e\in E}$-submodule generated by matrix coefficients of $R_{\mbf{m},\mbf{m}'}^{\pm}:=\text{Stab}_{\pm\sigma,\mbf{m}}^{-1}\circ\text{Stab}_{\pm\sigma,\mbf{m}'}$ respectively.

Similarly, one can define $\mc{A}^{\pm,\mbb{Z}}_{Q,[\mbf{m},\mbf{m}']}$ as the subspace of $\mc{A}^{\pm,\mbb{Z}}_{Q,\leq\mbf{m}'}$ such that:
\begin{equation}\label{other-way-filtration}
\begin{aligned}
&\mc{A}^{+,\mbb{Z}}_{Q,[\mbf{m},\mbf{m}']}:=\{F\in\mc{A}^{+,\mbb{Z}}_{Q,\leq\mbf{m}'}|\lim_{\xi\rightarrow0}\frac{F(\cdots,\xi z_{i1},\cdots,\xi z_{ik_i},z_{i,k_i+1},\cdots,z_{i,n_i},\cdots)}{\xi^{\mbf{m}\cdot\mbf{k}-\langle\mbf{k},\mbf{n}-\mbf{k}\rangle}}<\infty\}\\
&\mc{A}^{-}_{Q,[\mbf{m},\mbf{m}']}:=\{G\in\mc{A}^{-,\mbb{Z}}_{Q,\leq\mbf{m}'}|\lim_{\xi\rightarrow\infty}\frac{G(\cdots,\xi z_{i1},\cdots,\xi z_{ik_i},z_{i,k_i+1},\cdots,z_{i,n_i},\cdots)}{\xi^{-\mbf{m}\cdot\mbf{k}+\langle\mbf{k},\mbf{n}-\mbf{k}\rangle}}<\infty\}
\end{aligned}
\end{equation}

\begin{lem}\label{intermediate-lemma:label}
Both $U_{q}^{MO,\pm,\mbb{Z}}(\hat{\mf{g}}_{Q})_{[\mbf{m},\mbf{m}']}$ and $\mc{A}^{\pm,\mbb{Z}}_{Q,[\mbf{m},\mbf{m}']}$ are $\mbb{Z}[q^{\pm1},t_e^{\pm1}]_{e\in E}$-algebras. Moreover, we have an $\mc{A}^{\pm}_{Q,[\mbf{m},\mbf{m}']}$-algebra embedding
\begin{align*}
\mc{A}^{\pm}_{Q,[\mbf{m},\mbf{m}']}\hookrightarrow U_{q}^{MO,\pm}(\hat{\mf{g}}_{Q})_{[\mbf{m},\mbf{m}']}.
\end{align*}
\end{lem}
\begin{proof}
The fact that $\mc{A}^{\pm,\mbb{Z}}_{Q,[\mbf{m},\mbf{m}']}$ preserves the algebra structure comes from the computation and the comparison of the degrees. The fact that $U_{q}^{MO,\pm,\mbb{Z}}(\hat{\mf{g}}_{Q})_{[\mbf{m},\mbf{m}']}$ preserves the algebra structure comes from the triangle lemma of the stable envelope. Thus the algebra embedding comes naturally from Theorem \ref{Hopf-embedding}.
\end{proof}

It is easy to see that $\mc{A}^{\pm,\mbb{Z}}_{Q,[\mbf{m},\mbf{m}']}$ admits the factorisation:
\begin{align}
\mc{A}^{\pm,\mbb{Z}}_{Q,[\mbf{m},\mbf{m}']}=\bigotimes^{\rightarrow}_{\mu\in[0,1]}\mc{B}_{\mbf{m}+\mu\bm{\theta}}^{\pm,\mbb{Z}}.
\end{align}
with the choice of $\bm{\theta}=\mbf{m}'-\mbf{m}\in(\mbb{Z}^+)^{|I|}$.

Now we come to the proof of Theorem \ref{isomorphism-of-integral-positive-half:theorem}. Now we choose a dimension vector $\mbf{w}$. Since $K_{T_{\mbf{w}}}(\mc{M}_{Q}(\mbf{v},\mbf{w}))$ is a free $K_{T_{\mbf{w}}}(pt)$-module of finite rank, this implies that the surjective map from \ref{surjectivity-noetherian-lemma-localised:theorem} can be factorised via:
\begin{align}\label{use-it-again-01}
\mc{A}^{+}_{Q,\mbf{v},[\mbf{m},\mbf{m}']}\otimes\mbb{F}_{\mbf{w}}\twoheadrightarrow K_{T_{\mbf{w}}}(\mc{M}_{Q}(\mbf{v},\mbf{w}))_{loc}
\end{align}
for suitable choice of $\mbf{m},\mbf{m}'\in\mbb{Q}^I$. Now we choose the root subalgebra $\mc{B}_{\mbf{m},w,\mbf{v}}^+\subset\mc{A}^{+}_{Q,\mbf{v},[\mbf{m},\mbf{m}']}$ of slope $\mbf{m}$ and the corresponding wall subalgebra $U_{q}^{MO,+}(\mf{g}_{w})\subset U_{q}^{MO,+}(\hat{\mf{g}}_Q)_{[\mbf{m},\mbf{m}']}$ in the MO quantum loop group. By Lemma \ref{second-evaluation-injective-lemma:lemma}, we have the following commutative diagram:
\begin{equation}\label{First-important-commutative-diagram}
\begin{tikzcd}
\mc{B}_{\mbf{m},w,\mbf{v}}^+\otimes\mbb{F}_{\mbf{w}}\arrow[r,hook]\arrow[d,hook]&U_{q}^{MO,+}(\mf{g}_{w})\otimes\mbb{F}_{\mbf{w}}\arrow[d,hook]\\
\mc{A}^{+}_{Q,\mbf{v},[\mbf{m},\mbf{m}']}\otimes\mbb{F}_{\mbf{w}}\arrow[r,twoheadrightarrow]& K_{T_{\mbf{w}}}(\mc{M}_{Q}(\mbf{v},\mbf{w}))_{loc}
\end{tikzcd}
\end{equation}
with $\mbf{w}\geq\mbf{v}$.

\subsubsection{Choice of the section map}

Recall in the construction of the surjective map from \ref{surjectivity-noetherian-lemma-localised:theorem}, it is described via the following:
\begin{equation}
\begin{tikzcd}
\mc{M}_{Q}(0,\mbf{v},\mbf{w})\arrow[r,hook,"\pi"]\arrow[d,"p"]&\mc{M}_{Q}(\mbf{v},\mbf{w})\\
\mc{Y}_{\mbf{v}}
\end{tikzcd}
\end{equation}
where $\pi$ is an isomorphism after $T$-localisation.
We have used $K_{T_{\mbf{w}}}(\mc{M}_{Q}(\mbf{0},\mbf{v},\mbf{w}))$. In fact via the projection map $p:\mc{M}_{Q}(\mbf{0},\mbf{v},\mbf{w})\rightarrow\mc{Y}_{\mbf{v}}$, one can also give a slope factorisation for $K_{T_{\mbf{w}}}(\mc{M}_{Q}(\mbf{0},\mbf{v},\mbf{v}))$. It was proved in Theorem 4.1 in \cite{PT25} and here we only state the theorem in the equivariant $K$-theory case:
\begin{thm}[For the categorical version see Theorem 4.1 in \cite{PT25}]\label{categorical-slope-factorisation:Theorem}
If we choose $\bm{\theta}=(1,\cdots,1)$, via the pullback of the map $p$, we have the following slope factorisation $\mbb{N}^I$-graded $K_{T}(pt)$-module isomorphism:
\begin{align}
p^*:\mc{A}_{Q,[\eta\bm{\theta},(1+\eta)\bm{\theta}]}^{+,\mbb{Z}}\cong\bigotimes_{\mu\in[\eta,1+\eta)\cap\mbb{Q}}^{\rightarrow}\mc{B}_{\mu\bm{\theta}}^{+,\mbb{Z}}\cong\bigoplus_{\mbf{v}\in\mbb{N}^I}K_{T}(\mc{M}_{Q}(\mbf{0},\mbf{v},\mbf{w}))
\end{align}
for $\mbf{w}=\mbf{v}$ and generic $\eta\in\mbb{R}$. We denote the inverse of $p^*$ as $\psi$
\end{thm}
\begin{proof}
As we have mentioned below the Definition \ref{Slope-subalgebra-definition-derived
:label}, the quasi-BPS categories defined in \cite{PT25} is the categorified of the slope subalgebra. The proof in Theorem 4.1 of \cite{PT25} combining with the dimensional reduction argument (See for example 2.18 in \cite{PT25}) can be totally applied to the $T$-equivariant case.
\end{proof}

Now we choose a $\mu\in\mbb{Q}$ and some generic $\eta\in\mbb{R}$ such that $\mu\in[\eta,1+\eta)$. By the diagram \ref{First-important-commutative-diagram} and Lemma \ref{flow-condition:lemma}, one has the refinement of commutative diagram with the section map:
\begin{equation}\label{The-most-important-diagram}
\begin{tikzcd}
\mc{B}_{\mu\bm{\theta},w,\mbf{v}}^+\otimes K_{T_{\mbf{v}}}(pt)\arrow[r,hook]\arrow[d,hook]&U_{q}^{MO,+}(\mf{g}_{w})_{\mbf{v}}\otimes K_{T_{\mbf{v}}}(pt)\arrow[d,hook]\\
\mc{A}^+_{Q,\mbf{v},[\eta\bm{\theta},(1+\eta)\bm{\theta})}\otimes K_{T_{\mbf{v}}}(pt)\arrow[r,"\cong",shift left]&K_{T_{\mbf{v}}}(\mc{M}_{Q}(\mbf{0},\mbf{v},\mbf{v})).\arrow[l,"\psi",shift left]
\end{tikzcd}
\end{equation}

The injective map on the right hand side follows from the injective map in the Lemma \ref{second-evaluation-injective-lemma:lemma}, and the fact that the evaluation map in Lemma \ref{second-evaluation-injective-lemma:lemma} factors through $K_{T_{\mbf{v}}}(\mc{M}_{Q}(\mbf{0},\mbf{v},\mbf{v}))$ by Proposition \ref{Description-of-full-attracting-set:label}.

\subsubsection{Proof of Theorem \ref{isomorphism-of-integral-positive-half:theorem}}\label{ssub:proof_of_theorem_ref_isomorphism_of_integral_positive_half_theorem}
Now if we identify $K_{T_{\mbf{v}+\mbf{w}}}(\mc{M}_{Q}(\mbf{0},\mbf{v},\mbf{v}))$ with $K_{T_{\mbf{v}+\mbf{w}}}(\mc{M}_{Q}(\mbf{0},\mbf{v},\mbf{v})\times\mc{M}_{Q}(\mbf{0},\mbf{w}))$, and choose the torus degree $\mbf{v}+\mbf{w}=a\mbf{v}+\mbf{w}$. Given elements $F\in U_{q}^{MO,+}(\mf{g}_{w})_{\mbf{v}}$, it has been computed in Section \ref{sub:degree_bounding} that the $A$-degree of $Fv_{\emptyset}$ is given by $\mu\bm{\theta}\cdot\mbf{v}$. Now we choose $\mu$ to be the smallest rational number in $[\eta,1+\eta)$ such that $\mu\bm{\theta}\cdot\mbf{v}\in\mbb{Z}$. We claim that the image of the evaluation $\text{ev}: U_{q}^{MO,+}(\mf{g}_{w})_{\mbf{v}}\otimes K_{T_{\mbf{v}}}(pt)\hookrightarrow K_{T_{\mbf{v}}}(\mc{M}(\mbf{0},\mbf{v},\mbf{v}))$ lies in $\mc{B}_{\mu\bm{\theta},w,\mbf{v}}^{+}$. 

Now let us suppose the contradiction, which means that we can write down $Fv_{\emptyset}$ as $\sum_{\mbf{l}}a_{\mbf{l}}F_{l_1}\cdots F_{l_k}$ with $F_{l_i}\in\mc{B}_{\mu_i\bm{\theta},w,\mbf{v}_i}^{+}$. In this case the $A$-degree of the corresponding element is given by:
\begin{align*}
\bm{\theta}\cdot(\mu_1\mbf{v}_1+\cdots+\mu_k\mbf{v}_k)=\mu\bm{\theta}\cdot\mbf{v}
\end{align*}
with the constraint condition $\mbf{v}_1+\cdots+\mbf{v}_k=\mbf{v}$. Since we know that $\mu_k\geq\cdots\geq\mu_1\geq\mu$, this means that $\bm{\theta}\cdot(\mu_1\mbf{v}_1+\cdots+\mu_k\mbf{v}_k)\geq\mu\bm{\theta}\cdot\mbf{v}$. The equality holds if and only if $\mu_1=\mu_2=\cdots=\mu_k$. Thus one can see that the image of the evaluation map is given by $p^*\mc{B}_{\mu\bm{\theta},w,\mbf{v}}^+$, and by the section map $\psi$, we have the chain of the injective $K_{T_{\mbf{v}}}(pt)$-module map:
\begin{align*}
\mc{B}_{\mu\bm{\theta},w,\mbf{v}}^+\otimes K_{T_{\mbf{v}}}(pt)\hookrightarrow U_{q}^{MO,+}(\mf{g}_{w})_{\mbf{v}}\otimes K_{T_{\mbf{v}}}(pt)\hookrightarrow\mc{B}_{\mu\bm{\theta},w,\mbf{v}}^+\otimes K_{T_{\mbf{v}}}(pt)
\end{align*}
which is equal to identity after the composition of these two maps. This implies that we have an isomorphism of $K_{T_{\mbf{v}}}(pt)$-modules:
\begin{align*}
\mc{B}_{\mu\bm{\theta},w,\mbf{v}}^+\otimes K_{T_{\mbf{v}}}(pt)\cong U_{q}^{MO,+}(\mf{g}_{w})_{\mbf{v}}\otimes K_{T_{\mbf{v}}}(pt).
\end{align*}
Now since $K_{T_{\mbf{v}}}(pt)$ is faithfully flat over $K_{T}(pt)$, one has the isomorphism of $K_{T}(pt)$-modules:
\begin{align}\label{One-piece-isomorphism-important}
\mc{B}_{\mu\bm{\theta},w,\mbf{v}}^+\cong U_{q}^{MO,+}(\mf{g}_{w})_{\mbf{v}}.
\end{align}

and now by Proposition \ref{injeticivty-from-KHA-to-MO:label}, this implies the isomorphism of $\mbb{Q}(q,t_e)$-algebras:
\begin{align}
U_{q}^{MO,+}(\hat{\mf{g}}_{Q})\cong\mc{A}^{+}_{Q}.
\end{align}

Hence the proof is finished.

\end{proof}

\subsection{Isomorphism on the negative half}
The proof of the isomorphism on the negative half is totally similar to the methods of the proof for the positive half. Over this subsection we are going to use $\Delta_{\mbf{m}}^{MO}$ to stand for the nilpotent geometric coproduct $\Delta_{\mbf{m}}^{MO,\mc{L}}$ defined in \ref{nilpotent-geometric-coproduct-definition}

\begin{prop}\label{injectivity-of-integral-negative-half:label}
There is an injective map of $\mbb{Z}[q^{\pm1},t_{e}^{\pm1}]_{e\in E}$-algebras
\begin{align*}
(\mc{A}^{+,nilp,\mbb{Z}}_{Q})^{op}\hookrightarrow U_{q}^{MO,-,\mbb{Z}}(\hat{\mf{g}}_{Q})
\end{align*}
which intertwines the action over $K_{T}(\mc{M}_{Q}(\mbf{w}))$. Moreover, after identifying via the perfect pairing \ref{perfect-pairing:label}, the embedding can be refined to the bialgebra embedding when restricted to nilpotent root subalgebras and nilpotent wall subalgebras
\begin{align*}
(\mc{B}_{\mbf{m},w}^{nilp,\geq},\Delta_{\mbf{m}})\hookrightarrow (U_{q}^{MO,nilp,\geq}(\mf{g}_{w}),\Delta_{\mbf{m}}^{MO}).
\end{align*}

\end{prop}
\begin{proof}
The second part of the Proposition can be thought as the conclusion of the first part of the Proposition and Theorem \ref{Hopf-embedding}.

For the proof of the second part, this is equivalent to prove the injectivity of $\mbb{Z}[q^{\pm1},t_{e}^{\pm1}]$-algebra map
\begin{align*}
\mc{A}^{+,nilp,\mbb{Z}}_{Q}\hookrightarrow U_{q}^{MO,+,nilp,\mbb{Z}}(\hat{\mf{g}}_{Q}).
\end{align*}

Since both sides are generated by the subalgebras $\mc{B}_{\mbf{m},w}^{+,nilp,\mbb{Z}}$ and $U_{q}^{MO,nilp,+}(\mf{g}_{w})$ respectively by \ref{slope-factoristaion-integral}, \ref{integral-root-subalgebra} and Section \ref{subsection:factorisation_of_geometric_r_matrices_and_integral_maulik_okounkov_quantum_affine_algebras}. The injectivity can be deduced from the injectivity of the following map:
\begin{align}\label{nilpotent-slope-embedding}
\mc{B}_{\mbf{m},w}^{+,nilp,\mbb{Z}}\hookrightarrow U_{q}^{MO,nilp,+}(\mf{g}_{w}).
\end{align}

This can be done by evaluating the nilpotent geometric $R$-matrix in the following:
\begin{equation}\label{nilpotent-matrix-elements}
\begin{tikzcd}
&R^{+,\mc{L}}_{w,e,1}:K_{T}(\mc{L}_{Q}(0,\mbf{w}_{aux})\times\mc{L}_{Q}(\mbf{v},\mbf{w}))\arrow[r,"e\otimes\text{Id}"]&K_{T}(\mc{L}_{Q}(\mbf{n},\mbf{w}_{aux})\times\mc{L}_{Q}(\mbf{v},\mbf{w}))\\
\arrow[r,"(R^{-,\mc{L}}_{w})^{-1}"]&K_{T}(\mc{L}_{Q}(0,\mbf{w}_{aux})\times\mc{L}_{Q}(\mbf{v}+\mbf{n},\mbf{w}))
\end{tikzcd},\qquad e\in\mc{B}^{+,nilp,\mbb{Z}}_{\mbf{m},w}.
\end{equation}

Or we can write it in the following matrix way:
\begin{align*}
\langle v_{\emptyset},(R_{w}^{-,\mc{L}})^{-1}(e\otimes\text{Id})v_{\emptyset}\rangle,\qquad e\in\mc{B}^{+,nilp,\mbb{Z}}_{\mbf{m},w}.
\end{align*}

Now we take the shuffle coproduct $\Delta_{\mbf{m}}$ on $e$ via the identifying $e$ as an element in $\mc{A}^{+}_{Q}$, and by Theorem \ref{Hopf-embedding} and Proposition \ref{transpose-Hopf-algebra:proposition}, one can write down $\Delta_{\mbf{m}}(e)$ as $\Delta_{\mbf{m}}^{MO,op}(e)$ with the geometric coproduct $\Delta_{\mbf{m}}^{MO}$ defined over $U_{q}^{MO,nilp}(\hat{\mf{g}}_{Q})$. While we know that:
\begin{align}\label{coproduct-for-nilpotent-elements}
\Delta_{\mbf{m}}(e)=e\otimes\text{Id}+\sum_{i}h_{i}e_{i}'\otimes e_{i}'',\qquad e_{i}',e_{i}''\in\mc{B}_{\mbf{m},\mbf{n}'}^{+,nilp},\mbf{n}'<\mbf{n}
\end{align}
and here $h_{i}$ are the Cartan elements of the form $h_{\mbf{n}}$. Using the following relations:
\begin{align*}
(R_{w}^{-,\mc{L}})\Delta_{\mbf{m}}^{MO,op}(e)=\Delta_{\mbf{m}}^{MO}(e)(R_{w}^{-,\mc{L}}).
\end{align*}

By the definition of $\Delta_{\mbf{m}}$ and Theorem \ref{Integral-nilpotent-KHA-coproduct:label}, and also notice that the nilpotent geometric coproduct is the tranpose of the original geometric coproduct as in \ref{geometric-coproduct-transpose}. Thus we can see that $\Delta_{\mbf{m}}(e)=\Delta_{\mbf{m}}^{MO}(e)$ has the image in $\mc{B}_{\mbf{m}}^{\geq,nilp,\mbb{Z}}\otimes\mc{B}_{\mbf{m}}^{\leq,nilp,\mbb{Z}}$. Thus we have that $e_{i}',e_{i}''\in\mc{B}^{+,nilp,\mbb{Z}}_{\mbf{m},\mbf{n}'}$.

We now have that:
\begin{equation*}
\begin{aligned}
\langle v_{\emptyset},(R_{w}^{-,\mc{L}})(e\otimes\text{Id})v_{\emptyset}\rangle=&\langle v_{\emptyset},(R_{w}^{-,\mc{L}})\Delta^{MO,op}_{\mbf{m}}(e)v_{\emptyset}\rangle-\sum_{i}\langle v_{\emptyset},(R_{w}^{-,\mc{L}})h_ie_i'v_{\emptyset}\rangle e_i''\\
=&\langle v_{\emptyset},\Delta^{MO}_{\mbf{m}}(e)(R^{-,\mc{L}}_{w})v_{\emptyset}\rangle-\sum_{i}\langle v_{\emptyset},(R_{w}^{-,\mc{L}})h_ie_i'v_{\emptyset}\rangle e_i''\\
=&\langle v_{\emptyset},\Delta^{MO}_{\mbf{m}}(e)v_{\emptyset}\rangle-\sum_{i}\langle v_{\emptyset},(R_{w}^{-,\mc{L}})^{-1}e_i'v_{\emptyset}\rangle e_i''.
\end{aligned}
\end{equation*}
Now using the induction on the degrees, we have that the second term on the right hand side above belongs to $\mc{A}^{+,nilp,\mbb{Z}}_{Q}$, and the first term is equal to $e$, thus we obtain that $\langle v_{\emptyset},(R^{-,\mc{L}}_{w})^{-1}(e\otimes\text{Id})v_{\emptyset}\rangle\in\mc{A}^{+,nilp,\mbb{Z}}$. Therefore we obtain the embedding \ref{nilpotent-slope-embedding}.

\end{proof}

Now we can strengthen the proposition to the isomorphism:
\begin{thm}\label{Isomorphism-main-theorem-on-negative-half:label}
There is an isomorphism $\mbb{Z}[q^{\pm1},t_{e}^{\pm1}]_{e\in E}$-algebras
\begin{align*}
(\mc{A}^{+,nilp,\mbb{Z}}_{Q})^{op}\cong U_{q}^{MO,-,\mbb{Z}}(\hat{\mf{g}}_{Q})
\end{align*}
intertwining the action over $\bigoplus_{\mbf{v}\in\mbb{N}^I}K_{T_{\mbf{w}}}(\mc{M}_{Q}(\mbf{v},\mbf{w}))$. Moreover, we have an isomorphism of graded $\mbb{Z}[q^{\pm1},t_{e}^{\pm1}]$-modules:
\begin{align*}
(\mc{A}^{+,nilp,\mbb{Z}}_{Q})^{op}\cong U_{q}^{MO,-,\mbb{Z}}(\hat{\mf{g}}_{Q}).
\end{align*}
\end{thm}

\begin{proof}
By the perfect pairing \ref{perfect-pairing:label}, this is equivalent to prove the isomorphism of the following $\mbb{Z}[q^{\pm1},t_{e}^{\pm1}]$-algebras
\begin{align*}
\mc{A}^{+,nilp,\mbb{Z}}_{Q}\rightarrow U_{q}^{MO,+,nilp,\mbb{Z}}(\hat{\mf{g}}_{Q})
\end{align*}
and the isomorphism of graded $\mbb{Z}[q^{\pm1},t_{e}^{\pm1}]$-modules:
\begin{align*}
\mc{A}^{+,nilp,\mbb{Z}}_{Q}\cong U_{q}^{MO,+,nilp,\mbb{Z}}(\hat{\mf{g}}_{Q}).
\end{align*}

Similar to the proof of Theorem \ref{isomorphism-of-integral-positive-half:theorem}, we have the nilpotent analog  \ref{other-way-filtration} and Lemma \ref{intermediate-lemma:label}. Thus one can obtain the nilpotent analog of the commutative diagram \ref{The-most-important-diagram}:
\begin{equation}
\begin{tikzcd}
\mc{B}_{\mu\bm{\theta},\mbf{v}}^{+,nilp,\mbb{Z}}\otimes K_{T_{\mbf{w}}}(pt)\arrow[r,hook]\arrow[d,hook]&U_{q}^{MO,+,nilp,\mbb{Z}}(\mf{g}_{w})_{\mbf{v}}\otimes  K_{T_{\mbf{w}}}(pt)\arrow[d,hook]\\
\mc{A}^{+,nilp,\mbb{Z}}_{Q,\mbf{v},[\eta\bm{\theta},(1+\eta)\bm{\theta}]}\otimes  K_{T_{\mbf{w}}}(pt)\arrow[r,"\cong",shift left]& K_{T_{\mbf{w}}}(\mc{L}_{Q}(\mbf{v},\mbf{w}))\arrow[l,"\psi_{\mc{L}}",shift left]
\end{tikzcd}
\end{equation}
for $\mbf{w}=\mbf{v}$, with  $\mu\in\mbb{Q}$ and some generic $\eta\in\mbb{R}$ such that $\mu\in[\eta,1+\eta)$ and it is the smallest nontrivial rational number in this interval. Here the isomorphism at the bottom is tha nilpotent analog of the Theorem 4.1 in \cite{PT25}, where the argument over there can be replaced by the matrix factorisation category with the nilpoetnt support. The rest of the argument is totally the same as the one in Section \ref{ssub:proof_of_theorem_ref_isomorphism_of_integral_positive_half_theorem}. We conclude the isomorphism of $\mbb{Z}[q^{\pm1},t_e^{\pm1}]_{e\in E}$-modules:
\begin{align}
\mc{B}_{\mu\bm{\theta},\mbf{v}}^{+,nilp,\mbb{Z}}\cong U_{q}^{MO,+,nilp,\mbb{Z}}(\mf{g}_{w})_{\mbf{v}}.
\end{align}

Thus we conclude the isomorphism as $\mbb{Z}[q^{\pm1},t_{e}^{\pm1}]$-algebras:
\begin{align*}
\mc{A}^{+,nilp,\mbb{Z}}_{Q}\cong U_{q}^{MO,+,nilp,\mbb{Z}}(\hat{\mf{g}}_{Q}).
\end{align*}

Hence the proof is finished.
\end{proof}

\subsection{Proof of the main theorem}
Now we come to the proof of the main theorem \ref{Main-theorem-on-integral-form:theorem}. The triangular decomposition of $U_{q}^{MO,\mbb{Z}}(\hat{\mf{g}}_{Q})$ has been stated in Section \ref{subsubsection:integral_maulik_okounkov_quantum_affine_algebras}. Now combining Theorem \ref{Isomorphism-main-theorem-on-negative-half:label} and Theorem \ref{isomorphism-of-integral-positive-half:theorem}, we obtain the isomorphism of the positive half and the negative half. For the Cartan part, one can observe that 
\begin{equation*}
\begin{tikzcd}
R^{e,f}_{\mbf{w}^{aux},\mbf{w}}:&K(\mbf{0},\mbf{w}^{aux})_{loc}\otimes K(\mbf{w})_{loc}\arrow[r,"e\otimes\text{Id}"]&K(\mbf{w}^{aux})_{loc}\otimes K(\mbf{w})\\
&\arrow[r,"(\mc{R}_{\mbf{w}^{aux},\mbf{w}}^{\mbf{m}'})"]&K(\mbf{w}^{aux})_{loc}\otimes K(\mbf{w})\arrow[r,"f\otimes\text{Id}"]&K(\mbf{0},\mbf{w}^{aux})_{loc}\otimes K(\mbf{w})_{loc}.
\end{tikzcd}
\end{equation*}

It is known that if $e=f=\text{Id}$, the matrix coefficients correspond to the following expansion \cite{N23}:
\begin{align*}
\langle v_{\emptyset},(\mc{R}^{\mbf{m}'}_{\mbf{w}^{aux},\mbf{w}})v_{\emptyset}\rangle=\prod_{i\in I}\prod_{a=1}^{w_{i}^{aux}}q^{\frac{v_i}{2}}\wedge^*(\frac{(1-q)\mc{V}_i}{a_{ik}})\otimes(-).
\end{align*}

By the formula in \ref{infinite-slope-R-matrix} with \ref{Normal-bundle-formula} and \ref{Positive-half-normal-bundle-formula}, we know that the infinite slope $R$-matrix $\mc{R}^{\infty}_{\mbf{w}^{aux},\mbf{w}}$ is generated by the tautological classes $p_{d}((q-1)\mc{V}_{i})$. Thus we have matched the Cartan part of $U_{q}^{MO,0}(\hat{\mf{g}}_{Q})$ with $\mc{A}^{ext,0}_{Q}$ defined in \ref{Cartan-part-double-KHA-localised}.

Since $U_{q}^{MO,0,\mbb{Z}}(\hat{\mf{g}}_{Q})$ is generated by the matrix coefficients of $\mc{R}^{\infty}$ in \ref{infinite-slope-R-matrix}, using the formula \ref{Normal-bundle-formula} and \ref{Positive-half-normal-bundle-formula}. One can see that the coefficients are generated by $p_d(\mc{V}_i(1-q^{-1}))$ and $p_d(\mc{W}_i(1-q^{-1}))$, which is just the same as the case in \ref{main-theorem-just-localised-algebra}. Thus the proof of Theorem \ref{Main-theorem-on-integral-form:theorem} is finished.

\subsection{Isomorphism on localised wall subalgebras}\label{subsection:isomorphism_on_localised_wall_subalgebras}

The above theorem \ref{main-theorem-just-localised-algebra} implies the following result:
\begin{prop}\label{isomorphism-of-localised-wall-subalgebra:proposition}
There is an $\mbb{Q}(q,t_e)_{e\in E}$-Hopf algebra isomorphism:
\begin{align*}
(\mc{B}_{\mbf{m},w},\Delta_{\mbf{m}},S_{\mbf{m}},\epsilon,\eta)\cong (U_{q}^{MO}(\mf{g}_{w}),\Delta_{\mbf{m}}^{MO},S_{\mbf{m}}^{MO},\epsilon,\eta)
\end{align*}
which intertwines over the action over $K(\mbf{w})$.
\end{prop}
\begin{proof}
It is enough to prove the proposition for the positive or negative half on both sides. For simplicity, we will only show the proof for the positive half.

By Theorem \ref{Main-theorem-on-slope-factorisation:Theorem} and Proposition \ref{factorisation-for-wall-subalgebras:label}, the isomorphism in Theorem \ref{Main-theorem-on-integral-form:theorem} can be factorised as: 
\begin{align*}
\mc{A}^{+}_{Q}\cong\bigotimes^{\rightarrow}_{t\in\mbb{Q},\mbf{m}+t\bm{\theta}\in w}\mc{B}^{+}_{\mbf{m}+t\bm{\theta},w}\hookrightarrow\bigotimes^{\rightarrow}_{t\in\mbb{Q},\mbf{m}+t\bm{\theta}\in w}U_{q}^{MO,+}(\mf{g}_{\mbf{m}+t\bm{\theta},w})\cong U_{q}^{MO,+}(\hat{\mf{g}}_{Q})
\end{align*}
and note that the first arrow will not be a surjection if there is one pair of $\mc{B}^{+}_{\mbf{m},w}$ that is not surjective to $U_{q}^{MO,+}(\mf{g}_{\mbf{m},w})$. But the above map is an isomorphism by Theorem \ref{main-theorem-just-localised-algebra}. Therefore every map between $\mc{B}^{+}_{\mbf{m},w}$ and $U_{q}^{MO,+}(\mf{g}_{\mbf{m},w})$ should be surjective. Thus we have finished the proof.
\end{proof}

Recall that $\mc{B}_{\mbf{m},w}$ can be realised as the Drinfeld pairing between $\mc{B}_{\mbf{m},w}^{\leq}$ and $\mc{B}_{\mbf{m},w}^{\geq}$ in Proposition \ref{Drinfeld-pairing-to-slope-subalgebra:prop}. The corresponding universal $R$-matrix will be denoted as $R_{\mbf{m},w}^{+}$.

Moreover, one can refine the result by identifying the universal $R$-matrices on both sides:
\begin{prop}\label{identification-of-wall-R-matrices:label}
We have the following identity on $\mc{B}_{\mbf{m},w}\hat{\otimes}\mc{B}_{\mbf{m},w}$:
\begin{align*}
R_{w}^{-,MO}=(R_{\mbf{m},w}^{+})^{-1}
\end{align*}
and $R_{w}^{-,MO}:=R_{w}^{-}q^{\Omega}$ stands for the wall $R$-matrix in the definition \ref{wall-R-matrix-definition} with the multiplication $q^{\Omega}$. 
\end{prop}
\begin{proof}
Note that the isomorphism \ref{isomorphism-of-localised-wall-subalgebra:proposition} and Theorem \ref{primitivity:Theorem} and Proposition \ref{primitivity-of-wall-subalgebra:proposition} imply that we have the isomorphism on the $\mbb{Q}(q,t_e)$-modules of the primitive parts:
\begin{align*}
\mc{B}^{\pm,prim}_{\mbf{m},w}\cong U_{q}^{MO,\pm,prim}(\mf{g}_{w}).
\end{align*}
On the other hand, since the universal $R$-matrix $R^{+}_{\mbf{m},w}$ is independent of the choice of the basis in $\mc{B}^{\pm,prim}_{\mbf{m},w}$, it is equivalent to say that we can choose the corresponding suitable basis in $U_{q}^{MO,\pm,prim}(\mf{g}_{w})$. 

By the grading on $\mc{B}^{\pm,prim}_{\mbf{m},w}=\bigoplus_{\mbf{v}\in\mbb{N}^I}\mc{B}^{\pm,prim}_{\mbf{m},w,\mbf{v}}$, one could write down the universal $R$-matrix $(R^{+}_{\mbf{m},w})^{-1}$ as:
\begin{align*}
(R^{+}_{\mbf{m},w})^{-1}=q^{\Omega}(\text{Id}+\sum_{\mbf{v}\in\mbb{N}^I}R^{+}_{\mbf{m},w})_{\mbf{v}}).
\end{align*}
Now choosing arbitrary $E\in\mc{B}^{\pm,prim}_{\mbf{m},w,\mbf{v}}\cong U_{q}^{MO,\pm,prim}(\mf{g}_{w})_{\mbf{v}}$. Similar to the result in \ref{commutation-relation-for-primitive-elements}, we have that:
\begin{align*}
[(1\otimes E),(R^{+}_{\mbf{m},w})_{\mbf{v}}]=E\otimes (h_{\mbf{v}}-h_{-\mbf{v}})
\end{align*}
which implies that:
\begin{align*}
[(1\otimes E),R_{w,\mbf{v}}^{-}-(R^{+}_{\mbf{m},w})_{\mbf{v}}]=0
\end{align*}
for arbitrary primitive vectors $E$. Now we denote $S_{\mbf{v}}:=R_{w,\mbf{v}}^{-}-(R^{+}_{\mbf{m},w})_{\mbf{v}}$ and $S=\sum_{\mbf{v}\in\mbb{N}^I}S_{\mbf{v}}$, since by Theorem \ref{primitivity:Theorem} and Proposition \ref{primitivity-of-wall-subalgebra:proposition}, $\mc{B}_{\mbf{m},w}^{\pm}\cong U_{q}^{MO,\pm}(\mf{g}_{w})$ is generated by the primitive elements, we have that for arbitrary elements $L\in\mc{B}_{\mbf{m},w}^{+}\cong U_{q}^{MO,+}(\mf{g}_{w})$, we have that:
\begin{align*}
[(1\otimes L),S]=0,\qquad\forall S\in\mc{B}_{\mbf{m},w}^{+}.
\end{align*}
On the other side, one can also do the similar proof as above to show that for arbitrary $M\in\mc{B}_{\mbf{m},w}^{-}\cong U_{q}^{MO,-}(\mf{g}_{w})$, we have that:
\begin{align*}
[(M\otimes 1),S]=0,\qquad\forall M\in\mc{B}_{\mbf{m},w}^{-}
\end{align*}
and thus we have that $S$ is a constant operator concentrating on degree $\mbf{0}$, and by definition $S=0$. Thus the proof is finished.
\end{proof}

Now combining the above two propositions, we have the following theorem:
\begin{thm}\label{isomorphism-of-slope-wall-quasi-triangular:label}
There is a quasi-triangular $\mbb{Q}(q,t_e)_{e\in E}$-Hopf algebra isomorphism:
\begin{align*}
(\mc{B}_{\mbf{m},w},R_{\mbf{m},w}^+,\Delta_{\mbf{m}},S_{\mbf{m}},\epsilon,\eta)\cong (U_{q}^{MO}(\mf{g}_{w}),q^{-\Omega}(R_{w}^{-})^{-1},\Delta_{\mbf{m}}^{MO},S_{\mbf{m}}^{MO},\epsilon,\eta)
\end{align*}
which intertwines over the action over $K(\mbf{w})$.
\end{thm}

\subsubsection{Integrality for the slope $R$-matrices}
One of the interesting result of Theorem \ref{isomorphism-of-slope-wall-quasi-triangular:label} is that one can prove that the evaluation of the universal $R$-matirx $R_{\mbf{m}}^{+}$ for the localised slope subalgebra $\mc{B}_{\mbf{m}}$ on the modules $K(\mbf{w}_1)\otimes K(\mbf{w}_2)$ can be lifted to its integral form:
\begin{prop}\label{integrality-for-R-matrix:label}
Let $(\pi_{\mbf{w}_1}\otimes\pi_{\mbf{w}_2})(R_{\mbf{m}}^+)$ be the universal $R$-matrix of $\mc{B}_{\mbf{m}}$ valued in $K(\mbf{w}_1)\otimes K(\mbf{w}_2)$. Then it can be lifted to the integral form $K_{T_{\mbf{w}_1}}(\mc{M}_{Q}(\mbf{w}_1))\otimes K_{T_{\mbf{w}_2}}(\mc{M}_{Q}(\mbf{w}_2))$.
\end{prop}
\begin{proof}
By definition, we know that $\mc{B}_{\mbf{m}}$ is generated by the root subalgebra $\mc{B}_{\mbf{m},w}$. This implies that the universal $R$-matrix $R_{\mbf{m}}^+$ can be written as the ordered product of the universal $R$-matrix $R_{\mbf{m},w}^+$ for $\mc{B}_{\mbf{m},w}$:
\begin{align}\label{factorisation-of-slope-R-matrices-into-root-R-matrices}
R_{\mbf{m}}^+=\prod_{w}R_{\mbf{m},w}^+.
\end{align}
Now we denote the decomposition of $R_{\mbf{m}}^+$ by degree as:
\begin{align}
R_{\mbf{m}}^+=\text{Id}+\sum_{\mbf{n}\in\mbb{N}^I}R_{\mbf{m},\mbf{n}}^+.
\end{align}
By Theorem \ref{isomorphism-of-slope-wall-quasi-triangular:label}, this can be expressed as the composition of $q^{-\Omega}(R_{w,\mbf{n}}^-)^{-1}$ from the wall $R$-matrices, which is an integral $K$-theory class.

Next observe that when we restricted to each weight pieces $K(\mbf{v}_1,\mbf{w}_1)\otimes K(\mbf{v}_2,\mbf{w}_2)$, only finitely many walls in the product \ref{factorisation-of-slope-R-matrices-into-root-R-matrices}, therefore each $R_{\mbf{m},\mbf{n}}^+$ is an integral $K$-theory class, thus the proposition is proved.

\end{proof}

\subsection{Localised isomorphism as the Hopf algebra}

In this subsection we prove the isomorphism of the MO quantum loop group $U_{q}^{MO}(\hat{\mf{g}}_{Q})$ and the extended double KHA $\mc{A}^{ext}_Q$ in the localised form, i.e. as the Hopf $\mbb{Q}(q,t_e)_{e\in E}$-algebras:

\begin{thm}\label{main-theorem-just-localised-algebra}
There exists an isomorphism of Hopf $\mbb{Q}(q,t_e)_{e\in E}$-algebras between the Maulik-Okounkov quantum loop group and the extended double KHA
\begin{align*}
(U_{q}^{MO}(\hat{\mf{g}}_{Q}),\Delta^{MO,op}_{\mbf{m}},S_{\mbf{m}},\epsilon,\eta)\cong(\mc{A}_{Q}^{ext},\Delta_{(\mbf{m})},S_{\mbf{m}},\epsilon,\eta)
\end{align*}
which intertwines the action over $K(\mbf{w})$. Here the coproduct $\Delta_{(\mbf{m})}$ is defined in \ref{definition-of-m-universal-coproduct}.
\end{thm}

\begin{proof}
We first define the coproduct $\Delta_{(\mbf{m})}$ on $\mc{A}^{ext}_{Q}$.
\subsubsection{Coproduct on $\mc{A}^{ext}_{Q}$}
Recall from the paragraph below the Proposition \ref{Drinfeld-pairing-to-slope-subalgebra:prop}, we denote the reduced universal $R$-matrix for the slope subalgebra $\mc{B}_{\mbf{m}}$ as $R_{\mbf{m}}'$. In this way, fix the slope $\mbf{m}\in\mbb{Q}^{I}$, one can define the coproduct $\Delta_{(\mbf{m})}$ on $\mc{A}^{ext}_{Q}$ by:
\begin{equation}\label{definition-of-m-universal-coproduct}
\begin{aligned}
\Delta_{(\mbf{m})}(F)=&[\prod^{\rightarrow}_{\mu\in\mbb{Q}_{>0}\sqcup\{\infty\}}(R_{\mbf{m}+\mu\bm{\theta}}^-)^{-1}]\cdot\Delta(F)[\prod^{\rightarrow}_{\mu\in\mbb{Q}_{>0}\sqcup\{\infty\}}(R_{\mbf{m}+\mu\bm{\theta}}^-)^{-1}]^{-1}\\
=&[\prod^{\rightarrow}_{\mu\in\mbb{Q}_{>0}\sqcup\{\infty\}}(\prod^{\rightarrow}_{w}R_{\mbf{m}+\mu\bm{\theta},w}^{-})^{-1}]\cdot\Delta(F)[\prod^{\rightarrow}_{\mu\in\mbb{Q}_{>0}\sqcup\{\infty\}}(\prod^{\rightarrow}_{w}R_{\mbf{m}+\mu\bm{\theta},w}^{-})^{-1}]^{-1}.
\end{aligned}
\end{equation}

The definition was given in \cite{N22}, and when $\Delta_{(\mbf{m})}$ is restricted to $\mc{B}_{\mbf{m}}$, the coproduct $\Delta_{(\mbf{m})}$ is equal to $\Delta_{\mbf{m}}$ on $\mc{B}_{\mbf{m}}$ as defined in \ref{coproduct-m-slope-subalgebra-positive} and \ref{coproduct-m-slope-subalgebra-negative}, and it has been proved in \cite{Z24}.

\subsubsection{Matching the coproduct $\Delta_{\mbf{m}}^{MO}$ on $U_{q}^{MO}(\hat{\mf{g}}_{Q})$}
On the other hand, by the factorisation property of the stable envelope \ref{factorisation-geometry}, the coproduct $\Delta_{\mbf{m}}^{MO}$ defined in \ref{coproduct-geometric} on $U_{q}^{MO}(\hat{\mf{g}}_{Q})$ can be interpreted as:
\begin{align*}
\Delta_{\mbf{m}}^{MO}(F)=[\prod_{w>\mbf{m}}(R_{w}^{+})]\Delta_{\infty}^{MO}(F)[\prod_{w>\mbf{m}}(R_{w}^{+})]^{-1}.
\end{align*}

Now by the result of Theorem \ref{isomorphism-of-slope-wall-quasi-triangular:label}, we have that:
\begin{align*}
\Delta^{MO}_{\mbf{m}}(F)=\Delta_{(\mbf{m})}(F).
\end{align*}

Since the antipode map structure $S_{\mbf{m}}$ are induced from the coproduct structure. Combining the Theorem \ref{Main-theorem-on-integral-form:theorem}, We have finished the proof of Theorem \ref{main-theorem-just-localised-algebra}.

\end{proof}

\subsection{Isomorphism as the integral Hopf algebras}
Now we can define an integral form of the slope subalgebra $\mc{B}_{\mbf{m}}^{\mbb{Z}}$ as follows:
\begin{align*}
\mc{B}_{\mbf{m}}^{\mbb{Z}}:=\mc{B}_{\mbf{m}}^{+,\mbb{Z}}\otimes\mbb{Z}[q^{\pm1},t_{e}^{\pm}][h_{i,\pm0}]_{i\in I}\otimes(\mc{B}_{\mbf{m}}^{+,nilp,\mbb{Z}})^{op}
\end{align*}
and similarly for the integral root subalgebra $\mc{B}_{\mbf{m},w}^{\mbb{Z}}$ inside of the slope subalgebra:
\begin{align*}
\mc{B}_{\mbf{m},w}^{\mbb{Z}}:=\mc{B}_{\mbf{m},w}^{+,\mbb{Z}}\otimes\mbb{Z}[q^{\pm1},t_{e}^{\pm}][h_{i,\pm0}]_{i\in I}\otimes(\mc{B}_{\mbf{m},w}^{+,nilp,\mbb{Z}})^{op}
\end{align*}
for a wall $w$ which contains the point $\mbf{m}$.

Now combining Theorem \ref{Main-theorem-on-integral-form:theorem}, Theorem \ref{isomorphism-of-slope-wall-quasi-triangular:label} and Proposition \ref{integrality-for-R-matrix:label}, we can obtain the isomorphism of the integral root subalgebra and the integral wall subalgebra:
\begin{prop}\label{isomorphism-of-integral-wall-subalgebra:proposition}
There is an embedding of $\mbb{Z}[q^{\pm1},t_{e}^{\pm1}]$-Hopf algebras:
\begin{align*}
(\mc{B}_{\mbf{m},w}^{\mbb{Z}},\Delta_{\mbf{m}},S_{\mbf{m}},\epsilon,\eta)\hookrightarrow (U_{q}^{MO,\mbb{Z}}(\mf{g}_{w}),\Delta_{\mbf{m}}^{MO,op},S_{\mbf{m}}^{MO},\epsilon,\eta)
\end{align*}
of the same graded rank, which intertwines the action over $K_{T_{\mbf{w}}}(\mc{M}_{Q}(\mbf{w}))$. Moreover, it is an isomorphism of quasi-triangular $\mbb{Z}[q^{\pm1},t_{e}^{\pm1}]$-Hopf algebras 
\begin{align*}
(\mc{B}_{\mbf{m},w}^{\mbb{Z}},R_{\mbf{m},w}^+,\Delta_{\mbf{m}},S_{\mbf{m}},\epsilon,\eta)\cong (U_{q}^{MO,\mbb{Z}}(\mf{g}_{w}),q^{-\Omega}(R_{w}^-)^{-1},\Delta_{\mbf{m}}^{MO,op},S_{\mbf{m}}^{MO},\epsilon,\eta)
\end{align*}

\end{prop}
\begin{proof}
The isomorphism as $\mbb{Z}[q^{\pm1},t_{e}^{\pm1}]$-algebras comes from Theorem \ref{Main-theorem-on-integral-form:theorem} and Proposition \ref{isomorphism-of-localised-wall-subalgebra:proposition}. The isomorphism as Hopf algebras comes from the Hopf embedding in Theorem \ref{Hopf-embedding}.
\end{proof}

Now combining Theorem \ref{main-theorem-just-localised-algebra} and Proposition \ref{isomorphism-of-integral-wall-subalgebra:proposition}. We obtain the isomorphism as $\mbb{Z}[q^{\pm1},t_{e}^{\pm1}]_{e\in E}$-algebras:
\begin{thm}\label{big-boss-theorem:label}
There is an isomorphism of Hopf $\mbb{Z}[q^{\pm1},t_{e}^{\pm1}]_{e\in E}$-algebras between the Maulik-Okounkov quantum loop group and the integral extended double KHA $\mc{A}^{ext,\mbb{Z}}_{Q}$ defined in \ref{integral-form-double-KHA}:
\begin{align}
(U_{q}^{MO,\mbb{Z}}(\hat{\mf{g}}_{Q}),\Delta_{\mbf{m}}^{MO,op},S_{\mbf{m}},\epsilon,\eta)\cong(\mc{A}^{ext,\mbb{Z}}_{Q},\Delta_{(\mbf{m})},S_{\mbf{m}},\epsilon,\eta).
\end{align}
\end{thm}

\subsection{Freeness of the preprojective KHA}
In this subsection we use our main result to prove the freeness of the preprojective KHA with arbitrary equivariant parametres, which can be thought of as the equivariant $K$-theory analog of the freeness of preprojective CoHA result in \cite{Dav23}.

\begin{thm}\label{freeness-of-preKHA:label}
Given $\mbb{C}_{q}^*\subset A\subset T$ a subtorus of $T$ which contains $\mbb{C}_{q}^*$. The $A$-equivariant $K$-theory $K_{A}(\mc{Y}_{\mbf{v}})$ of the preprojective stack is a free $K_{A}(pt)$-module.
\end{thm}
\begin{proof}
Let us first sketch the proof for the case when $A=T$. By \ref{slope-factoristaion-integral}, it remains to prove the freeness of $\mc{B}_{\mbf{m},\mbf{v}}^{+,\mbb{Z}}$. By Theorem \ref{root-factorisation-for-slope:label}, it remains to prove the freeness of each root pieces $\mc{B}_{\mbf{m},w,\mbf{v}}^{+,\mbb{Z}}$. Since we have known that by \ref{One-piece-isomorphism-important} we have the isomorphism $\mc{B}_{\mbf{m},w,\mbf{v}}^{+,\mbb{Z}}\cong U_{q}^{MO,+,\mbb{Z}}(\mf{g}_{w})_{\mbf{v}}$. By Lemma \ref{freeness-of-positive-half:label}, $U_{q}^{MO,+,\mbb{Z}}(\mf{g}_{w})_{\mbf{v}}$ is a free $K_{T}(pt)$-module. 

For the general $A$, note that the Lemma \ref{freeness-of-positive-half:label} still holds true for $A$ containing $\mbb{C}_{q}^*$. The integral slope subalgebra $\mc{B}_{\mbf{m},\mbf{v}}^{+,\mbb{Z}}$ can be defined using the geometric definition introduced in Definition \ref{Slope-subalgebra-definition-derived
:label}, and it still admits the factorisation as stated in \ref{slope-factoristaion-integral}, and the proof can be found in \cite{PT25}. 

Since we have the factorisation property for the MO quantum loop group as stated in Proposition \ref{factorisation-for-wall-subalgebras:label}, we can define $U_{q}^{MO,+,\mbb{Z}}(\mf{g}_{\mbf{m}})$ as at the end of the Section \ref{ssub:_textbf_wall_subalgebra}, which is the algebra generated by $U_{q}^{MO,+,\mbb{Z}}(\mf{g}_{\mbf{m},w})$ with the walls $w$ containing $\mbf{m}$, and it is easy to verify as in the proof of the Proposition \ref{factorisation-for-wall-subalgebras:label} that we have the decomposition:
\begin{align*}
U_{q}^{MO,+,\mbb{Z}}(\mf{g}_{\mbf{m}})\cong\bigotimes_{\mbf{m}\in w}^{\rightarrow}U_{q}^{MO,+,\mbb{Z}}(\mf{g}_{\mbf{m},w})
\end{align*}
and this implies that $U_{q}^{MO,+,\mbb{Z}}(\mf{g}_{\mbf{m}})_{\mbf{v}}$ is a free $K_{A}(pt)$-module. Doing the similar computation as in the proof of Theorem \ref{isomorphism-of-integral-positive-half:theorem}, we obtain the following commutative diagram as in \ref{The-most-important-diagram}:
\begin{equation}\label{The-second-most-important-diagram}
\begin{tikzcd}
\mc{B}_{\mu\bm{\theta},\mbf{v}}^+\otimes K_{A\times G_{\mbf{v}}}(pt)\arrow[r]\arrow[d,hook]&U_{q}^{MO,+}(\mf{g}_{\mu\bm{\theta}})_{\mbf{v}}\otimes K_{A\times G_{\mbf{v}}}(pt)\arrow[d,hook]\\
\mc{A}^+_{Q,\mbf{v},[\eta\bm{\theta},(1+\eta)\bm{\theta})}\otimes K_{A\times G_{\mbf{v}}}(pt)\arrow[r,"\cong",shift left]&K_{A\times G_{\mbf{v}}}(\mc{M}_{Q}(\mbf{0},\mbf{v},\mbf{v}))\arrow[l,"\psi",shift left]
\end{tikzcd}
\end{equation}
and the isomorphism at the bottom is still the result of Theorem \ref{categorical-slope-factorisation:Theorem} by replacing $T$ with $A$. Here $\mu$ and $\eta$ are still of the same choice as in Section \ref{ssub:proof_of_theorem_ref_isomorphism_of_integral_positive_half_theorem}. But one should notice that the map on the left hand side of \ref{The-second-most-important-diagram} might not be injective, but using the same argument as in \ref{ssub:proof_of_theorem_ref_isomorphism_of_integral_positive_half_theorem}. One can obtain the map:
\begin{equation}
\begin{tikzcd}
\mc{B}_{\mu\bm{\theta},\mbf{v}}^{+,\mbb{Z}}\otimes K_{A\times G_{\mbf{v}}}(pt)\arrow[r,shift left,"f"]& U_{q}^{MO,+,\mbb{Z}}(\mf{g}_{\mu\bm{\theta}})_{\mbf{v}}\otimes K_{A\times G_{\mbf{v}}}(pt)\arrow[l,hook,shift left,"g"]
\end{tikzcd}
\end{equation}
such that $g\circ f=\text{Id}$, which implies that $f$ is an isomorphism of $K_{A\times G_{\mbf{v}}}(pt)$-module, and it implies the isomorphism of $K_{A}(pt)$-module:
\begin{align}
\mc{B}_{\mu\bm{\theta},\mbf{v}}^{+,\mbb{Z}}\cong U_{q}^{MO,+,\mbb{Z}}(\mf{g}_{\mu\bm{\theta}})_{\mbf{v}}.
\end{align}
Thus the proof is finished.
\end{proof}

THANK YOU FOR YOUR PATIENT READING, HAVE A GOOD REST AND HOPE YOU CAN STAY HEALTHY EVERYDAY :))

\end{document}